\documentclass[12pt,letterpaper]{amsart}

\newcounter{myenum}
\newcounter{myenumbak}
\newenvironment{myenumerate}[1]
{ \setcounter{myenumbak}{\value{myenum}}
  \setcounter{myenum}{0}
  \def\item{\refstepcounter{myenum} \smallskip \par\noindent{\bf #1 \arabic{myenum}.\ }}
}{ \setcounter{myenum}{\value{myenumbak}}}

\newcommand{\indentalign}{\hspace{0.3in}&\hspace{-0.3in}}
\newcommand{\la}{\langle}
\newcommand{\ra}{\rangle}
\renewcommand{\Re}{\operatorname{Re}}
\renewcommand{\Im}{\operatorname{Im}}
\newcommand{\ds}{\displaystyle}

\usepackage{amssymb}
\usepackage{amsthm}
\usepackage{amsxtra}

\newtheorem{theorem}{Theorem}

\newtheorem{lemma}[theorem]{Lemma}
\newtheorem{corollary}[theorem]{Corollary}

\newcommand{\cR}{\mathbb{R}}

\newcommand{\supp}{\mathrm{supp}\,}

\numberwithin{equation}{section}
\numberwithin{theorem}{section}

\title[Circle blow-up solutions to 3d NLS] {A class of solutions to the 3d cubic nonlinear Schr\"odinger equation that blow-up on a circle}
\author{Justin Holmer}
\address{Brown University}
\author{Svetlana Roudenko}
\address{Arizona State University}

\begin{document}

\begin{abstract}
We consider the 3d cubic focusing nonlinear Schr\"odinger equation (NLS)
$i\partial_t u + \Delta u + |u|^2 u=0$, which appears as a model in condensed matter theory and plasma physics. We construct a family of axially symmetric solutions, corresponding to an open set in $H^1_\textnormal{axial}(\mathbb{R}^3)$ of initial data, that blow-up in finite time with singular set a circle in $xy$ plane.
Our construction is modeled on Rapha\"el's construction \cite{R} of a family of solutions to the 2d quintic focusing NLS, $i\partial_t u + \Delta u + |u|^4 u=0$, that blow-up on a circle.
\end{abstract}

\maketitle

\section{Introduction}

Consider, in dimension $n\geq 1$, the $p$-power focusing nonlinear Schr\"odinger %($n$ d-NLS-$p$) equation
\begin{equation}
\label{E:NLS}
i\partial_t u + \Delta u + |u|^{p-1}u=0 \,.
\end{equation}
This equation obeys the scaling symmetry
$$
u(x,t) \text{ solves }\eqref{E:NLS} \implies \lambda^\frac{2}{p-1} u( \lambda x, \lambda^2 t)
\text{ solves }\eqref{E:NLS} \,,
$$
which implies that the homogeneous Sobolev norm $\dot H^s$ is scale invariant provided $s = \frac{n}{2}-\frac{2}{p-1}$.  The equation \eqref{E:NLS} also obeys mass, energy, and momentum conservation, which are respectively defined as
$$
M[u]=\|u\|_{L^2}^2 \,, \quad E[u] = \frac12 \|\nabla u\|_{L^2}^2 - \frac{1}{p+1} \|u\|_{L^{p+1}}^{p+1} \,, \qquad P[u] = \Im \int \nabla u \; \bar u \,dx \,.
$$
In the $\dot H^1$ subcritical setting ($1<p<1+\frac{4}{n-2}$), there exist soliton solutions $u(x,t)=e^{it}Q(x)$, where
\begin{equation}
\label{E:Q}
-Q + \Delta Q + |Q|^{p-1}Q=0 \,.
\end{equation}
We take $Q$ to be the unique, radial, positive smooth solution (in $\cR^n$) of this nonlinear elliptic equation of minimal mass. For further properties see, for example, \cite{HR1}-\cite{HR2} and references therein.

The local theory in $H^1$ ($p \leq 1+\frac4{n-2}$) is known from work of Ginibre-Velo \cite{GV}.
%(see also \cite{Caz-book} for review).
Local existence in time extends to the maximal interval $(-T_*, T^*)$, and if $T^*$ or $T_*$ are finite, it is said that the corresponding solution blows up in finite time. The existence of blow up solutions are known and the history goes back to work of Vlasov-Petrishev-Talanov '71 \cite{VPT}, Zakharov '72 \cite{Z72} and Glassey '77 \cite{G77} who showed that negative energy solutions with finite variance, $ \| x u_0\|_{L^2} < \infty$, blow up in finite time.
Ogawa-Tsutsumi \cite{OT91} extended this result to radial solutions by localizing the variance. Martel \cite{M97} showed that further relaxation to the nonisotropic finite variance or radiality only in some of the variables guarantees that negative energy solutions blow up in finite time. First result for nonradial infinite variance (negative energy) solutions
blowing up in finite time was by Glangetas - Merle \cite{GM} using a concentration-compactness
method (see also Nawa \cite{N} for a similar result).
%for the case $P[u] \neq 0$).
Positive energy blow up solutions are also known and go back to \cite{VPT}, \cite{Z72} (see  \cite[Theorem 5.1]{SS} for the precise statement). Turitsyn \cite{T93} and Kuznetsov et al \cite{KRRT} extended the blow up criteria for finite variance solutions of $\dot{H}^1$ subcritical NLS and, in particular, for the $3d$ cubic NLS showed that finite variance solutions blow up in finite time provided they are under the `mass-energy' threshold $M[u_0]E[u_0] < c(\|Q\|_{L^2})$, where $Q$ is the ground state solution of \eqref{E:Q}, and further assumption on the initial size of the gradient $\|\nabla u_0\|_{L^2} > c(M[u_0], \|Q\|_{L^2})$.  Independently, the same conditions for finite variance as well as for radial data (for all $\dot{H}^1$ subcritical, $L^2$ supercritical NLS) were obtained in \cite{HR1}. [Similar sufficient blow up conditions for finite variance or radial data for the $\dot{H}^1$ critical NLS ($s=1$ or $p = 1+ \frac4{n-2}$) are due to Kenig - Merle \cite{KM}; for the situation in the $L^2$ critical NLS ($s=0$ or $p=1+\frac4{n}$) with $H^1$ data refer to Weinstein \cite{W83} for the sharp threshold and Merle \cite{M97} for the characterization of minimal mass blow up solutions.]
The nonradial infinite variance solutions of the 3d cubic NLS blow up in finite or infinite time provided they are under the  `mass-energy' threshold having the same condition on the size of the gradient as discussed above. This was shown using a variance of the rigidity /concentration-compactness method in \cite{HR3}, thus, extending the result of Glangetas-Merle to positive energy solutions. Further extensions of sufficient conditions for finite time blow up were done by Lushnikov \cite{Lu95} and Holmer-Platte-Roudenko \cite{HPR} which include blow up solutions above the `mass-energy' threshold and are given via variance (or its localized version) and its first derivative if the data is not real. It is also possible to construct solutions which blow up in one time direction and scatter or approach the soliton solution (up to the symmetries of the equation) in the other time direction, see \cite{HPR}, \cite{DM}, \cite{DR}.

A detailed description of the dynamics of blow-up solutions in the $L^2$ critical ($p=1+\frac{4}{n}$) case has been developed by Merle-Raphael \cite{MR-Annals, MR-GAFA, MR-JAMS, MR-Invent, MR-CMP, R-stability}.  They show that blow-up solutions are, to leading order in $H^1$, described by the profile $Q_{b(t)}$, rescaled at rate $\lambda(t) \sim ((T-t)/\log|\log(T-t)|)^{1/2}$, where $b(t) \sim 1/\log|\log(T-t)|$.  Here, $Q_b$ is a slight modification of $Q$ -- see notational item N\ref{N:Qb} in \S\ref{S:notation} for the details of the definition in the 2d cubic case.

In the $L^2$ supercritical, $\dot H^1$ subcritical regime ($1+\frac{4}{n}<p<1+\frac{4}{n-2}$),
%one has local well-posedness in the energy space, and
one has a large family of blow-up solutions as discussed above,
 %(e.g., see Duyckaerts-Holmer-Roudenko \cite{DHR} for blow up solutions below the `mass-energy' threshold (also \cite{HR3} for infinite variance nonradial case), Duyckaerts-Roudenko \cite{DR} for blow up solutions at the `mass-energy' threshold and Holmer-Platte-Roudenko \cite{HPR} for blow up solutions above the `mass-energy' threshold),
but there are fewer results characterizing the dynamical behavior of blow-up solutions -- we are only aware of three: Rapha\"el \cite{R} (quintic NLS in 2d) and the extension by Rapha\"el-Szeftel \cite{RS} (quintic NLS in all dimensions) and Merle-Rapha\"el-Szeftel \cite{MRS} (slighly mass-supercritical NLS). In this paper\footnote{A result similar to the one presented in this paper was simultaneously developed by Zwiers \cite{Zw}, see remarks at the end of the introduction.}, we consider the 3d cubic equation ($n=3$, $p=3$) which is a physically relevant case in condensed matter and plasma physics.  We adapt the method introduced by Rapha\"el \cite{R} to give a construction of a family of finite-time blow-up solutions that blow-up on a circle in the $xy$ plane.  Rapha\"el constructed a family of finite-time blow-up solutions to the 2d quintic equation ($n=2$, $p=5$) which is $\dot H^{1/2}$ critical, that blow-up on a circle.  He accomplished this by introducing a radial symmetry assumption, and at the focal point of blow-up (without loss $r=1$), the equation effectively reduces to the 1d quintic NLS, which is $L^2$ critical and for which there is a well-developed theory characterizing the dynamics of blow-up.  We employ a similar dimensional reduction scheme -- starting from the 3d cubic NLS (which is $\dot H^{1/2}$ critical) we impose an axial symmetry assumption, and construct blow-up solutions with a focal point at $(r,z)=(1,0)$, where the equation in $(r,z)$ coordinates effectively reduces to the 2d cubic NLS equation, which is $L^2$ critical.

Our main result is

\begin{theorem}
\label{T:main}
Let $Q_b=Q_b(\tilde r,\tilde z)$ be as defined in N\ref{N:Qb} in \S\ref{S:notation}.  There exists an open set $\mathcal{P}$ in $H^1_{\textnormal{axial}}(\mathbb{R}^3)$, defined precisely in \S\ref{S:id}, such that the following holds true.  Let $u_0\in \mathcal{P}$.  Then the corresponding solution $u(t)$ to 3d cubic NLS (\eqref{E:NLS} with $n=3$, $p=3$) blows-up in finite time $0<T<\infty$ according to the following dynamics.
\begin{enumerate}
\item \emph{Description of the singularity formation}.
There exists $\lambda(t)>0$, $r(t)>0$, $z(t)$, and $\gamma(t)\in \mathbb{R}$ such that, if we define
$$
u_{\textnormal{core}}(r,z,t) = \frac{1}{\lambda(t)} Q_b\left( \frac{r-r(t)}{\lambda(t)}, \frac{z-z(t)}{\lambda(t)}\right)e^{i\gamma(t)} \,,
$$
$$
\tilde u(t) = u(t)-u_{\textnormal{core}}(t) \,,
$$
then
\begin{equation}
\label{E:L2profile}
\tilde u(t) \to u^* \text{ in } L^2(\mathbb{R}^3) \text{ as }t \to T \,,
\end{equation}
\begin{equation}
\label{E:H1error}
\| \nabla \tilde u(t)\|_{L^2(\mathbb{R}^3)} \lesssim \|\nabla u_{\textnormal{core}}(t)\|_{L^2(\mathbb{R}^3)} \cdot \frac{1}{|\log(T-t)|^c} \,,
\end{equation}
and the position of the singular circle converges:
\begin{align}
\label{E:r-conv}
& r(t) \to r(T)>0 \text{ as }t \to T , \\
\label{E:z-conv}
& z(t) \to z(T) \text{ as }t \to T\,.
\end{align}

\item \emph{Estimate on the blow-up speed}.  We have, as $t\nearrow T$,
\begin{align}
\label{E:lambda-rate}
& \lambda(t) \sim \left( \frac{T-t}{\log|\log(T-t)|} \right)^{1/2}
\,, \\
\label{E:b-rate}
&  b(t) \sim \frac{1}{\log |\log (T-t)|} \,,\\
\label{E:gamma-rate}
& \gamma(t) \sim |\log (T-t)| \log|\log(T-t)| .
\end{align}
\item \emph{Structure of the $L^2$ remainder $u^*$}.  For all $R>0$ small enough,
\begin{equation}
\label{E:prof-asymp1}
\int_{|r-r(T)|^2+|z-z(T)|^2\leq R^2} |u^*(r)|^2 \, dr\, dz \sim \frac{1}{(\log|\log R|)^2} \,,
\end{equation}
and, in particular, $u^*\notin L^p \text{ for }p>2$.
\item \emph{$H^{1/2}$ gain of regularity outside the singular circle}.  For any $R>0$,
\begin{equation}
\label{E:H.5gain}
u^*\in H^{1/2}(|r-r(T)|^2+|z-z(T)|^2>R^2) \,.
\end{equation}
\end{enumerate}
\end{theorem}

A key ingredient in exploiting the cylindrical geometry away from the $z$-axis is the
\emph{axially symmetric Gagliardo-Nirenberg} inequality, which we prove in \S\ref{S:cGN}.
This takes the role of the radial Gagliardo-Nirenberg inequality of Strauss \cite{S} employed
by Rapha\"el.  In \S\ref{S:notation}, we collect most of the notation employed, and
in \S\ref{S:outline}, we outline the structure of the proof of Theorem \ref{T:main}.
Most of the argument is a lengthy bootstrap, and in \S\ref{S:outline}, we enumerate the bootstrap input statements (BSI 1--8) and the corresponding bootstrap output statements (BSO 1--8). As the output statements are stronger than the input statements, we conclude that all the BSO assertions hold for the full time interval of existence.  The steps involved in deducing BSO 1--8 under the assumptions BSI 1--8 are outlined in \S\ref{S:outline}, and carried out in detail in \S\ref{S:bs2}--\ref{S:bs16}. \S\ref{S:bs16} already proves \eqref{E:H.5gain} in
Theorem \ref{T:main}, the $H^{1/2}$ gain of regularity outside the singular circle.
The proof of Theorem \ref{T:main} is completed in \S\ref{S:loglog}--\ref{S:rem-conv}.
%\ref{S:rem-desc}.
In \S\ref{S:loglog}, we prove the log-log rate of blow-up \eqref{E:lambda-rate}.
In \S\ref{S:position}, we prove the convergence of the position of the singular circle,
\eqref{E:r-conv} and \eqref{E:z-conv}. The proof of the size estimates on the remainder profile,
\eqref{E:prof-asymp1} is the same as for Theorem 3 in \cite{MR-CMP} (we refer the reader to Sections 3, 4, 5 there).
Finally, in \S\ref{S:rem-conv}, we prove the convergence of the remainder in $L^2$,
the estimate \eqref{E:L2profile}.
%and finally in \S\ref{S:rem-desc}, we prove the size estimates on this remainder profile, \eqref{E:prof-asymp1}.

A similar result to Theorem \ref{T:main} but for a slightly smaller class of initial data was recently obtained by Zwiers \cite{Zw} (our $H^1_{\textnormal{axial}}$ class of initial data is replaced by a smoother version $H^3_{\textnormal{axial}}$). We mention that the methods to prove our main theorem do not treat the cubic equation in higher dimensions, but it is addressed by Zwiers in \cite[Theorem 1.3]{Zw}.
We think that it should be possible to adapt our method to treat the $\dot{H}^{1/2}$
critical ($p=1+4/(n-1)$) case in higher dimensions, proving the existence of
solutions blowing up on a circle (dimension $1$ blow-up set), however, dealing with fractional nonlinearities is a delicate matter.
Zwiers' result treats blow up sets of codimension 2 in all dimensions;
our approach might be able to treat the blow up sets of dimension 1 in all dimensions.
(In 3 dimensions, of course, codimension 2 equals dimension 1, and thus, the results intersect.)

\subsection{Acknowledgements}
J.H. is partially supported by a Sloan fellowship and NSF grant DMS-0901582.
S.R. is partially supported by NSF grant DMS-0808081. This project has started at the MSRI program ``Nonlinear Dispersive Equations" in Fall 2005 and both authors are grateful for the support and stimulating workshops during that program.  We thank Ian Zwiers and Fabrice Planchon for remarks on this paper.

\section{A Gagliardo-Nirenberg inequality for axially symmetric functions}
\label{S:cGN}

We begin with an axially symmetric Gagliardo-Nirenberg estimate, analogous to the radial Gagliardo-Nirenberg estimate of Strauss \cite{S}.  Consider a function $f=f(x,y,z) = f(r\cos \theta, r\sin \theta, z) = f(r,z)$, independent of $\theta$.  We call such a function \textit{axially symmetric}.

\begin{lemma}
\label{L:cyl-Strauss}
Suppose that $f$ is axially symmetric.  Then for each $\epsilon>0$,
\begin{equation}
\label{E:101}
\|f\|_{L^4(r>\epsilon)}^4 \leq \frac{1}{\epsilon}\|f\|_{L_{xyz}^2(r>\epsilon)}^2\|\nabla f\|_{L_{xyz}^2(r>\epsilon)}^2 .
\end{equation}
\end{lemma}

\begin{proof}
The proof is modeled on the classical proof of the Sobolev estimates.
We use the notation $\nabla=(\partial_x, \partial_y, \partial_z)$
(i.e., not $(\partial_r, \partial_z)$). Note that $\|f(x,y,z)\|_{L_{xyz}^2}=\|f(r,z)\|_{L_{rdrdz}^2}$, and also note $\partial_r f(r,z) = (\partial_xf)(r,z)\cos \theta + (\partial_yf)(r,z)\sin\theta$, and thus, $|\partial_rf| \leq |\nabla f|$.
Observe that for a fixed $r>\epsilon$, $z\in \mathbb{R}$, by the fundamental theorem of calculus and the Cauchy-Schwarz inequality,
\begin{align*}
|f(r,z)|^2 &= |f(+\infty,z)|^2 - |f(r,z)|^2 \\
&\leq 2\int_{r'=\epsilon}^{+\infty} |f(r',z)| |\partial_r f(r',z)| \, dr \\
&\leq \frac{2}{\epsilon}\int_{r'=\epsilon}^{+\infty} |f(r',z)| |\partial_r f(r',z)| \, r'dr' \\
&\leq \frac{2}{\epsilon}\|f(r',z)\|_{L_{r'dr'}^2(r'>\epsilon)} \|\nabla f(r',z)\|_{L_{r'dr'}^2(r'>\epsilon)} \,,
\end{align*}
and also,
$$|f(r,z)|^2 \leq 2 \int_{z'=-\infty}^{+\infty} |f(r,z')| |\partial_z f(r,z')|\, dz'\leq \|f(r,z')\|_{L_{z'}^2}\|\nabla f(r,z')\|_{L_{z'}^2} \,.$$
By multiplying the above two inequalities, and then integrating against $rdrdz$, we get
\begin{align*}
\int_{r>\epsilon,z}|f(r,z)|^4r \, dr\, dz \leq \frac{4}{\epsilon}\left( \int_z \|f(r',z)\|_{L_{r'dr'}^2(r'>\epsilon)} \|\nabla f(r',z)\|_{L_{r'dr'}^2(r'>\epsilon)} \, dz \right) & \\
\times \left( \int_{r=\epsilon}^{+\infty} \|f(r,z')\|_{L_{z'}^2}\|\nabla f(r,z')\|_{L_{z'}^2} r\, dr\right) & \,.
\end{align*}
Following through with Cauchy-Schwarz in each of the two integrals gives \eqref{E:101}.
\end{proof}

As a corollary of \eqref{E:101}, we have for $1\leq p\leq 3$, that
\begin{equation}
\label{E:interpolation}
\|f\|_{L_{xyz}^{p+1}(r>\epsilon)}^{p+1} \leq \frac{2^{p-1}}{\epsilon^{\frac{p-1}{2}}} \|f\|_{L_{xyz}^2(r>\epsilon)}^2 \|\nabla f\|_{L_{xyz}^2(r>\epsilon)}^{p-1} \,.
\end{equation}
This follows by the interpolation estimate (H\"older's inequality)
$$
\|f\|_{L^{p+1}}^{p+1} \leq \|f\|_{L^2}^{3-p}\|f\|_{L^4}^{2p-2} \,.
$$

Before proceeding, we present a simple application of Lemma \ref{L:cyl-Strauss}.

\begin{corollary} \label{C:axial}
If $u$ is a cylindrically symmetric solution to
$$
i\partial_t u + \Delta u + |u|^{p-1}u =0
$$
for $p<3$ in $\mathbb{R}^3$ that blows-up at finite time $T>0$, then blow-up must occur along the $z$-axis.  Specifically for any fixed $\epsilon>0$,
\begin{equation}
\label{E:blow-up-near-z}
\lim_{t\uparrow T} \|\nabla u\|_{L_{xyz}^2(r<\epsilon)} = +\infty \,.
\end{equation}
\end{corollary}

\begin{proof}
Fix any $\epsilon>0$.  All $L^p$ norms will be with respect to $dxdydz$.
For any $t>0$, by \eqref{E:interpolation}, we have for $u=u(t)$
\begin{align*}
\frac12 \|\nabla u\|_{L^2}^2 &= E + \frac{1}{p+1} \|u\|_{L^{p+1}(r<\epsilon)}^{p+1} + \frac1{p+1} \|u\|_{L^{p+1}(r>\epsilon)}^{p+1} \\
&\leq E + \frac{1}{p+1} \|u\|_{L^{p+1}(r<\epsilon)}^{p+1}
+ \frac{C}{\epsilon^\frac{p-1}{2}}\|u\|_{L^2}^2 \|\nabla u\|_{L^2}^{p-1} .
\end{align*}
Using the %Peter-Paul
inequality $\alpha \beta \leq \frac{3-p}{2} \alpha^\frac{2}{3-p} + \frac{p-1}{2} \beta^\frac{2}{p-1}$, we obtain
$$
\| \nabla u \|_{L^2}^2 \leq 4E + \frac{4}{p+1} \|u\|_{L^{p+1}(r<\epsilon)}^{p+1} + \frac{C}{\epsilon^\frac{p-1}{3-p}} \|u\|_{L^2}^\frac{4}{3-p} \,.
$$
Since $\lim_{t\nearrow +\infty} \|\nabla u(t)\|_{L^2} = +\infty$, we obtain that $\lim_{t\nearrow +\infty} \|u(t)\|_{L^{p+1}(r>\epsilon)} = +\infty$.  By the (standard) Gagliardo-Nirenberg inequality, we obtain \eqref{E:blow-up-near-z}.
\end{proof}

\section{Notation}
\label{S:notation}

Recall that we will impose the \emph{axial symmetry} assumption, i.e.,
$$
\tilde u(r,\theta, z) = u(r\cos\theta, r\sin \theta, z)
$$
is assumed independent of $\theta$.
The equation \eqref{E:NLS} in cylindrical coordinates $(r,z)$, assuming the axial symmetry, is
\begin{equation}
\label{E:NLSaxial}
i \partial_t u +\frac1{r} \partial_r u + \partial^2_r u + \partial^2_z u + |u|^2\,u = 0.
\end{equation}

Denote by $Q=Q(r)$, $r = |x|$, $x \in \cR^2$, a ground state solution to the 2d nonlinear elliptic equation (which corresponds to the mass-critical cubic NLS equation in 2d):
$$
-Q+\Delta_{\cR^2} Q+|Q|^2 \, Q = 0.
$$
We emphasize that $Q$ is a two-dimensional object.

We now enumerate our notational conventions.

\begin{myenumerate}{N}
\item
We will adopt the convention that
$\nabla = (\partial_x, \partial_y, \partial_z)$ is the full gradient and
$\Delta = \partial_x^2 + \partial_y^2 + \partial_z^2$ is the full
Laplacian. When viewing an axially symmetry function as a function of $r$ and $z$, we will write the corresponding operators as  $\nabla_{(r,z)} =
(\partial_r, \partial_z)$ and $\Delta_{(r,z)} = \partial_r^2 +
\partial_z^2$.  Note that under the axial symmetry
assumption, $\Delta = \Delta_{(r,z)} + r^{-1}\partial_r$.

\item
In the argument, parameters $\lambda(t)$, $\gamma(t)$, $r(t)$ and
$z(t)$ emerge.  Rescaled time, for given $b_0$, is
\begin{equation}
\label{E:s0}
s(t)= \int_0^t \frac{dt'}{\lambda^2(t')} + s_0, \quad s_0 = e^{\frac{3\pi}{4b_0}} \,,
\end{equation}
and rescaled position is
$$
\tilde r(t) = \frac{r-r(t)}{\lambda(t)}, \qquad \tilde z(t)
= \frac{z-z(t)}{\lambda(t)}.
$$
We note that the $\tilde r$ and $\tilde z$ can both be negative, although
there is the restriction that $\lambda(t)\tilde r + r(t) \geq 0$.
Introduce the full (rescaled) radial variable
$$
\tilde R = \sqrt{\tilde r^2 + \tilde z^2}.
$$
We have $r=\lambda(t)\tilde r +r(t)$. To help avoid confusion between $r$ and $\tilde r$,
we will use the following notation:  given two parameters $\lambda(t)$ and $r(t)$,
define
$$
\mu_{\lambda(t),r(t)}(\tilde r) = (\lambda(t) \tilde r +
r(t))\textbf{1}_{\{\lambda(t) \tilde r + r(t) \geq 0\}}.
$$
Often the $\lambda(t),r(t)$ subscript will be dropped.

\item
The inner product $(\cdot, \cdot)$ will mean 2d real inner product in the
$\tilde r$, $\tilde z$ variables.

\item
If $f=f(\tilde r, \tilde z)$, then define
$\Lambda f = f + (\tilde r, \tilde z)\cdot (\partial_{\tilde r},
\partial_{\tilde z})f$, the \emph{generator for scaling}.
Observe that $\Lambda f = f + \tilde R
\partial_{\tilde R} f$.

\item
\label{N:Qb}
For a parameter $b$, $Q_b(\tilde r,\tilde z)$ is the \emph{2d
localized self-similar profile}.  Specifically, following Prop. 8 in  \cite{MR-Invent} which is a refinement of Prop. 1 in \cite{MR-GAFA}, given $b$, $\eta>0$, define
$$\tilde R_b = \frac{2}{|b|} \sqrt{1-\eta} \,.$$
There exist universal constants $C>0$, $\eta^*>0$ such that the following holds true.  For all $0<\eta<\eta^*$, there exists constants $\epsilon^*(\eta)>0$, $b^*(\eta)>0$ going to zero as $\eta\to 0$ such that for all $|b|<b^*(\eta)$, there exists a unique {\it radial} solution $Q_b$ (i.e., $Q_b$ depends only on $\tilde R$) to
$$
\left\{\begin{aligned}
& \Delta_{(\tilde r, \tilde z)} Q_b - Q_b +ib\Lambda Q_b + Q_b|Q_b|^2=0 \\
&Q_b(\tilde R)e^\frac{ib\tilde R^2}{4}>0 \text{ for } \tilde R \in [0,\tilde R_b) \\
&Q_b(\tilde R_b)=0 \\
& Q(0)-\epsilon^*(\eta) < Q_b(0)< Q(0)+\epsilon^*(\eta) .
\end{aligned}
\right.
$$

\item
For a parameter $b$, $\tilde Q_b(\tilde r,\tilde z)=\tilde Q_b(\tilde R)$ is a \emph{truncation} of the 2d localized self-similar profile $Q_b$.  Specifically, (following Prop. 8 in \cite{MR-Invent}), given $b>0$, $\eta>0$ small, let $\tilde R_b^- = \sqrt{1-\eta}\tilde R_b$ so that $\tilde R_b^- < \tilde R_b$.  Let $\phi_b$ be a radial smooth cut-off function such that $\phi_b(x)=0$ for $|x|\geq \tilde R_b$ and $\phi_b(x)=1$ for $|x|\leq \tilde R_b^-$, and everywhere $0\leq \phi(x) \leq 1$, such that
$$
\| \phi_b'\|_{L^\infty} + \| \Delta \phi_b \|_{L^\infty} \to 0 \text{ as }|b|\to 0\,.
$$
Now set
$$
\tilde Q_b(\tilde R) = Q_b(\tilde R) \phi_b(\tilde R) \,.
$$
Then
\begin{equation}
\label{E:Q_b}
\Delta_{(\tilde r, \tilde z)} \tilde Q_b - \tilde Q_b +ib\Lambda \tilde Q_b + \tilde Q_b|\tilde Q_b|^2=-\Psi_b,
\end{equation}
where
\begin{equation}
\label{E:Psi_b}
-\Psi_b = \Delta \varphi_b \, Q_b + 2 \nabla Q_b \cdot \nabla \varphi_b + i b Q_b \, (\tilde r, \tilde z)\cdot \nabla \varphi_b + (\varphi_b^3 - \varphi_b) Q_b |Q_b|^2
\end{equation}
with the property that for any polynomial $P(y)$ and $k = 0, 1$
\begin{equation}
\label{E:Psi1}
\|P(y) \, \Psi_b^{(k)}\|_{L^\infty} \leq e^{-C_P/|b|}.
\end{equation}
In terms of $\Psi_b(\tilde r, \tilde z)$, we define an adjusted
$\tilde \Psi_b(t,\tilde r, \tilde z)$ as
\begin{equation}
\label{E:Psitilde} \tilde \Psi_b(t, \tilde r, \tilde z) =
\Psi_b(t,\tilde r, \tilde z) -
\frac{\lambda(t)}{\mu_{\lambda(t),r(t)}(\tilde r)}
\partial_{\tilde r} \tilde Q_b(\tilde r, \tilde z)
\end{equation}
so that $\tilde Q_b$ solves
\begin{equation}
\label{E:Qtilde} \Delta_{(\tilde r, \tilde z)} \tilde Q_b
+\frac{\lambda}{\mu}\partial_{\tilde r}\tilde Q_b- \tilde Q_b
+ib\Lambda \tilde Q_b + \tilde Q_b|\tilde Q_b|^2=-\tilde \Psi_b.
\end{equation}
We also split $\tilde Q_b$ into real and imaginary parts as
$$
\tilde Q_b = \Sigma + i\Theta.
$$
It is implicitly understood that $\Sigma$ and $\Theta$ depend on $b$ (or $b(t)$),
and when we want to emphasize dependence, it will be stated explicitly;
this decomposition is done only for the truncated profile $\tilde Q_b$ (not for $Q_b$).
Similarly, we denote $\ds \Psi_b = \Re \Psi + i \Im \Psi$ and $\ds \tilde \Psi_b = \Re \tilde \Psi + i \Im \tilde \Psi$.

\item
The $\tilde Q_b$ satisfies the following properties.
\begin{myenumerate}{QP}
\item ((44) in \cite{MR-GAFA})
Uniform closeness to the ground state.  For a fixed universal constant $C>0$,
$$
\|e^{C\tilde R}( \tilde Q_b - Q) \|_{C^3} \to 0 \text{ as }b \to 0 \,.
$$
In particular, this implies that $\|e^{c \tilde R} \partial^k_{\tilde r} \Theta \|_{L^2(\tilde r, \tilde z)} \to 0$ as $b \to 0$ for $0 \leq k \leq 3$.

\item ((45) in \cite{MR-GAFA})
Uniform closeness of the derivative $\partial_b Q_b$ to the ground state.  For a fixed universal constant $C>0$,
$$
\left\| e^{C\tilde R}(\partial_b\tilde Q_b +
i\tfrac14\tilde R^2 Q )\right\|_{C^2} \to 0 \text{ as }b \to 0 \,.
$$

\item (Prop. 2(ii) in \cite{MR-GAFA} and Prop. 1(iii) in \cite{MR-JAMS})
Degeneracy of the energy and momentum. Specifically,
\begin{equation}
\label{E:energy-degeneracy}
|E(\tilde Q_b)| \leq e^{-(1+C\eta)\pi/|b|}, \quad \text{since} \quad 2E(\tilde Q_b) = -\Re \int \Lambda \Psi_b \, \bar{\tilde{Q}}_b,
\end{equation}
and
\begin{equation}
\label{E:momentum-degeneracy}
\Im \int \nabla_{(\tilde r, \tilde z)} \tilde Q_b \; \bar{\tilde{Q}}_b = 0, \qquad
\Im \int (\tilde r, \tilde z) \cdot \nabla_{(\tilde r, \tilde z)} \tilde Q_b \; \bar{\tilde{Q}}_b = - \frac{b}2 \|\tilde R \, \tilde Q_b\|^2_{L^2}.
\end{equation}

\item
The profile $\tilde Q_b$ has supercritical mass, and more precisely
$$
0< \frac{d}{d(b^2)}\Big|_{b^2=0}  \int |\tilde Q_b|^2  = d_0 < +\infty.
$$

\item
Algebraic relations corresponding to Galilean, conformal and scaling
invariances:
\begin{equation}
 \label{E:eqyQb}
\Delta (\tilde R \tilde Q_b) - \tilde R \tilde Q_b + i b \tilde R
\Lambda \tilde Q_b + \tilde R \tilde Q_b |\tilde Q_b|^2 = 2
\partial_{\tilde R} \tilde Q_b - \tilde R \Psi_b
\end{equation}
\begin{equation}
 \label{E:eqy^2Qb}
\Delta (\tilde R^2 \tilde Q_b) - \tilde R^2 \tilde Q_b + i b \tilde
R^2 \Lambda \tilde Q_b + \tilde R^2 \tilde Q_b |\tilde Q_b|^2 = 4
\Lambda \tilde Q_b - \tilde R^2 \Psi_b
\end{equation}
\begin{equation}
 \label{E:eqLambdaQb}
\Delta (\Lambda \tilde Q_b) - \Lambda \tilde Q_b + (\Lambda \tilde
Q_b) |\tilde Q_b|^2 + 2 \tilde Q_b (\Sigma (\Lambda \Sigma) + \Theta
(\Lambda \Theta))
\end{equation}
$$
= 2(\tilde Q_b - i b \Lambda \tilde Q_b - \Psi_b) - \Lambda \Psi_b -
i b \Lambda^2 \tilde Q_b.
$$

The proof of the above identities are similar to Prop. 2(iii) in \cite{MR-GAFA} adapted to the 2d case. For example, to obtain the third equation from scaling invariance,
multiply \eqref{E:Qtilde} by $\lambda^3$ and take argument to be $(\lambda \tilde r, \lambda \tilde z)$, then differentiate with respect to $\lambda$ and evaluate the derivative at $\lambda = 1$. Note that $\frac{d}{d\lambda}|_{\lambda=1}
\left(\lambda^3 \Psi_b(\lambda \tilde r, \lambda \tilde z) \right) = 3 \Psi_b +
(\tilde r, \tilde z) \cdot \nabla \Psi_b \equiv 2 \Psi_b + \Lambda \Psi_b$, and thus,
we obtain the claimed equation.
\end{myenumerate}

\item
The linear operator close to $\tilde Q_b$ in \eqref{E:Qtilde} is $M=(M_+,M_-)$, where ($\epsilon=\epsilon_1+i\epsilon_2$)
$$
M_+(\epsilon) = -\Delta_{(\tilde r, \tilde z)}\epsilon_1 - \frac{\lambda}{\mu}\partial_{\tilde r} \epsilon_1 + \epsilon_1 - \left( \frac{2\Sigma^2}{|\tilde Q_b|^2}+1\right) |\tilde Q_b|^2 \epsilon_1 - 2\Sigma \Theta \epsilon_2,
$$
$$
M_-(\epsilon) = -\Delta_{(\tilde r, \tilde z)} \epsilon_2 - \frac{\lambda}{\mu}\partial_{\tilde r} \epsilon_2 + \epsilon_2 - \left( \frac{2\Theta^2}{|\tilde Q_b|^2}+1\right) |\tilde Q_b|^2 \epsilon_2 - 2\Sigma \Theta \epsilon_1 .
$$

\item
For a given parameter $b$, the function $\zeta_b(\tilde r,\tilde z)$ is the \emph{2d linear outgoing radiation}.  Specifically, following Lemma 15 in \cite{MR-Invent}, there exist universal constants $C>0$ and $\eta^*>0$ such that for all $ 0<\eta< \eta^*$, there exists $b^*(\eta)>0$ such that for all $0<b<b^*(\eta)$, the following holds true: There exists a unique radial solution $\zeta_b$ to
$$
\left\{
\begin{aligned}
&(\partial_{\tilde r}^2+\partial_{\tilde z}^2) \zeta_b - \zeta_b +ib\Lambda \zeta_b = \Psi_b \\
& \int | (\partial_{\tilde r}, \partial_{\tilde z})\zeta_b|^2 \, d\tilde r d\tilde z < +\infty,
\end{aligned}
\right.
$$
where $\Psi_b$ is the error in the $\tilde Q_b$ equation above.  The number $\Gamma_b$ is defined as the \emph{radiative asymptotic}, i.e.,
$$
\Gamma_b = \lim_{\tilde R \to +\infty} |\tilde R|^2 |\zeta_b(\tilde R)|^2 .
$$
The $\zeta_b(\tilde r,\tilde z)$ and $\Gamma_b$  have the following properties:
\begin{myenumerate}{ZP}
\item
Control (and hence, smallness by ZP2) of $\zeta_b$ in $\dot H^1$
$$
\int |\nabla_{(\tilde r, \tilde z)} \zeta_b|^2 \leq \Gamma_b^{1-C\eta}.
$$
\item
Smallness of the radiative asymptotic
$$
\forall \, |\tilde R|> R_b, \qquad e^{-(1+C\eta)\pi/b} \leq \frac{4}{5} \Gamma_b \leq |\tilde R|^2 |\zeta_b(\tilde R)|^2 \leq e^{-(1-C\eta)\pi/b}.
$$
\end{myenumerate}

\item
We will make the \emph{spectral assumption} made in \cite{MR-Annals, MR-GAFA, MR-Invent, MR-JAMS, MR-CMP}.  We note that it involves $Q$ and not $\tilde Q_{b}$. A numerically assisted proof\footnote{The spectral property was proved in 1d and in dimensions 2-4 has a numerically assisted proof.} is given in Fibich-Merle-Rapha\"el \cite{FMR}.
Let (see {\bf N} 4)
%S: I changed the notation from Q_1 to \Lambda Q below, similar for Q_2:
$$
\Lambda Q = Q + (\tilde r, \tilde z) \cdot \nabla Q,
%Q_1 = Q + (\tilde r, \tilde z) \cdot \nabla Q$,
\qquad \Lambda^2 Q = \Lambda Q + (\tilde r, \tilde z) \cdot \nabla (\Lambda Q).
$$
Recall that $(\cdot , \cdot)$ denotes the 2d inner product in $(\tilde r, \tilde z)$. Consider the two Schr\"odinger operators
$$
\mathcal{L}_1 = -\Delta + 3 Q[(\tilde r, \tilde z) \cdot \nabla Q],
$$
$$
\mathcal{L}_2 = -\Delta +  Q [(\tilde r, \tilde z) \cdot \nabla Q],
$$
and the real-valued quadratic form for $\epsilon=\epsilon_1 + i \epsilon_2 \in H^1$:
$$
H(\epsilon, \epsilon) = ( \mathcal{L}_1 \epsilon_1, \epsilon_1) + (\mathcal{L}_2 \epsilon_2, \epsilon_2) \,.
$$
Then there exists a universal constant $\tilde \delta_1>0$ such that $\forall \, \epsilon \in H^1$, if the following orthogonality conditions hold:
$$
(\epsilon_1, Q)= 0, \quad (\epsilon_1, \Lambda Q) = 0, \quad (\epsilon_1, yQ) =0 ,
$$
$$
(\epsilon_2, \Lambda Q) = 0, \quad (\epsilon_2, \Lambda^2 Q) =0, \quad (\epsilon_2, \nabla Q) =0 ,
$$
then
$$
H(\epsilon, \epsilon) \geq \tilde\delta_1  \left( \int |\nabla \epsilon|^2 + \int |\epsilon|^2 e^{-|y|} \right) \,.
$$

\item The \emph{ring cutoffs} are the following.  The \emph{tight external cutoff} is
$$
\chi_0(r) =
\begin{cases}
1 & \text{for } 0 \leq r \leq \frac{13}{14} \text{ and for }r\geq \frac{14}{13} \\
0 & \text{for } \frac{15}{16} \leq r \leq \frac{16}{15}.
\end{cases}
$$
The \emph{wide external cutoff} is
$$
\chi_1(r) =
\begin{cases}
1 & \text{for } 0 \leq r \leq \frac{1}{4} \text{ and for }r\geq 4 \\
0 & \text{for } \frac{1}{2} \leq r \leq 2.
\end{cases}
$$
The \emph{internal cutoff} is
$$
\psi(r) =
\begin{cases}
0 & \text{for } 0 \leq r \leq \frac{1}{4} \text{ and for }r\geq 3 \\
r & \text{for } \frac{1}{2} \leq r \leq 2.
\end{cases}
$$
\end{myenumerate}

\section{Description of the set $\mathcal{P}$ of initial data}
\label{S:id}

We now write down the assumptions on the initial data set $\mathcal P$.  Let $\mathcal{P}$ be the set of axially symmetric $u_0 \in H^1(\mathbb{R}^3)$ of the form
$$
u_0(r,z) = \frac{1}{\lambda_0} \tilde Q_{b_0} \left( \frac{r-r_0}{\lambda_0}, \frac{z-z_0}{\lambda_0} \right) e^{i\gamma_0} + \tilde u_0(r,z) \,,
$$
and define the rescaled error
$$
\epsilon_0(\tilde r, \tilde z) = \lambda_0\tilde u_0 ( \lambda_0(\tilde r, \tilde z) + (r_0,z_0))e^{-i\gamma_0} \,,
$$
with the following controls:

\begin{myenumerate}{IDA}
\item
Localization of the singular circle:
$$
|r_0 - 1|< \alpha^*  \,, \quad |z_0| < \alpha^* \,.
$$

\item
Smallness of $b_0$, or closeness of $Q_{b_0}$ to $Q$ on the singular circle:
$$
0 < b_0 < \alpha^* \,.
$$

\item
Orthogonality conditions on $\epsilon_0$:
$$
(\Re\epsilon_0, |\tilde R|^2 \Sigma_{b_0}) + (\Im\epsilon_0, |\tilde R|^2 \Theta_{b_0})=0\,,
$$
$$
(\Re\epsilon_0, (\tilde r, \tilde z) \Sigma_{b_0}) + (\Im\epsilon_0, (\tilde r, \tilde z)\Theta_{b_0}) =0 \,,
$$
$$
-(\Re \epsilon_0, \Lambda^2 \Theta_{b_0}) + (\Im \epsilon_0, \Lambda^2 \Sigma_{b_0}) =0\,,
$$
$$
-(\Re \epsilon_0, \Lambda \Theta_{b_0}) + (\Im \epsilon_0, \Lambda \Sigma_{b_0}) =0 \,.
$$

\item
Smallness condition on $\epsilon_0$:
$$
\mathcal{E}(0)\equiv \int |\nabla_{(\tilde r, \tilde z)} \epsilon_0|^2 \mu_{\lambda_0,r_0}(\tilde r) \, d\tilde r d\tilde z + \int_{|(\tilde r, \tilde z)|\leq 10/b_0} |\epsilon_0(\tilde r, \tilde z)|^2 e^{-|\tilde R|} \, d\tilde r  d\tilde z \leq \Gamma_{b_0}^\frac{6}{7} \,.$$

\item
Normalization of the energy and localized momentum (recall $\psi$ from {\bf N 11}):
$$
\lambda_0^2 |E_0| + \lambda_0 \left| \Im \int \nabla \psi \cdot \nabla u_0 \; \bar u_0 \right| \leq \Gamma_{b_0}^{10} \,.
$$

\item
Log-log regime:
$$
0< \lambda_0 < \exp \left( - \exp \left( \frac{8\pi}{9b_0} \right) \right) \,.
$$

\item
Global $L^2$ smallness:
$$
\|\tilde u_0 \|_{L^2} \leq \alpha^* \,.
$$

\item
$H^{1/2}$ smallness outside the singular circle:
$$
\| \chi_0 \tilde u_0 \|_{H^{1/2}} \leq \alpha^* \,.
$$

\item
Closeness to the 2d ground state mass:
$$
\|u_0\|_{L^2(dxdydz)}^2 \leq \|Q\|_{L^2(d\tilde r d\tilde z)}^2 + \alpha^* \,.
$$

\item
Negative energy:
$$
E(u_0)<0 \,.
$$

\item
Axial symmetry of $u_0$.
\end{myenumerate}

\begin{lemma}
The set $\mathcal{P}$ is nonempty.
\end{lemma}
\begin{proof}
This follows as in Remark 3 of Rapha\"el \cite{R}.
\end{proof}

\section{Outline}
\label{S:outline}

Now let $u(t)$ be the solution to NLS with initial data from the
above set $\mathcal{P}$, and let $T>0$ be its maximal time of
existence (which at this point could be $+\infty$).  Because $u_0$ is
axially symmetric and the Laplacian is rotationally invariant in the
$xy$ plane, the solution $u(t)$ will be axially symmetric.  Thus, we
occasionally write $u(r,z,t)$.  The first step is to obtain a
``geometrical description'' of the solution.  Since we are not truly in the $L^2$ critical setting and cannot appeal to the variational characterization of $Q$, we need to incorporate this geometrical description into the bootstrap argument.  The lemma that we need, which follows from the implicit function theorem, is:

\begin{lemma}[cf. Merle-Rapha\"el \cite{MR-GAFA} Lemma 2, \cite{MR-Annals} Lemma 1]
\label{L:geomdecomp}
If for $0\leq t \leq t_1$, there exist parameters $(\bar \lambda(t), \bar \gamma(t), \bar r(t), \bar z(t), \bar b(t))$ such that $\|\bar \epsilon(t) \|_{H^1} \leq \alpha_* \ll 1$ on $0\leq t \leq t_1$,
where
$$
\bar \epsilon(\tilde r, \tilde z) = e^{i\bar \gamma} \bar \lambda u( \bar \lambda\tilde r+\bar r, \bar \lambda\tilde z +\bar z) - \tilde Q_{\bar b}(\tilde r, \tilde z),
$$
then there exist modified parameters $(\lambda(t), \gamma(t), r(t),z(t), b(t))$ such that $\epsilon$ defined by
$$
\epsilon(\tilde r, \tilde z)= e^{i\gamma} \lambda u( \lambda\tilde r+r, \lambda\tilde z +z) - \tilde Q_{b}(\tilde r, \tilde z)
$$
satisfies the following orthogonality conditions (with $\epsilon=\epsilon_1+i\epsilon_2$):
\begin{myenumerate}{ORTH}
\item
$(\epsilon_1, |\tilde R|^2 \Sigma_{b(t)}) + (\epsilon_2, |\tilde R|^2\Theta_{b(t)}) =0$
\item
$(\epsilon_1, (\tilde r, \tilde z) \Sigma_{b(t)}) + (\epsilon_2, (\tilde r, \tilde z)\Theta_{b(t)})=0$
\item
$-(\epsilon_1, \Lambda^2 \Theta_{b(t)}) + (\epsilon_2, \Lambda^2 \Sigma_{b(t)}) =0$
\item
$-(\epsilon_1, \Lambda \Theta_{b(t)}) + (\epsilon_2, \Lambda \Sigma_{b(t)}) =0$
\end{myenumerate}
and
$$
\left| 1- \lambda(t) \frac{\|\nabla u(t)\|_{L^2}}{\|\nabla Q\|_{L^2}} \right|
+ \|\epsilon(t)\|_{H^1} + |b(t)| \leq \delta(\alpha_0),
$$
where $\delta(\alpha_0)\to 0$ as $\alpha_0 \to 0$.
\end{lemma}
{\noindent}Note that the condition {\bf ORTH 2} is a vector equation.

We also use the notation
$$
\tilde u( r,z,t) = \frac{1}{\lambda(t)} \, \epsilon\left(
\frac{r-r(t)}{\lambda(t)}, \frac{z-z(t)}{\lambda(t)}, t\right) \, e^{i \gamma(t)}.
$$
It is important that we consider, by default, $u$ and $\tilde u$ as
~{\it 3d objects}\, in the spatial variables.  Thus, when we write $\|\tilde
u \|_{L^2}$, we mean $ \| \tilde u(x,y,z,t)\|_{L^2(dxdydz)} = \|
\tilde u(r,z,t) \|_{L^2(rdrdz)}$.  On the other hand, we consider,
by default, $\tilde Q_{b(t)}$ and $\epsilon$ to be {\it 2d objects} in the
spatial variables, and thus, if we were to write $\|\tilde
Q_{b(t)}\|_{L^2}$, we would just mean $\|\tilde Q_{b(t)}(\tilde r,
\tilde z) \|_{L^2(d\tilde r, d\tilde z)}$. However, if working with
the $\tilde r$, $\tilde z$ variables and the function $\tilde
Q_{b(t)}$ and $\epsilon$, we will write the integrals out to help
avoid confusion.

We would like to know that the geometric description holds for $0\leq t < T$,
and in addition that  properties BSO 1--8 listed below hold for all $0 \leq t<T$.
To show this, we do a bootstrap argument.  By IDA 1--11 and continuity of the flow $u(t)$
in $H^1$, we know that for some $t_1>0$, Lemma \ref{L:geomdecomp} applies giving
$(\lambda(t), \gamma(t), r(t),z(t), b(t))$ and $\epsilon$ such that ORTH 1--4 hold
on $0\leq t<t_1$ with initial configuration $\lambda(0)=\lambda_0$, etc.
Again by the continuity of $u(t)$ (and $\epsilon(t)$) in $H^1$ we know that
BSI 1--8 hold on $0\leq t < t_1$ by taking $t_1$ smaller, if necessary.
Now take $t_1$ to be the \emph{maximal} time for which Lemma \ref{L:geomdecomp} applies
and BSI 1--8 hold on $0\leq t<t_1$.   (By the above reasoning, we must have $t_1>0$.)
Under these hypotheses, we show that BSO 1--8 hold on $0\leq t<t_1$.  Since these
properties are all \emph{strictly} stronger than those of BSI 1--8, we must have $t_1=T$.

We now outline this bootstrap argument.  We have the following
``bootstrap inputs'' that we enumerate as BSI 1, etc.  We assume
that all the following properties hold for times $0\leq t <t_1 <T$.

\begin{myenumerate}{BSI}
\item
Proximity of $r(t)$ to $1$, or localization of the singular circle
$$
|r(t)-1| \leq (\alpha^*)^\frac{1}{2}
$$
and proximity of $z(t)$ to $0$
$$
|z(t)| \leq (\alpha^*)^\frac{1}{2} \,.
$$
\item
Smallness of $b(t)$, or closeness of $Q_b$ to $Q$ on the singular
circle
$$
0<b(t)< (\alpha^*)^\frac{1}{8}.
$$
\item
Control of $\epsilon(t)$ error by radiative asymptotic
$\Gamma_{b(t)}$
$$
\mathcal{E}(t) \equiv \int |\nabla_{(\tilde r, \tilde z)}\epsilon(\tilde r,\tilde z,t)|^2
\mu_{\lambda(t),r(t)} \, d\tilde r d\tilde z + \int_{|\tilde R|< \frac{10}{b(t)}}
|\epsilon(\tilde r, \tilde z,t)|^2 e^{-|\tilde R|}
\, d\tilde r d\tilde z \leq \Gamma_{b(t)}^\frac{3}{4}.
$$
\item
Control of the scaling parameter $\lambda(t)$ by the
radiative asymptotic $\Gamma_{b(t)}$
$$
\lambda^2(t)|E_0| \leq \Gamma_{b(t)}^2.
$$
\item
Control of the $r$-localized momentum by the radiative asymptotic
$$
\lambda(t) \left| \Im \int \nabla \psi \cdot \nabla u(t)
\, \bar u(t) \, dxdydz \right| \leq \Gamma_{b(t)}^2.
$$
\item
Control of the scaling parameter $\lambda(t)$ by $b(t)$
$$
0< \lambda(t) \leq \exp \left( -\exp \frac{\pi}{10b(t)}\right).
$$
\item
Global $L^2$ bound on $\tilde u(t)$
$$
\|\tilde u(t)\|_{L^2} \leq (\alpha^*)^\frac{1}{10}.
$$
\item
Smallness of $\tilde u(t)$ outside the singular circle
$$
\| \chi_1(r) \tilde u(t) \|_{\dot H^\frac{1}{2}} \leq (\alpha^*)^\frac{1}{4}.
$$
\end{myenumerate}

Assuming that properties BSI 1--8 hold for $0\leq t<t_1<T$, we prove
that the following ``bootstrap outputs,'' labeled BSO 1--9 hold.

\begin{myenumerate}{BSO}
\item
Proximity of $r(t)$ to $1$, or localization of the singular circle
$$
|r(t)-1| \leq (\alpha^*)^\frac{2}{3}
$$
and proximity of $z(t)$ to $0$
$$
|z(t)| \leq (\alpha^*)^\frac{2}{3}.
$$

\item
Smallness of $b(t)$, or closeness of $Q_b$ to $Q$ on the singular circle
$$
0<b(t)< (\alpha^*)^\frac{1}{5}.
$$

\item
Control of $\epsilon(t)$ error by radiative asymptotic $\Gamma_{b(t)}$
$$
\mathcal{E}(t)\equiv \int |\nabla_{(\tilde r, \tilde z)}\epsilon(\tilde r,\tilde z,t)|^2 \mu_{\lambda(t),r(t)} \, d\tilde r d\tilde z + \int_{|\tilde R|< \frac{10}{b(t)}} |\epsilon(\tilde r, \tilde z,t)|^2 e^{-|\tilde R|} \, d\tilde r d\tilde z \leq \Gamma_{b(t)}^\frac{4}{5}.
$$

\item
Control of the scaling parameter $\lambda(t)$ by the radiative asymptotic $\Gamma_{b(t)}$
$$
\lambda^2(t)|E_0| \leq \Gamma_{b(t)}^4.
$$

\item
Control of the $r$-localized momentum by the radiative asymptotic
$$
\lambda(t) \left| \Im \int \nabla \psi \cdot \nabla u(t) \, \bar u(t) \, dxdydz \right| \leq \Gamma_{b(t)}^4.
$$

\item
Control of the scaling parameter $\lambda(t)$ by $b(t)$ (which will imply an upper bound on the blow-up rate)
$$
0< \lambda(t) \leq \exp \left( -\exp \frac{\pi}{ 5b(t)}\right).
$$

\item
Global $L^2$ bound on $\tilde u(t)$
$$
\|\tilde u(t)\|_{L^2} \leq (\alpha^*)^\frac{1}{ 5}.
$$

\item
$H^\frac{1}{2}$ smallness of $\tilde u(t)$ outside the singular circle
$$
\| \chi_1(r) \tilde u(t) \|_{\dot H^\frac{1}{2}} \leq (\alpha^*)^\frac{3}{8}.
$$
\end{myenumerate}

The bootstrap argument proceeds in the following steps. The nontrivial steps are detailed in the remaining sections of the paper.

\smallskip

\noindent \textbf{Step 1.  Relative sizes of the parameters $\lambda(t)$, $\Gamma_{b(t)}$, $\alpha^*$.}  Using BSI 2, BSI 6, and ZP 2, we have
\begin{equation}
\label{E:IE 0.1}
\lambda(t) \leq \Gamma_{b(t)}^{10}. \end{equation}

\smallskip

\noindent \textbf{Step 2.  Application of mass conservation.}  Using BSI 1, BSI 3, BSI 6, and $L^2$ conservation, we obtain BSO 2 and BSO 7.  In other words, mass conservation reinforces the smallness of $b(t)$ and also the smallness of the $L^2$ norm of $\tilde u$.  This is carried out in \S \ref{S:bs2}.

\smallskip

\noindent \textbf{Step 3. $\epsilon$ interaction energy is dominated by $\epsilon$ kinetic energy.}
That is, the $\epsilon$ energy behaves as if it were $L^2$ critical and subground state.
We obtain by splitting the $L^4$ term in the energy of $\epsilon$ into inner and outer radii,
using the axial %radial
Gagliardo-Nirenberg for outer radii and the usual 3d Gagliardo-Nirenberg for inner radii
\begin{equation}
\label{E:IE 1.1}
\int |\epsilon|^4 \mu(\tilde r) d\tilde rd\tilde z \leq \delta(\alpha^*) \int |\nabla_{(\tilde r, \tilde z)}\epsilon |^2 \mu(\tilde r)\, d\tilde r d\tilde z.
\end{equation}
This states that in the $\epsilon$-energy, the interaction energy term is suitably dominated by the kinetic energy term.

A useful statement that comes out of this computation is
\begin{equation}
\label{E:IE 1.3}
\int_{\tilde R<10/b} |\epsilon|^4 d\tilde rd\tilde z \leq \|\tilde u\|_{L_x^2}^2 \, \mathcal{E}(t).
\end{equation}
The proof of this statement only uses the 2d Gagliardo-Nirenberg inequality (since we have the localization $\tilde R<10/b$) and does not use the $H^{1/2}$ assumption.  This is carried out in \S \ref{S:bs3}.

\smallskip

\noindent \textbf{Step 4.  Energy conservation of $u$ recast as an $\epsilon$ statement.}  Using BSI 3, \eqref{E:IE 0.1}, BSI 1, BSI 6, BSI 7, BSI 4, BSI 8, energy conservation, and properties of $\tilde Q_b$, we obtain
\begin{equation}
\label{E:IE 1.2}
\begin{aligned}
\Big| 2(\epsilon_1, \Sigma) + 2(\epsilon_2, \Theta) - \int|\nabla_{(\tilde r, \tilde z)}\epsilon|^2\mu(\tilde r) \, d\tilde r d\tilde z & \\
+ 3\int_{|\tilde R| \leq 10/b} Q^2\epsilon_1^2 \, d\tilde r d\tilde z
+ \int_{|\tilde R|<10/b} Q^2 \epsilon_2^2 \, d\tilde r d\tilde z \Big|
&\leq \Gamma_b^{1-C\eta} + \delta(\alpha^*) \mathcal{E}(t).
\end{aligned}
\end{equation}
It results from plugging the representation of $u$ in terms of $\tilde Q_{b(t)}$ and $\epsilon$ into the energy conservation equation for $u$.  The result is basically
\begin{itemize}
\item
the energy of $\tilde Q_{b(t)}$, which is small and shows up on the right side as the $\Gamma_{b(t)}^{1-C\eta}$ term,
\item
the energy of $\epsilon$, which shows up as the $- \int|\nabla_{(\tilde r, \tilde z)}\epsilon|^2\mu(\tilde r) \, d\tilde r d\tilde z$ term on the left (from Step 3, the interaction component of the energy is small and is put on the right),
\item
cross terms resulting from $|u|^4$  which are linear, quadratic, and cubic in $\epsilon$.  The linear terms are kept on the left as $2(\epsilon_1, \Sigma) + 2(\epsilon_2, \Theta)$.  The quadratic terms are kept on the left as well, while the cubic term is estimated away.
\end{itemize}
This is carried out in \S \ref{S:bs4}.
\smallskip

\noindent \textbf{Step 5.  Momentum control assumption (BSI 5) recast as an $\epsilon$ statement.}  Using BSI 5, \eqref{E:IE 0.1}, properties of
$\tilde Q_b$, BSI 1, BSI 2, we obtain
\begin{equation}
\label{E:IE 2}
| (\epsilon_2, \nabla_{(\tilde r,\tilde z)}\Sigma)| \leq \delta(\alpha^*)\mathcal{E}(t)^{1/2}+ \Gamma_b^2.
\end{equation}
The term that we keep on the left side comes from the cross term.  This is carried out in \S \ref{S:bs5}.

\smallskip

\noindent \textbf{Step 6.  Application of the orthogonality
conditions.} In this step, the orthogonality assumptions are used to
deduce ``laws'' for time evolution of the parameters $\lambda(t)$,
$b(t)$, $\gamma(t)$.  Using the orthogonality conditions ORTH 1-4, BSI 3, and
\eqref{E:IE 0.1}, we obtain the following estimates on the modulation parameters:
\begin{equation}
\label{E:IE 3.1}
 \left| \frac{\lambda_s}{\lambda} + b\right| + |b_s| \leq c \, \mathcal{E}(t) + \Gamma_b^{1-C\eta}, \end{equation}
\begin{equation}
\label{E:IE 3.2}
\left| \tilde \gamma_s - \frac{1}{\|\Lambda Q\|_{L^2}^2} ( \epsilon_1, L_+(\Lambda^2 Q)) \right| + \left| \frac{(r_s,z_s)}{\lambda} \right| \leq \delta(\alpha^*)\mathcal{E}(t)^{1/2} + \Gamma_b^{1-C\eta},
\end{equation}
where, we recall, $L_+ = -\Delta + 1 -3Q^2$. This is carried out in \S \ref{S:bs6}.

\smallskip

\noindent \textbf{Step 7.  Deduction of BSO 4 from \eqref{E:IE 3.1}}. Using \eqref{E:IE 3.1}, BSI 4, BSI 6, properties of $\Gamma_b$, IDA 3, we obtain
$$
\frac{d}{ds}(\lambda^2e^{5\pi/b}) \leq 0,
$$
which, upon integrating in time, gives BSO 4.
This is carried out in \S \ref{S:bs7}.

\smallskip

\noindent \textbf{Step 8.  Deduction of a local virial inequality.}  Using \eqref{E:IE 1.2}, \eqref{E:IE 3.1}, \eqref{E:IE 3.2}, the coercivity property, and the spectral property for $d=2$, the orthogonality conditions,  we obtain a ``local virial inequality''
\begin{equation}
\label{E:IE 4}
b_s \geq \delta_0\mathcal{E}(t) - \Gamma_{b(t)}^{1-C\eta},
\end{equation}
This is carried out in \S \ref{S:bs8}.

\smallskip

\noindent \textbf{Step 9.  Lower bound on $b(s)$ ($\implies$ upper bound on blow-up rate)}  Using \eqref{E:IE 4}, \eqref{E:IE 3.1}, BSI 4, BSI 6, BSI 2, and IDA 1, IDA 4, we obtain
\begin{equation}
\label{E:IE 5}
b(s) \geq \frac{3\pi}{4\log s}
\end{equation}
\begin{equation}
\label{E:IE 6}
\lambda(s) \leq \sqrt{\lambda_0} e^{-\frac{\pi}{3}\frac{s}{\log s}}
\end{equation}
which, together imply BSO 6:
$$
\eqref{E:IE 5}+\eqref{E:IE 6}\quad \implies \quad \text{BSO 6} \Longleftrightarrow \frac{\pi}{5}\leq b(t)\log|\log\lambda(t)| .
$$
This will later be used to give an upper bound on the blow-up rate.  \eqref{E:IE 5}--\eqref{E:IE 6} are consequences of a careful integration of \eqref{E:IE 4} and an application of the law for the scaling parameter \eqref{E:IE 3.1}.
This is carried out in \S \ref{S:bs9}.

\smallskip

\noindent \textbf{Step 10.  Control on the radius of concentration.}  Using \eqref{E:IE 3.2}, \eqref{E:IE 6}, IDA 2, IDA 4, IDA 1, we obtain BSO 1.
This is carried out in \S \ref{S:bs10}.

\smallskip

\noindent \textbf{Step 11.  Momentum conservation implies momentum control estimate (BSO 5).}  Using BSI 3, BSI 7, proof of BSO 6, \eqref{E:IE 3.1}, BSI 4, BSI 6, IDA 3, we obtain BSO 5.
This is carried out in \S \ref{S:bs11}.

\smallskip

\noindent \textbf{Step 12.  Refined virial inequality in the radiative regime}.
Here we prove a refinement of \eqref{E:IE 4} in the radiative regime.
Let $\phi_3(\tilde R) =1$ for $\tilde R \leq 1$ and $\phi_3(\tilde R) = 0$ for $\tilde R \geq 2$ be a radial cutoff. With $\tilde\zeta_b = \phi_3(\tilde R/A)\zeta_b$,
$A=e^{2a/b}$ for some small constant $0<a\ll 1$, we define $\tilde\epsilon=\epsilon-\tilde\zeta$.  In this step we establish
\begin{equation}
\label{E:IE 7}
\partial_s  f_1(s) \geq \delta_1 \tilde{\mathcal{E}}(t) +c\Gamma_b - \frac{1}{\delta_1}\int_{A\leq \tilde R \leq 2A} |\epsilon|^2 d\tilde rd\tilde z,
\end{equation}
where
$$
f_1(s)\equiv \frac{b}{4}\|\tilde R\tilde Q_b\|_{L^2_{\tilde r\tilde z}}^2 + \frac12\Im \int (\tilde r,\tilde z)\cdot \nabla \tilde \zeta \; \tilde \zeta + (\epsilon_2,\Lambda\tilde \zeta_{\text{re}}) - (\epsilon_1,\Lambda\tilde \zeta_{\text{im}})
$$
and
$$
\tilde{\mathcal{E}}(t) = \int |\nabla_{(\tilde r, \tilde z)}\tilde \epsilon(\tilde r,\tilde z,t)|^2 \mu_{\lambda(t),r(t)} \, d\tilde r d\tilde z + \int_{|\tilde R|< \frac{10}{b(t)}} |\tilde \epsilon(\tilde r, \tilde z,t)|^2 e^{-|\tilde R|} \, d\tilde r d\tilde z .
$$
This is carried out in \S \ref{S:bs12}.

\smallskip

\noindent\textbf{Step 13.  $L^2$ dispersion at infinity in space}.  Let $\phi_4(\tilde R)$ be a (nonstrictly) increasing radial cutoff to large radii.  Specifically, we require that $\phi_4(\tilde R) = 0$ for $0\leq \tilde R \leq \frac12$ and $\phi_4(\tilde R) = 1$ for $\tilde R \geq 3$, with $\frac14 \leq \phi_4'(\tilde R) \leq \frac12$ for $1\leq \tilde R \leq 2$.    We next prove, via a flux type computation, an estimate giving control on the term $\int_{A\leq |\tilde R|\leq 2A} |\epsilon|^2 \, d\tilde rd\tilde z$:
\begin{equation}
\label{E:IE 8}
\partial_s \left( \frac{1}{r(s)}\int \phi_4\left(\frac{\tilde R}{A}\right) |\epsilon|^2 \mu(\tilde r) d\tilde r d\tilde z\right) \geq \frac{b}{400}\int_{A\leq |\tilde R| \leq 2A} |\epsilon|^2 - \Gamma_b^{a/2}\mathcal{E}(t) - \Gamma_b^2 .
\end{equation}
This is carried out in \S \ref{S:bs13}.

\smallskip

\noindent\textbf{Step 14.  Lyapunov functional in $H^1$.}  By combining \eqref{E:IE 7} and \eqref{E:IE 8}, we define a Lyapunov functional $\mathcal{J}$ and show
\begin{equation}
\label{E:IE 9}
\partial_s \mathcal{J} \leq -cb\left( \Gamma_b + \tilde{\mathcal{E}}(t) + \int_{A\leq |\tilde R|\leq 2A} |\epsilon|^2 \right),
\end{equation}
where $\mathcal{J}$ (defined later) can be shown to satisfy
\begin{equation}
\label{E:IE 10}
|\mathcal{J}-d_0b^2|\leq \tilde\delta_1b^2
\end{equation}
for some universal constant $0< \tilde \delta_1\ll 1$, and a more refined control
\begin{equation}
\label{E:IE 11}
\mathcal{J}(s) - f_2(b(s)) \;
\begin{cases}
\geq -\Gamma_b^{1-Ca} + \frac{1}{C}\mathcal{E}(s) \\
\leq +\Gamma_b^{1-Ca} + CA^2\mathcal{E}(s)
\end{cases}
\end{equation}
with $f_2$ given by
$$
f_2(b) = \left( \int |\tilde Q_b|^2 - \int |Q|^2 \right) - \frac{\delta_1}{800} \left( b\tilde f_1(b) - \int_0^b \tilde f_1(v)dv \right)
$$
and
$$
\tilde f_1(b) = \frac{b}{4}\|\tilde R \tilde Q_b\|_{L^2}^2 +\frac12\Im \int (\tilde r,\tilde z) \cdot \nabla \tilde \zeta \; \overline{\tilde \zeta} d\tilde rd\tilde z.
$$
Here, $0< \frac{df_2}{db^2} \Big|_{b^2=0} < +\infty$, and hence, \eqref{E:IE 11} refines \eqref{E:IE 10}.
This is carried out in \S \ref{S:bs14}.

\smallskip

\noindent \textbf{Step 15.  Deduction of control on $\mathcal{E}(t)$ and upper bound on $b$ ($\implies$ lower bound on blow-up rate).}  By integrating \eqref{E:IE 9} and applying \eqref{E:IE 10}, we prove
\begin{equation}
\label{E:IE 12}
b(s) \leq \frac{4\pi}{3\log s}
\end{equation}
and
\begin{equation}
\label{E:IE 13}
\int_{s_0}^s \left( \Gamma_{b(\sigma)} + \tilde{\mathcal{E}}(\sigma) \right) d\sigma \leq c\,\alpha^*. \end{equation}
Using \eqref{E:IE 9} and \eqref{E:IE 11}, we prove BSO 3, the key dispersive control on the remainder term $\epsilon$.
This is carried out in \S \ref{S:bs15}.

\smallskip

\noindent \textbf{Step 16. $H^{1/2}$ interior smallness.}  Using a local smoothing estimate and \eqref{E:IE 13}, we prove BSO 8.
This is carried out in \S \ref{S:bs16}.

\smallskip

This concludes the outline of the bootstrap argument, and we now know that BSO 1--8 hold
for all times $0<t<T$.  The remainder of the proof of Theorem \ref{T:main} is carried out
in \S\ref{S:loglog}--\ref{S:rem-conv}.
%\ref{S:rem-desc}.

\section{The equation for $\epsilon$}

Recall \eqref{E:NLSaxial} and write
\begin{equation}
\label{E:comp_ep_1}
u(t,r,z) = \frac{1}{\lambda(t)} \tilde Q_{b(t)}\left(\frac{r-r(t)}{\lambda(t)}, \frac{z-z(t)}{\lambda(t)} \right)e^{i\gamma(t)} + \frac{1}{\lambda(t)} \epsilon\left( \frac{r-r(t)}{\lambda(t)}, \frac{z-z(t)}{\lambda(t)}, t\right)e^{i\gamma(t)}.
\end{equation}
In the remainder of this section, $Q$ will always mean $\tilde Q_{b(t)}$.
Recall
$$
\tilde r = \frac{r-r(s)}{\lambda(s)}, \qquad \tilde z = \frac{z-z(s)}{\lambda(s)},
$$
and thus,
$$
\partial_s \tilde r = -\frac{r_s}{\lambda} - \tilde r \frac{\lambda_s}{\lambda}, \qquad \partial_s \tilde z = -\frac{z_s}{\lambda} - \tilde z \frac{\lambda_s}{\lambda}.
$$
Direct computation when substituting \eqref{E:comp_ep_1} into \eqref{E:NLSaxial} gives
\begin{align*}
e^{-i\gamma} \lambda^3 i\partial_t u &= e^{- i\gamma} \lambda i \partial_s u \\
&=
\begin{aligned}[t]
& -i \frac{\lambda_s}{\lambda} Q + i \left( - \frac{r_s}{\lambda} - \tilde r \frac{\lambda_s}{\lambda} \right)\partial_{\tilde r}Q + i\left( - \frac{z_s}{\lambda} - \tilde z \frac{\lambda_s}{\lambda} \right)\partial_{\tilde z} Q + ib_s \partial_bQ - \gamma_sQ \\
& +i\partial_s \epsilon -i \frac{\lambda_s}{\lambda} \epsilon + i  \left( - \frac{r_s}{\lambda}  - \tilde r \frac{\lambda_s}{\lambda} \right)\partial_{\tilde r}\epsilon  + i\left( - \frac{z_s}{\lambda} - \tilde z \frac{\lambda_s}{\lambda} \right)\partial_{\tilde z} \epsilon - \gamma_s\epsilon
\end{aligned} \\
&=
\begin{aligned}[t]
& -i\left( \frac{\lambda_s}{\lambda} +b \right)\Lambda Q -\frac{i}{\lambda}(r_s,z_s)\cdot \nabla_{(\tilde r, \tilde z)}Q +ib_s \partial_b Q - \tilde \gamma_s Q + (ib\Lambda Q - Q)\\
& +i\partial_s \epsilon -i \left( \frac{\lambda_s}{\lambda} +b \right) \Lambda \epsilon - \frac{i}{\lambda}(r_s,z_s) \cdot \nabla_{(\tilde r, \tilde z)} \epsilon - \tilde\gamma_s \epsilon + (ib\Lambda\epsilon  - \epsilon),
\end{aligned}
\end{align*}
where $\tilde \gamma(s) = \gamma(s) - s$.

Also,
$$
\lambda^3 e^{-i\gamma}\Delta_{(r,z)}u
=  \Delta_{(\tilde r, \tilde z)}Q + \Delta_{(\tilde r, \tilde z)}\epsilon
$$
and
$$
\lambda^3 e^{-i\gamma}\frac{1}{r} \partial_r u
= \frac{\lambda}{\mu} \partial_{\tilde r}Q
+  \frac{\lambda}{\mu} \partial_{\tilde r}\epsilon .
$$
Finally the nonlinear term:
\begin{align*}
\lambda^3 e^{-i\gamma} |u|^2 u &= |Q + \epsilon|^2(Q + \epsilon) \\
&= |Q|^2Q + \underbrace{2(\Re Q\bar\epsilon)Q + |Q|^2\epsilon}_{\textnormal{linear}} + \underbrace{Q|\epsilon|^2 + 2(\Re Q\bar \epsilon)\epsilon}_{\textnormal{quadratic}}+ \underbrace{|\epsilon|^2\epsilon}_{\textnormal{cubic}} \\
&= |Q|^2Q + [(2\Sigma^2 + |Q|^2) \epsilon_1 + 2\Sigma \Theta \epsilon_2] + i[ (2\Theta^2+|Q|^2)\epsilon_2 + 2\Sigma \Theta \epsilon_1] + R_1(\epsilon) + iR_2(\epsilon),
\end{align*}
where the quadratic and cubic terms are put into $R(\epsilon)$.

Adding up all of the above in \eqref{E:NLSaxial}, and taking the real part, we get:
\begin{equation}
 \label{E:Epsilon2}
 b_s\partial_b \Theta + \partial_s \epsilon_2 + M_+(\epsilon) + b \Lambda \epsilon_2 =
\begin{aligned}[t]
&~\left( \frac{\lambda_s}{\lambda} + b \right) \Lambda \Theta + \frac{1}{\lambda}(r_s,z_s)\cdot \nabla_{(\tilde r, \tilde z)} \Theta - \tilde \gamma_s \Sigma \\
&+\left( \frac{\lambda_s}{\lambda} + b \right) \Lambda \epsilon_2 + \frac{1}{\lambda} (r_s,z_s) \cdot \nabla_{(\tilde r, \tilde z)} \epsilon_2 - \tilde \gamma_s \epsilon_1\\
&-\Re \tilde \Psi + R_1(\epsilon),
\end{aligned}
\end{equation}
where
\begin{equation}
 \label{E:M+}
M_+(\epsilon) = \epsilon_1 - \Delta_{(\tilde r, \tilde z)}
\epsilon_1 - \frac{\lambda}{\mu} \partial_{\tilde r} \epsilon_1 -
\left( \frac{2\Sigma^2}{|Q|^2} + 1 \right)|Q|^2\epsilon_1 -
2\Sigma\Theta\epsilon_2.
\end{equation}

Taking the imaginary part, we get

\begin{equation}
 \label{E:Epsilon1}
 b_s \partial_b \Sigma + \partial_s \epsilon_1 - M_-(\epsilon) + b\Lambda \epsilon_1 =
\begin{aligned}[t]
& ~\left( \frac{\lambda_s}{\lambda} + b \right) \Lambda\Sigma
+ \frac{1}{\lambda}(r_s,z_s)\cdot \nabla_{(\tilde r, \tilde z)} \Sigma
+ \tilde \gamma_s \Theta \\
&+ \left( \frac{\lambda_s}{\lambda}+b \right) \Lambda \epsilon_1
+ \frac{1}{\lambda}(r_s,z_s)\cdot \nabla_{(\tilde r, \tilde z)}\epsilon_1
+ \tilde \gamma_s \epsilon_2 \\
&+ \Im \tilde \Psi - R_2(\epsilon),
\end{aligned}
\end{equation}
where
\begin{equation}
 \label{E:M-}
 M_-(\epsilon) = \epsilon_2 - \Delta_{(\tilde
r, \tilde z)} \epsilon_2 - \frac{\lambda}{\mu}\partial_{\tilde r}
\epsilon_2 - \left( \frac{2\Theta^2}{|Q|^2} + 1 \right)|Q|^2
\epsilon_2 - 2\Sigma\Theta\epsilon_1.
\end{equation}

\section{Bootstrap Step 2.  Application of mass conservation}
\label{S:bs2}

In this section we prove BSO 2 and BSO 7 are consequences of the
$L^2$ conservation of $u(t)$ and several other BSI's.  The assumed
smallness of $b(t)$, $\epsilon(t)$, and $\lambda(t)$, initial
smallness of $\tilde u$ (at $t=0$) combined with $L^2$ norm
conservation for $u(t)$ reinforces the smallness of $\tilde u(t)$ in $L^2$
for $0<t<t_1$ and smallness of $b(t)$ for $0<t<t_1$.

%Write
Recall \eqref{E:comp_ep_1} and denote
$$\tilde u(r,z,t) = \frac{1}{\lambda(t)} \, \epsilon \left(
\frac{r-r(t)}{\lambda(t)}, \frac{z-z(t)}{\lambda(t)},t \right)
e^{i\gamma(t)}. $$
%where the second term can also be written $\tilde u(r,z,t)$.
Substitute \eqref{E:comp_ep_1} into the mass conservation law
$$
\int |u(r,z,t)|^2 rdrdz = \| u_0\|_{L^2}^2
$$
and rescale to obtain
\begin{equation}
\label{E:L2cons}
\begin{aligned}
\|u_0\|_{L^2}^2 =
\begin{aligned}[t]
& \int |\tilde Q_{b(t)}(\tilde r, \tilde z)|^2 \mu(\tilde r) \, d\tilde r d\tilde z \\
&+ 2 \Re \int \tilde Q_{b(t)}(\tilde r, \tilde z)
\bar \epsilon( \tilde r, \tilde z) \, \mu(\tilde r) \, d\tilde r d\tilde z \\
&+ \int |\tilde u (r,z,t)|^2 rdrdz.
\end{aligned}
\end{aligned}
\end{equation}
In the first term, write out $\mu(\tilde r) =\lambda \tilde r +
(r(t)-1) + 1$, which splits the integral into three pieces. The
third piece we keep; the first two we estimate using
$$
[\text{BSI }6: \lambda \leq \exp(-\exp \pi/10b(t))
\text{ and BSI }2: b(t) \leq (\alpha^*)^{1/8}] \implies \lambda \leq (\alpha^*)^{1/2}
$$
and BSI 1: $|r(t)-1| \leq (\alpha^*)^{1/2}$. The second term in \eqref{E:L2cons} we estimate using Cauchy-Schwarz and the assumed control on
$\epsilon(t)$ given by BSI 3, ZP 2, and BSI 2:  $\mathcal{E}(t) \leq
\Gamma_{b(t)}^{3/4} \leq (\alpha^*)^{1/2}$.  Collecting the results, we now have
$$
\int |\tilde Q_{b(t)}|^2 d\tilde rd \tilde z + \|\tilde u\|_{L^2}^2 = \|u_0\|_{L^2}^2 + \mathcal{O}((\alpha^*)^{1/2}).
$$
Finally, we use QP3:
$$d(b^2) \equiv \frac{d}{db^2} \|\tilde Q_b\|_{L_{\tilde r\tilde z}^2}^2, \quad d_0=d(0)>0, \quad d \text{ differentiable},$$
which gives
$$
\|\tilde Q_b\|_{L_{\tilde r\tilde z}^2}^2 = \|Q\|_{L_{\tilde r\tilde z}^2}^2
+ d_0b^2 + \mathcal{O}(b^4 \sup_{0\leq \sigma\leq b^2}|d'(\sigma)|).
$$
Since $b^4\leq (\alpha^*)^{1/2}$, we have
$$
d_0b(t)^2 + \|\tilde u(t)\|_{L^2_{xyz}}^2 = \|u_0\|_{L^2}^2 - \|Q\|_{L^2_{\tilde r\tilde z}}^2 + \mathcal{O}((\alpha^*)^{1/2})
$$
and conclude by IDA 9 to get
$$
d_0b(t)^2 + \|\tilde u(t)\|_{L^2_{xyz}}^2 \leq c(\alpha^*)^{1/2},
$$
which implies BSO 2 and BSO 7.

\section{Bootstrap Step 3.  Interaction energy $\ll$ kinetic energy for $\epsilon$}
\label{S:bs3}

Here we deduce \eqref{E:IE 1.1} (or %in Rapha\"el's paper,
(3.94) in \cite{R}), which controls $|\epsilon|^4$ by $|\nabla \epsilon|^2$,
and is useful later on.

By change of variables, we have
$$
\int |\epsilon|^4 \mu(\tilde r) \, d\tilde r d\tilde z = \lambda^2(t) \int_{|r|\leq \frac12} |\tilde u|^4 \, r\, drdz + \lambda^2(t) \int_{|r|\geq \frac12} |\tilde u|^4 \, r\, drdz.
$$
For the first term, we use the ($r$-localized) 3d Gagliardo-Nirenberg estimate.  For the second term, we use the axially symmetric exterior Gagliardo-Nirenberg estimate Lemma \ref{L:cyl-Strauss}.  Together, they give the bound
$$
\leq \lambda^2(t)\|\tilde u\|_{H^{\frac12}_{\{xyz, \;
|r|<\frac14\}}}^2 \|\nabla \tilde u\|_{L^2_{xyz}}^2 +
\lambda^2(t)\|\tilde u\|_{L^2_{\{xyz, \; |r|>\frac14\}}}^2 \|\nabla
\tilde u\|_{L^2_{xyz}}^2.
$$
We then apply BSI 8 (smallness of $\|\tilde u\|_{\dot H^{1/2}}$ for $|r|<\frac14$) and BSO 7 (smallness of $\|\tilde u\|_{L^2}$ globally in $r$) to get the estimate
$$
\leq (\alpha^*)^{1/4}\lambda^2(t)\|\nabla_{xyz}\tilde
u\|_{L^2_{xyz}}^2.
$$
For axially symmetric functions, we have $|\partial_r u|^2 = |\partial_x u|^2 + |\partial_y u|^2$, so we may replace $\nabla_{xyz}$ by $\nabla_{rz}$, then rescale to obtain
$$\leq (\alpha^*)^{1/4} \int_{\tilde r, \tilde z} |\nabla_{(\tilde r, \tilde z)}\epsilon|^2 \mu(\tilde r) \, d\tilde r d\tilde z$$
and this is \eqref{E:IE 1.1}.

\section{Bootstrap Step 4.  Energy conservation}
\label{S:bs4}

In this step, we prove \eqref{E:IE 1.2} as a consequence of various bootstrap assumptions and the conservation of energy.

Plug \eqref{E:comp_ep_1}
%$$u(r,z,t) = \frac{1}{\lambda}\tilde Q_{b(t)}\left( \frac{r-r(t)}{\lambda}, %\frac{z-z(t)}{\lambda} \right)e^{i\gamma(t)} + \frac{1}{\lambda} \epsilon %\left(\frac{r-r(t)}{\lambda}, \frac{z-z(t)}{\lambda},t \right)e^{i\gamma(t)}$$
into the energy conservation identity
$$
2\lambda^2E_0 = \lambda^2\int |\nabla_{(r,z)}u|^2 rdrdz - \frac12 \lambda^2\int |u|^4 rdrdz.
$$
The result for the first term on the right side is
\begin{align*}
\lambda^2 \int |\nabla_{(r,z)}u|^2 rdrdz &=
\begin{aligned}[t]
&\int |\nabla_{(\tilde r, \tilde z)}\tilde Q_{b(t)}|^2 \mu(\tilde r) \, d\tilde rd\tilde z \\
&+ 2\Re \int \nabla_{(\tilde r, \tilde z)}\tilde Q_{b(t)} \cdot \nabla_{(\tilde r, \tilde z)} \bar \epsilon \; \mu(\tilde r) d\tilde r d\tilde z \\
&+ \int |\nabla_{(\tilde r, \tilde z)} \epsilon|^2 \, \mu(\tilde r)d\tilde r d\tilde z
\end{aligned} \\
&= \text{I}.1 + \text{I}.2 + \text{I}.3 .
\end{align*}
For the second term, we obtain
\begin{align*}
-\frac12 \lambda^2 \int |u|^4 rdrdz &=
\begin{aligned}[t]
&-\frac12\int |\tilde Q_{b(t)}|^4\mu(\tilde r) \,d\tilde r d\tilde z -2\int |\tilde Q_{b(t)}|^2 \Re (\tilde Q_{b(t)}\bar \epsilon) \mu(\tilde r) \, d\tilde r d\tilde z\\
&-2\int ( \Re \tilde Q_{b(t)}\bar\epsilon)^2 \mu(\tilde r) \, d\tilde r d\tilde z - \int |\tilde Q_{b(t)}|^2 |\epsilon|^2 \mu(\tilde r) \, d\tilde r d\tilde z &\leftarrow \text{quadratic}\\
&-2\int |\epsilon|^2 \Re (\tilde Q_{b(t)}\bar\epsilon) \mu(\tilde r) \, d\tilde r d\tilde z &\leftarrow \text{cubic}\\
&-\frac12 \int |\epsilon|^4 \mu(\tilde r) \, d\tilde r d\tilde z
\end{aligned}
\\
&= \text{II.1}+\text{II.2}+\text{II.3}+\text{II.4}+\text{II.5}+\text{II.6} .
\end{align*}
In all of these terms except I.3 and II.6, we can write out $\mu(\tilde r) = \lambda \tilde r + r(t)$ and ``discard'' the $\lambda \tilde r$ term by estimation, since $\lambda(t) \leq \Gamma_{b(t)}^{10}$.  We should thus hereafter in this section replace $\mu(\tilde r)$ with $r(t)$ for all terms but I.3 and II.6.
We have the $\tilde Q_{b(t)}$ energy terms:
$$
r(t)^{-1}(\text{I.1} + \text{II.1}) = 2 \,E(\tilde Q_{b(t)}).
$$
We also have the linear in $\epsilon$ terms that we combine, and then substitute the equation for $\tilde Q_{b(t)}$ to get \begin{align*}
r(t)^{-1}(\text{I.2} + \text{II.2}) &= -2\Re \int (\Delta_{(\tilde r, \tilde z)} \tilde Q_{b(t)} + |\tilde Q_{b(t)}|^2 \tilde Q_{b(t)}) \bar \epsilon \, d\tilde r d\tilde z \\
&= -2\Re \int \tilde Q_{b(t)} \bar\epsilon \, d\tilde r  d\tilde z - 2b\Im \int \Lambda \tilde Q_{b(t)} \, \bar \epsilon \, d\tilde r d\tilde z +2\Re \int \Psi_{b(t)} \bar \epsilon \, d\tilde r d\tilde z .
\end{align*}
The middle term is zero by ORTH 4, %orthogonality conditions,
and the first term can be rewritten to obtain
$$
r(t)^{-1}(\text{I.2} + \text{II.2})= -2(\Sigma, \epsilon_1) -2(\Theta, \epsilon_2) + 2 \Re \int \Psi_{b(t)} \, \bar \epsilon \, d\tilde r d\tilde z .
$$
Next, for the quadratic in $\epsilon$ terms in II, replace $\tilde Q_{b(t)}$ by $Q$ and use the proximity estimates for $\tilde Q_{b(t)}$ to $Q$ to control the error
$$
r(t)^{-1}(\text{II}.3+\text{II}.4) = -3\int Q^2 \epsilon_1^2 \, d\tilde r d\tilde z- \int Q^2 \epsilon_2^2 d\tilde r d\tilde z + \text{errors}.
$$
For the cubic term in II, we estimate as (using that $1/b \ll 1/\lambda$)
\begin{align*}
r(t)^{-1} |\text{II}.5| &\leq 2 \int |\epsilon|^3 |\tilde Q_{b(t)}| d\tilde r d\tilde z\\
&\leq 2\left( \int_{\tilde R \leq 1/10\lambda} |\epsilon|^4 d\tilde r d\tilde z \right)^{1/2} \left( \int |\epsilon|^2 |\tilde Q|^2 d\tilde r d\tilde z \right)^{1/2}. \end{align*}
By 2d Gagliardo-Nirenberg applied to the first term (where $\phi$ is a smooth cutoff to $\tilde R \leq 1/10\lambda$)
$$
\lesssim 2\left( \int_{\tilde R \leq 1/10\lambda} |\epsilon|^2 d\tilde r d\tilde z \right)^{1/2}\left(\int |\nabla_{(\tilde r, \tilde z)}[ \phi(\tilde R \lambda)\epsilon(\tilde R)]|^2 d\tilde r d\tilde z\right)^{1/2}  \left( \int |\epsilon|^2 |\tilde Q|^2 d\tilde r d\tilde z \right)^{1/2}.
$$
Note that if $\tilde R\leq 1/10\lambda$, then $\mu(\tilde r) \sim 1$. The first term above is controlled by some $\alpha^*$ power (using the $\tilde R$ restriction to reinsert a $\mu(\tilde r)$ factor) by rescaling back to $\tilde u$ by BSI 7. The second term is controlled as
\begin{align*}
\indentalign \left( \lambda \int_{|\tilde R|\ll \frac1{\lambda}} |\epsilon|^2 \,d\tilde r \, d\tilde z + \int_{|\tilde R| \ll \frac{1}{\lambda}} |\nabla \epsilon|^2 \,d\tilde r \, d\tilde z\right)^{1/2} \\
&\leq \lambda \|\tilde u\|_{L^2} + \left( \int_{|\tilde R| \ll \frac{1}{\lambda}} |\nabla \epsilon|^2 \mu(\tilde r) \, d\tilde r \, d\tilde z \right)^{1/2} \\
&\leq \lambda + \mathcal{E}(t)^{1/2} \,.
\end{align*}
For the third term, we use $|\tilde Q_{b(t)}| \leq e^{-\tilde R/2}$.  These considerations give
$$
r(t)^{-1}|\text{II}.5| \leq (\alpha^*)^{1/5}( \lambda^2 + \mathcal{E}(t)) \,.
$$
The quartic term in $\epsilon$ in II is controlled by \eqref{E:IE 1.1}.
We next note that
$$
\text{BSI }1: \; |r(t)-1| \leq (\alpha^*)^{1/2}\implies \frac{1}{r(t)} \sim 1.
$$
%and then \eqref{E:IE 2} follows.
Collecting all of the above estimates and manipulations, we obtain \eqref{E:IE 1.2}.

\section{Bootstrap Step 5.  Momentum control assumption recast as an $\epsilon$ statement}
\label{S:bs5}

In this section we prove \eqref{E:IE 2}. Recall BSI 5:
$$
\lambda(t) \left| \Im \int \nabla \psi \cdot \nabla u(t) \; \bar u(t) \, dxdydz \right|
\leq \Gamma_{b(t)}^2.
$$
A basic calculus fact states that for axially symmetric functions $f$ and $g$,
$$
\nabla_{(x,y,z)}f \cdot \nabla_{(x,y,z)}g = \nabla_{(r,z)}f \cdot \nabla_{(r,z)} g.
$$
Plug in \eqref{E:comp_ep_1} into the left side of BSI 5
%$$u(r,z,t) = \frac{1}{\lambda}\tilde Q_{b(t)}\left( \frac{r-r(t)}{\lambda}, %\frac{z-z(t)}{\lambda} \right)e^{i\gamma(t)} + \frac{1}{\lambda} \epsilon %\left(\frac{r-r(t)}{\lambda}, \frac{z-z(t)}{\lambda},t \right)e^{i\gamma(t)}$$
and change variables $(r,z)$ to $(\tilde r,\tilde z)$ to get
\begin{align*}
\indentalign
\lambda(t) \Im \int \nabla \psi \cdot \nabla u(t) \; \bar u(t) \, dxdydz \\
&= \Im \int \left[ (\nabla_{(r,z)}\psi)(r,z) \cdot \left(\nabla_{(\tilde r, \tilde z)}
(\tilde Q_{b(t)}+\epsilon)\right)(\tilde r, \tilde z) \right]
\; (\overline{\tilde Q_{b(t)}}+\bar\epsilon) (\tilde r, \tilde z) \,
\mu(\tilde r) d\tilde r\, d\tilde z \,, \end{align*}
then using that $\psi(r)=(1,0)$ on the support of $\tilde Q_{b(t)}$, we continue as
\begin{align*}
&=
\begin{aligned}[t]
&\Im \int \partial_{\tilde r} \tilde Q_{b(t)} \, \overline{\tilde Q_{b(t)}}\, \mu(\tilde r) \, d\tilde r d\tilde z\\
&+ 2\Im \int \partial_{\tilde r} \tilde Q_{b(t)} \, \bar\epsilon \mu(\tilde r) \, d\tilde rd\tilde z \\
&+ \Im \int \partial_{\tilde r} \epsilon \; \bar\epsilon \, \mu(\tilde r)\, d\tilde rd\tilde z
\end{aligned}\\
&= \text{I}+\text{II}+\text{III}. \end{align*}
In term I, we expand out $\mu(\tilde r) = \lambda(t)\tilde r + r(t)$
which gives two terms: the first of these we estimate out and use
$\lambda(t)\leq \Gamma_{b(t)}^2$, and the second of these is $0$ by
QP 3.  In term II, we expand out $\mu(\tilde r)$ to get
\begin{equation}
\label{E:momentum}
\text{II} = 2\lambda(t) \Im \int \partial_{\tilde r}\tilde Q_{b(t)}
\, \bar \epsilon \, \tilde r \, d\tilde r \, d\tilde z + 2r(t)
(\partial_{\tilde r}\Theta, \epsilon_1) - 2r(t)
(\partial_{\tilde r}\Sigma, \epsilon_2).
\end{equation}
The first of these is estimated away using $\lambda(t)\leq
\Gamma_{b(t)}^2$, the third term we keep, and for the second we note
$$
\| e^{\tilde R/2}\partial_{\tilde r}\Theta \|_{L^2_{(\tilde
r, \tilde z)}} \to 0 \quad \text{as} \quad b\to 0
$$
by QP 1, and thus, we can estimate $(\partial_{\tilde r}\Theta, \epsilon_1)$
by Cauchy-Schwarz.  For term III, we first estimate by Cauchy-Schwarz to get
$$
|\text{III}| \leq \left( \int |\partial_{\tilde
r}\epsilon|^2\mu(\tilde r) \, d\tilde rd\tilde z \right)^{1/2}
\left( \int |\epsilon|^2\mu(\tilde r) \, d\tilde rd\tilde z
\right)^{1/2}.
$$
For the second factor, we convert $\epsilon$ back in terms of $\tilde u$:
$$
\epsilon(\tilde r,\tilde z) = \lambda(t) \tilde u(\lambda(t)\tilde r
+ r(t), \lambda(t)\tilde z + z(t), t)
$$
to obtain
$$
\int |\epsilon|^2\mu(\tilde r) \, d\tilde rd\tilde z  = \|\tilde u\|_{L^2}^2,
$$
which is $\leq \alpha^{2/5}$ by BSO 2. Combining all of the above information, we get
\begin{equation}
\label{E:momentum-final}
r(t) (\partial_{\tilde r}\Sigma, \epsilon_2) \leq
\delta(\alpha^*)\mathcal{E}(t)^{1/2} + \Gamma_{b(t)}^2.
\end{equation}
We finish by using that
$$
\text{BSI }1: \; |r(t)-1| \leq (\alpha^*)^{1/2}\implies \frac{1}{r(t)} \leq 2
$$
and then \eqref{E:IE 2} follows.

\section{Bootstrap Step 6.  Application of orthogonality conditions}
\label{S:bs6}

\subsection{Computation of $\lambda_s/\lambda+b$}

Here we explain how to obtain \eqref{E:IE 3.1}.

Multiply the equation for $\epsilon_1$ (\ref{E:Epsilon1}) by $\tilde
R^2\Sigma$ and the equation for $\epsilon_2$ (\ref{E:Epsilon2}) by
$\tilde R^2 \Theta$ and add. We study the resulting terms one by
one.

\noindent\textit{Term 1}.
\begin{equation}
\label{E:lam1}
b_s\left[ (\partial_b\Sigma,\tilde R^2\Sigma) + (\partial_b\Theta, \tilde R^2\Theta)\right] = b_s \Re (\partial_b\tilde Q_b, \tilde R^2\tilde Q_b). \end{equation}
Using the QP properties:
$$
\left\| e^{\tilde R/2}(\partial_b\tilde Q_b -i\tfrac14\tilde R^2 Q )\right\|_{C^2} \to 0 \text{ as }b \to 0
$$
and
$$
\left\| e^{\tilde R/2}(\tilde Q_b -Q )\right\|_{C^3} \to 0 \text{ as }b \to 0
$$
to show that \eqref{E:lam1} $\sim b_s\Re (\tfrac14 i\tilde R^2 Q, \tilde R^2 Q) =0$.  More specifically, we can show that \eqref{E:lam1} is
$$
\leq |b_s| \delta(b), \quad  \text{ where }\delta(b) \to 0\text{ as }b\to 0.
$$
\noindent\textit{Term 2}.
Using ORTH 1 condition, we have
\begin{equation}
\label{E:lam2}
\begin{aligned}
\indentalign (\partial_s \epsilon_1, \tilde R^2\Sigma)
+ (\partial_s \epsilon_2,\tilde R^2 \Theta) \\
&= \partial_s\left[ (\epsilon_1, \tilde R^2\Sigma) + (\epsilon_2,\tilde R^2\Theta)\right] - b_s\left[ (\epsilon_1,\tilde R^2\partial_b\Sigma) + (\epsilon_2, \tilde R^2 \partial_b\Theta)\right]\\
&= - b_s\left[ (\epsilon_1,\tilde R^2\partial_b\Sigma) + (\epsilon_2, \tilde R^2 \partial_b\Theta)\right],
\end{aligned}
\end{equation}
which we then estimate as
$$
\leq |b_s| \left( \int_{\tilde R\leq \frac{10}{b}} |\epsilon|^2 e^{-\tilde R} d\tilde rd\tilde z \right)^{1/2} \leq |b_s|\mathcal{E}(t)^{1/2}.
$$
\noindent\textit{Term 3}.
\begin{equation}
\label{E:lam3}
\begin{aligned}
\indentalign (-M_-(\epsilon) + b\Lambda\epsilon_1,\tilde R^2\Sigma)
+ (M_+(\epsilon) + b\Lambda\epsilon_2,\tilde R^2 \Theta) \\
&=
\begin{aligned}[t]
&-(\epsilon_2, \tilde R^2\Sigma) + (\epsilon_1,\tilde R^2\Theta) \\
&+(\Delta_{(\tilde r,\tilde z)}\epsilon_2,\tilde R^2\Sigma) -(\Delta_{(\tilde r,\tilde z)}\epsilon_1,\tilde R^2\Theta) \\
&+\frac{\lambda}{\mu}(\partial_{\tilde r}\epsilon_2,\tilde R^2\Sigma) - \frac{\lambda}{\mu}(\partial_{\tilde r}\epsilon_1,\tilde R^2\Theta)\\
&+\left( \left( \frac{2\Theta^2}{|\tilde Q_b|^2} + 1\right)|\tilde Q_b|^2\epsilon_2, \tilde R^2\Sigma\right) - \left( \left( \frac{2\Sigma^2}{|\tilde Q_b|^2}+1\right)|\tilde Q_b|^2 \epsilon_1, \tilde R^2\Theta \right) \\
&+ (2\Sigma\Theta\epsilon_1,\tilde R^2\Sigma) - (2\Sigma\Theta\epsilon_2,\tilde R^2\Theta)\\
&+ b[(\Lambda\epsilon_1,\tilde R^2\Sigma)+(\Lambda\epsilon_2,\tilde R^2\Theta)]
\end{aligned}\\
&= \text{I}+\text{II}+\text{III}+\text{IV}+\text{V} +\text{VI}.
\end{aligned}
\end{equation}
By integration by parts,
$$
\text{I}+\text{II} = -(\epsilon_2,\tilde R^2\Sigma) + (\epsilon_1,\tilde R^2\Theta) + (\epsilon_2,\Delta_{(\tilde r,\tilde z)}(\tilde R^2\Sigma)) -(\epsilon_1,\Delta_{(\tilde r,\tilde z)}(\tilde R^2\Theta)).
$$
Computation gives, for any function $f$, that
$\Delta_{(\tilde r,\tilde z)}(\tilde R^2 f) = 4\Lambda f + \tilde R^2 \Delta_{(\tilde r,\tilde z)}f$. Thus, using ORTH 4,
$$
\text{I}+\text{II} = \Im (\epsilon, \tilde R^2(-\tilde Q_b + \Delta \tilde Q_b)).
$$
Substituting the equation for $\tilde Q_b$ gives
$$
\text{I}+\text{II} = \Im (\epsilon, \tilde R^2(-ib\Lambda\tilde Q_b-\tilde Q_b|\tilde Q_b|^2-\Psi_b)).
$$
By examining IV+V and rearranging terms,
$$
\text{IV}+\text{V} = \Im (\epsilon, \tilde R^2\tilde Q_b|\tilde Q_b|^2),
$$
and thus,
$$
\text{I}+\text{II}+\text{IV}+\text{V}
= \Im (\epsilon, \tilde R^2(-ib\Lambda\tilde Q_b-\Psi_b)).
$$
In fact, adding VI cancels the middle term:
$$
\text{I}+\text{II}+\text{IV}+\text{V}+\text{VI} = -\Im (\epsilon, \tilde R^2\Psi_b)).
$$
Properties \eqref{E:Psi1} and ZP 2 imply
\begin{equation}
\label{E:Psi2}
\|P(y)\Psi^{(k)}(y)\|_{L^\infty} \leq C_{P,k}\Gamma_b^{\frac12(1-C\eta)},
\end{equation}
thus, we get
$$
|(\epsilon, \tilde R^2\Psi_b)| \leq \Gamma_b^{\frac12(1-C\eta)}\left( \int_{\tilde R<10/b} |\epsilon|^2e^{-\tilde R} \, d\tilde r d\tilde z \right)^{1/2}
=\Gamma_b^{\frac12(1-C\eta)} \mathcal{E}(t)^{1/2}. $$
For III, use $\lambda\leq \Gamma_b^{10}$ and $\mu\sim 1$ and Cauchy-Schwarz to obtain
$$
|\text{III}| \leq \Gamma_b^{10}\left( \int_{\tilde R<10/b}
|\epsilon|^2e^{-\tilde R} \, d\tilde r d\tilde z \right)^{1/2}.
$$
Together we have
$$
|\eqref{E:lam3}| \leq \Gamma_b^{\frac12(1-C\eta)} \left( \int_{\tilde R<10/b} |\epsilon|^2e^{-\tilde R} d\tilde rd\tilde z\right)^{1/2}.
$$

\noindent\textit{Term 4}.
\begin{equation}
\label{E:lam4}
\left( \frac{\lambda_s}{\lambda}+b\right)[(\Lambda\Sigma,\tilde R^2\Sigma) + (\Lambda\Theta,\tilde R^2\Theta)]
\end{equation}
By integration by parts, for any function $f$, we have $(\Lambda f, \tilde R^2 f) = -\int \tilde R^2f^2$, and thus, the above expression becomes
$$
-\left( \frac{\lambda_s}{\lambda}+b\right)\|\tilde R \tilde Q_b\|_{L^2}^2.
$$

\noindent\textit{Term 5}.
\begin{equation}
\label{E:lam5}
\left( \frac{1}{\lambda}(r_s,z_s)\cdot \nabla_{(\tilde r,\tilde z)}\Sigma, \tilde R^2\Sigma\right) + \left( \frac{1}{\lambda}(r_s,z_s)\cdot \nabla_{(\tilde r,\tilde z)}\Theta, \tilde R^2\Theta\right)
\end{equation}
Each of these two terms, it turns out, is zero.  For example,
\begin{align*}
(\partial_{\tilde r}\Theta, \tilde R^2\Theta) &= \frac12\int \tilde R^2 \partial_{\tilde r}(\Theta^2) d\tilde r d\tilde z \\
&=-\frac12\int 2\tilde r \, \Theta^2 \, d\tilde r d\tilde z\\
&= 0,
\end{align*}
since it is the integral of an odd function (recall $\tilde r$ goes from $-\infty$ to $+\infty$).

\noindent\textit{Term 6}.
\begin{equation}
\label{E:lam6}
\tilde \gamma_s [(\Theta,\tilde R^2\Sigma) - (\Sigma, \tilde R^2\Theta)] =0 .
\end{equation}

\noindent\textit{Term 7}.
\begin{equation}
\label{E:lam7}
\left( \frac{\lambda_s}{\lambda}+b\right)[(\Lambda\epsilon_2,\tilde R^2\Theta) + (\Lambda\epsilon_1,\tilde R^2\Sigma)] \leq \left| \frac{\lambda_s}{\lambda}+b\right| \mathcal{E}(t)^{1/2},
\end{equation}
by Cauchy-Schwarz.

\noindent\textit{Term 8}.
\begin{equation}
\label{E:lam8} \left( \frac{(r_s,z_s)}{\lambda}\cdot \nabla_{(\tilde
r,\tilde z)} \epsilon_1, \tilde R^2\Sigma\right) + \left(
\frac{(r_s,z_s)}{\lambda}\cdot \nabla_{(\tilde r,\tilde
z)}\epsilon_2, \tilde R^2 \Theta \right)
\end{equation}
we estimate by
$$
\frac{|(r_s,z_s)|}{\lambda} \mathcal{E}(t)^{1/2},
$$
similar to {\it Term 7}.

\noindent\textit{Term 9}.
\begin{equation}
\label{E:lam9}
\tilde \gamma_s[(\epsilon_2,\tilde R^2\Sigma)
-(\epsilon_1,\tilde R^2\Theta)] \leq |\tilde \gamma_s| \mathcal{E}(t)^{1/2}.
\end{equation}

\noindent\textit{Term 10}.
\begin{equation}
\label{E:lam10}
(\Im \tilde \Psi, \tilde R^2\Sigma) - (\Re \tilde \Psi,\tilde R^2\Theta)
 = \Im (\tilde \Psi_b, \tilde R^2 \tilde Q_b).
\end{equation}
It turns out that this term is merely zero ($\leq \mathcal{E}(t)$),
which is shown by substituting the equation for $\tilde \Psi_b$ in terms
of $\tilde Q_b$ into \eqref{E:lam10} and applying the property \eqref{E:momentum-degeneracy}, i.e.,
$$
\Im \int (\tilde r,\tilde z)\cdot \nabla_{(\tilde r,\tilde z)} Q_b \; \overline{\tilde Q_b} \, d\tilde rd\tilde z + \frac{b}{2}\|\tilde R \tilde Q_b \|_{L^2}^2=0
$$
and
$$
\frac{\lambda}{\mu}(\partial_r \tilde Q_b, \tilde R^2 \tilde Q_b) \leq \Gamma^{10}_b.
$$

\noindent\textit{Term 11}.
\begin{equation}
\label{E:lam11}
(R_1(\epsilon), \tilde R^2\Theta) - (R_2(\epsilon),\tilde R^2\Sigma)
\end{equation}
Recall that $R_1(\epsilon)$ and $R_2(\epsilon)$ consist of quadratic and cubic terms.  A typical quadratic term has the form
$$
\int |\epsilon|^2 |\tilde Q_{b}|^2 \,d\tilde rd\tilde z \leq \left(\int_{\tilde R<10/b} |\epsilon|^2 e^{-\tilde R} \, d\tilde rd\tilde z \right) \|e^{+\tilde R/2} \tilde Q_b\|_{L^\infty}^2 \leq \mathcal{E}(t).
$$
For the cubic terms, we use Cauchy-Schwarz, \eqref{E:IE 1.3} and properties of $\tilde Q_b$.
%2d Gagliardo-Nirenberg inequality).
A typical term is
\begin{align*}
\int |\epsilon|^3 |\tilde Q_b| d\tilde rd\tilde z &\leq \left( \int_{\tilde R<10/b} |\epsilon|^4 d\tilde rd\tilde z\right)^{1/2} \left( \int |\epsilon|^2|\tilde Q_b|^2 \right)^{1/2}\\
& \leq \left( \int_{\tilde R<10/b} |\epsilon|^4 d\tilde rd\tilde z\right)^{1/2} \left( \int |\epsilon|^2e^{-\tilde R} \right)^{1/2} \|e^{+\tilde R/2}\tilde Q_b\|_{L^\infty}\\
&\leq \mathcal{E}(t).
\end{align*}
Collecting the above estimates on terms \eqref{E:lam1}--\eqref{E:lam11}, keeping only \eqref{E:lam4}, we obtain
\begin{equation}
\label{E:est1}
\left| \frac{\lambda_s}{\lambda}+b\right| \lesssim \delta(b)|b_s| + \mathcal{E}(t)^{1/2}\left( \left| \frac{\lambda_s}{\lambda}+b\right| + \frac{|(r_s,z_s)|}{\lambda} + |\tilde \gamma_s|\right) + \mathcal{E}(t)^{1/2}\Gamma_b^{\frac12(1-C\eta)} + \mathcal{E}(t) \,.
\end{equation}

\subsection{Computation of $(r_s, z_s)/\lambda$}
Multiply the equation for $\epsilon_1$ (\ref{E:Epsilon1}) by
$(\tilde{r}, \tilde z) \Sigma$ and the equation for $\epsilon_2$
(\ref{E:Epsilon2}) by $(\tilde{r}, \tilde z) \Theta$ and add.
Note we will have now a vectorial equation and again study each term separately.
%(Here we will use second orthogonality condition, which is vectorial relation.)

\noindent\textit{Term 1}.
\begin{equation}
\label{E:rz1}
b_s\left[ (\partial_b\Sigma, (\tilde r, \tilde z) \Sigma)
+ (\partial_b\Theta, (\tilde r, \tilde z) \Theta)\right]
= b_s \Re (\partial_b\tilde Q_b, (\tilde r, \tilde z) \tilde Q_b).
\end{equation}
Recall that $\partial_b \tilde Q_b \sim \frac{i}4 \tilde R^2 Q$, and
$\ds \left\| e^{\tilde R/2}(\tilde Q_b -Q )\right\|_{C^3} \to 0 \text{ as }b \to 0$,
hence, we have
$$
b_s \Re (\partial_b\tilde Q_b, (\tilde r, \tilde z) \tilde Q_b) \approx \Re
(\frac{i}4 \, \tilde R^2 Q, (\tilde r, \tilde z) Q) = 0,
$$
or similar to {\it Term 1} in \eqref{E:lam1},
$\ds \sim |b_s| \, \delta(b)$ with $\delta (b) \to 0$ as $b \to 0$.

\noindent\textit{Term 2}.
Using ORTH 2 condition, we have
\begin{equation}
\label{E:rz2}
\begin{aligned}
\indentalign (\partial_s \epsilon_1, (\tilde r, \tilde z) \Sigma)
+ (\partial_s \epsilon_2, (\tilde r, \tilde z) \Theta) \\
&= \partial_s\left[ (\epsilon_1, (\tilde r, \tilde z) \Sigma)
+ (\epsilon_2, (\tilde r, \tilde z) \Theta)\right] -
b_s\left[ (\epsilon_1,(\tilde r, \tilde z) \partial_b\Sigma)
+ (\epsilon_2, (\tilde r, \tilde z) \partial_b\Theta)\right]\\
&= - b_s \Re (\epsilon,  (\tilde r, \tilde z) \partial_b \tilde Q_b)\\
& \leq |b_s| \, \delta(b),
\end{aligned}
\end{equation}
since again $\partial_b \tilde Q_b \approx \frac{i}4 \tilde R^2 Q$ and
$\Re (\epsilon,  (\tilde r, \tilde z) \partial_b \tilde Q_b) \approx 0$.

\noindent\textit{Term 3}.
\begin{equation}
\label{E:rz3}
\begin{aligned}
\indentalign (-M_-(\epsilon) + b\Lambda\epsilon_1, (\tilde r, \tilde z) \Sigma)
+ (M_+(\epsilon) + b\Lambda\epsilon_2, (\tilde r, \tilde z) \Theta) \\
&=
\begin{aligned}[t]
%&-(\epsilon_2, (\tilde r, \tilde z)\Sigma) - (\epsilon_1, (\tilde r, \tilde z) \Theta) \\
%&+(\Delta_{(\tilde r,\tilde z)}\epsilon_2, (\tilde r, \tilde z) \Sigma)
%-(\Delta_{(\tilde r,\tilde z)}\epsilon_1, (\tilde r, \tilde z) \Theta) \\
%&+\frac{\lambda}{\mu}(\partial_{\tilde r}\epsilon_2,(\tilde r, \tilde z) \Sigma) -
%\frac{\lambda}{\mu}(\partial_{\tilde r}\epsilon_1, (\tilde r, \tilde z) \Theta)\\
%&+\left( \left( \frac{2\Theta^2}{|\tilde Q_b|^2} + 1\right)|\tilde Q_b|^2\epsilon_2,
%(\tilde r, \tilde z) \Sigma\right)
%- \left( \left( \frac{2\Sigma^2}{|\tilde Q_b|^2}+1\right)|\tilde Q_b|^2
%\epsilon_1, (\tilde r, \tilde z) \Theta \right) \\
%&+ (2\Sigma\Theta\epsilon_1, (\tilde r, \tilde z)\Sigma)
%- (2\Sigma\Theta\epsilon_2, (\tilde r, \tilde z) \Theta)\\
%&+ b[(\Lambda\epsilon_1, (\tilde r, \tilde z) \Sigma)
%+(\Lambda\epsilon_2, (\tilde r, \tilde z) \Theta)]\\
& - \Im (\epsilon, (\tilde r, \tilde z)\tilde Q_b)
+ \Im (\Delta \epsilon, (\tilde r, \tilde z) \tilde Q_b)
+ \frac{\lambda}{\mu} \Im (\partial_r \epsilon, (\tilde r, \tilde z) \tilde Q_b)\\
& + \Im (\epsilon, (\tilde r, \tilde z)|\tilde Q_b|^2 \tilde Q_b)
+\Im (\epsilon, (\tilde r, \tilde z) \, i b \Lambda \tilde Q_b)\\
& = \Im (\epsilon, (\tilde r, \tilde z)
\left[-\tilde Q_b + \Delta Q_b + |\tilde Q_b|^2 \tilde Q_b
+ i b \Lambda \tilde Q_b\right])\\
& + 2 \Im(\epsilon, \nabla_{(\tilde r, \tilde z)} \tilde Q_b) +\frac{\lambda}{\mu}
\Im (\partial_r \epsilon, (\tilde r, \tilde z)\tilde Q_b)\\
& = \Im \left[ -(\epsilon, (\tilde r, \tilde z) \Psi_b)
+ 2 (\epsilon, \nabla_{(\tilde r, \tilde z)} \tilde Q_b) +\frac{\lambda}{\mu}
(\partial_r \epsilon, (\tilde r, \tilde z)\tilde Q_b) \right]\\
\end{aligned}\\
&= \text{I}+\text{II}+\text{III}.
%+\text{IV}+\text{V} +\text{VI}.
\end{aligned}
\end{equation}

%& \leq \Gamma_b^{\frac12(1-C\eta)} \, \mathcal{E}(t)^{1/2} +
%(\delta(b)+\delta(\alpha^*)) \mathcal{E}(t)^{1/2} + \Gamma_b^2
%+ \Gamma_b^{10}\mathcal{E}^{1/2}\\

For term I, we use the estimate \eqref{E:Psi2} for $\Psi_b$
%$$
%\|P(y)\Psi^{(k)}(y)\|_{L^\infty} \leq C_{P,k}\Gamma_b^{\frac12(1-C\eta)}
%$$
to get
$$
|(\epsilon, (\tilde r, \tilde z) \Psi_b)|
\leq %\Gamma_b^{\frac12(1-C\eta)} \left( \int_{\tilde R<10/b} |\epsilon|^2e^{-\tilde R} \, d\tilde r d\tilde z \right)^{1/2}=
\Gamma_b^{\frac12(1-C\eta)} \mathcal{E}(t)^{1/2}.
$$
The term II (for example, the first coordinate of the vector with $\partial_ r \tilde Q_b$)
we write as
$$
2 [(\epsilon_2, \partial_r \Sigma) - (\epsilon_1, \partial_r \Theta)]
+ (\epsilon_2, \tilde r \Delta \Sigma) -(\epsilon_1, \tilde r \Delta \Theta)
$$
and note that $\partial_r \Theta \approx 0$ as $b \to 0$, thus,
$|(\epsilon_1, \partial_r \Theta)| \leq \delta(b) \mathcal{E}(t)^{1/2}$,
and the term
$|(\epsilon_2, \partial_r \Sigma)| \leq \delta(\alpha^*) \mathcal{E}(t)^{1/2} + \Gamma_b^2$,
which we estimated as in %\eqref{E:momentum} by
\eqref{E:momentum-final}.
The last two terms are estimated out by Cauchy-Schwarz and properties QP.

For III, use $\lambda\leq \Gamma_b^{10}$ and $\mu\sim 1$ and Cauchy-Schwarz to obtain
$\ds |\text{III}| \leq \Gamma_b^{10} \mathcal{E}(t)^{1/2}$.
%\left( \int_{\tilde R<10/b} |\epsilon|^2e^{-\tilde R} \, d\tilde r d\tilde z \right)^{1/2}.
Together we have
$$
|\eqref{E:lam3}| \leq \Gamma_b^{\frac12(1-C\eta)} \mathcal{E}(t)^{1/2}
+ (\delta(b)+ \delta(\alpha^*))\mathcal{E}(t)^{1/2}
+ \Gamma_b^2 + \Gamma^{10}_b \mathcal{E}(t)^{1/2}.
$$

\noindent\textit{Term 4}.
\begin{equation}
\label{E:rz4}
\left( \frac{\lambda_s}{\lambda}+b\right)[(\Lambda\Sigma,(\tilde r, \tilde z)\Sigma) +
(\Lambda\Theta,(\tilde r, \tilde z)\Theta)]
 = \left( \frac{\lambda_s}{\lambda}+b\right) \Re (\Lambda \tilde Q_b,
(\tilde r, \tilde z) \tilde Q_b) = 0.
\end{equation}

\noindent\textit{Term 5}.
\begin{equation}
\begin{aligned}
&\left( \frac{1}{\lambda}(r_s,z_s)\cdot \nabla_{(\tilde r,\tilde z)}\Sigma,
(\tilde r,\tilde z)\Sigma\right)
+ \left( \frac{1}{\lambda}(r_s,z_s)\cdot \nabla_{(\tilde r,\tilde
z)}\Theta, (\tilde r,\tilde z) \Theta\right) \\
& = \Re \left(\frac{1}{\lambda}(r_s,z_s)\cdot \nabla_{(\tilde r,\tilde z)}\tilde Q_b,
(\tilde r,\tilde z)\tilde Q_b\right)
= - \frac{(r_s,z_s)}{\lambda} \, \| \tilde Q_b \|^2_{L^2{(\tilde r, \tilde z)}}.\\
\label{E:rz5}
\end{aligned}
\end{equation}

\noindent\textit{Term 6}.
\begin{equation}
\label{E:rz6}
\tilde \gamma_s [(\Theta,(\tilde r, \tilde z)\Sigma)
- (\Sigma, (\tilde r, \tilde z) \Theta)] =0 .
\end{equation}

\noindent\textit{Term 7}.
\begin{equation}
\label{E:rz7}
\left( \frac{\lambda_s}{\lambda}+b\right)[(\Lambda\epsilon_2, (\tilde r, \tilde z) \Theta) +
(\Lambda\epsilon_1, (\tilde r, \tilde z) \Sigma)]
\leq \left| \frac{\lambda_s}{\lambda}+b\right|
\mathcal{E}(t)^{1/2},
\end{equation}
applying ORTH 2 and by Cauchy-Schwarz.

\noindent\textit{Term 8}.
\begin{equation}
\label{E:rz8}
\left( \frac{(r_s,z_s)}{\lambda}\cdot \nabla_{(\tilde
r,\tilde z)} \epsilon_1, (\tilde r, \tilde z) \Sigma\right) + \left(
\frac{(r_s,z_s)}{\lambda}\cdot \nabla_{(\tilde r,\tilde
z)}\epsilon_2, (\tilde r, \tilde z) \Theta \right)
\leq
\frac{|(r_s,z_s)|}{\lambda} \mathcal{E}(t)^{1/2},
\end{equation}
similar to {\it Term 7}.

\noindent\textit{Term 9}.
\begin{equation}
\label{E:rz9}
\tilde \gamma_s[(\epsilon_2, (\tilde r, \tilde z)\Sigma)
- (\epsilon_1, (\tilde r, \tilde z)\Theta)]
\leq |\tilde \gamma_s| \mathcal{E}(t)^{1/2}.
\end{equation}

\noindent\textit{Term 10}.
\begin{equation}
\label{E:rz10}
(\Im \tilde \Psi, (\tilde r, \tilde z)\Sigma)
- (\Re \tilde \Psi, (\tilde r, \tilde z) \Theta)
 = \Im (\tilde \Psi_b, (\tilde r, \tilde z) \tilde Q_b).
\end{equation}
%This term is zero, which is shown by
Substituting \eqref{E:Qtilde} for $\tilde \Psi_b$ in terms of $\tilde Q_b$,
one can see that all terms are zero (integration by parts or by degeneracy of the momentum QP 3, first property in \eqref{E:momentum-degeneracy})
except for $\frac{\lambda}{\mu} \Im (\partial_r \tilde Q_b, (\tilde r, \tilde z) \tilde Q_b)$
which is bounded by $\Gamma_b^{10}$ since $\lambda < \Gamma_b^{10}$ (and localization of $\tilde Q_b$ implies $\mu \sim 1$ and boundedness of the inner product).

\noindent\textit{Term 11}.
\begin{equation}
\label{E:rz11}
(R_1(\epsilon), (\tilde r, \tilde z) \Theta) - (R_2(\epsilon), (\tilde r, \tilde z) \Sigma)
\leq \mathcal{E}(t),
\end{equation}
in the same fashion as {\it Term 11} in \eqref{E:lam11}.

Collecting the above estimates on terms \eqref{E:rz1}--\eqref{E:rz11}, keeping only
\eqref{E:rz5}, we obtain
\begin{equation}
\label{E:est2}
\frac{|(r_s,z_s)|}{\lambda} \|\tilde Q_b\|^2_{L^2} \leq
\delta(b)|b_s| +
\mathcal{E}(t)^{1/2}\left( \left| \frac{\lambda_s}{\lambda}+b\right| +
\frac{|(r_s,z_s)|}{\lambda} + |\tilde \gamma_s|\right) +
\mathcal{E}(t)^{1/2}\Gamma_b^{\frac12(1-C\eta)} + \mathcal{E}(t) \,.
\end{equation}

\subsection{Computation of $b_s$}
Multiply the equation for $\epsilon_1$ (\ref{E:Epsilon1}) by $-
\Lambda \Theta$ and the equation for $\epsilon_2$ (\ref{E:Epsilon2})
by $\Lambda \Sigma$ and add.
%(Here we will use fourth orthogonality condition.)

\noindent\textit{Term 1}.
\begin{equation}
\label{E:bs1}
b_s\left[ (\partial_b \Theta, \Lambda \Sigma) -
(\partial_b \Sigma, \Lambda \Theta)\right]
= b_s \Im (\partial_b\tilde Q_b, \Lambda \tilde Q_b).
\end{equation}
Recalling that $\partial_b \tilde Q_b \sim \frac{i}4 \tilde R^2 Q$ and
$\Lambda \tilde Q_b = \tilde Q_b + \tilde R \partial_R \tilde Q_b$, we estimate
\eqref{E:bs1} as
$$
b_s \Im (\partial_b\tilde Q_b, \Lambda \tilde Q_b) \approx
\frac14 \|\tilde R \, Q \|^2_{L^2{(\tilde r, \tilde z)}}
+ \frac14 \int \tilde R^3 Q \, \partial_R \tilde Q_b) \approx |b_s| \, \|\tilde R \, Q\|^2_{L^2{(\tilde r, \tilde z)}}.
$$

\noindent\textit{Term 2}.
Using ORTH 4 condition, we have
\begin{equation}
\label{E:bs2}
\begin{aligned}
\indentalign (\partial_s \epsilon_2, \Lambda \Sigma)
- (\partial_s \epsilon_1, \Lambda \Theta) \\
&= \partial_s\left[ (\epsilon_2, \Lambda \Sigma)
+ (\epsilon_1, \Lambda \Theta)\right] +
b_s \left[ (\epsilon_1, \Lambda \partial_b\Sigma)
- (\epsilon_2, \Lambda \partial_b\Theta)\right]\\
&= - b_s \Im (\epsilon,  \Lambda \partial_b \tilde Q_b)\\
& \leq |b_s| \, \mathcal{E}(t)^{1/2},
\end{aligned}
\end{equation}
by Cauchy-Schwarz and using properties of $\tilde Q_b$ (e.g., see (95) in \cite{MR-GAFA}).

\noindent\textit{Term 3}.
\begin{equation}
\label{E:bs3}
\begin{aligned}
\indentalign (M_+(\epsilon) + b\Lambda\epsilon_2, \Lambda \Sigma)
+ (M_-(\epsilon) - b\Lambda\epsilon_1, \Lambda \Theta) \\
&=
\begin{aligned}[t]
%&-(\epsilon_2, \Lambda\Sigma) - (\epsilon_1, \Lambda \Theta) \\
%&+(\Delta_{(\tilde r,\tilde z)}\epsilon_2, \Lambda \Sigma)
%-(\Delta_{(\tilde r,\tilde z)}\epsilon_1, \Lambda \Theta) \\
%&+\frac{\lambda}{\mu}(\partial_{\tilde r}\epsilon_2,\Lambda \Sigma) -
%\frac{\lambda}{\mu}(\partial_{\tilde r}\epsilon_1, \Lambda \Theta)\\
%&+\left( \left( \frac{2\Theta^2}{|\tilde Q_b|^2} + 1\right)|\tilde Q_b|^2\epsilon_2,
%\Lambda \Sigma\right)
%- \left( \left( \frac{2\Sigma^2}{|\tilde Q_b|^2}+1\right)|\tilde Q_b|^2
%\epsilon_1, \Lambda \Theta \right) \\
%&+ (2\Sigma\Theta\epsilon_1, \Lambda\Sigma)
%- (2\Sigma\Theta\epsilon_2, \Lambda \Theta)\\
%&+ b[(\Lambda\epsilon_1, \Lambda \Sigma)
%+(\Lambda\epsilon_2, \Lambda \Theta)]\\
& \Re (\epsilon, \Lambda\tilde Q_b)
- \Re (\Delta \epsilon, \Lambda \tilde Q_b)
- \frac{\lambda}{\mu} \Re (\partial_r \epsilon, \Lambda \tilde Q_b)\\
& - \Re (\epsilon, |\tilde Q_b|^2 \Lambda \tilde Q_b + 2 |\tilde Q_b|
(\Sigma \Lambda \Sigma + \Theta \Lambda \Theta))\\
& + b[(\Lambda \epsilon_2, \Lambda \Sigma)-(\Lambda \epsilon_1, \Lambda \Theta)]
%-\Re i b (\Lambda \epsilon, \Lambda \tilde Q_b)\\
\end{aligned}\\
&= \text{I}+\text{II}+\text{III}+\text{IV}+\text{V}. % +\text{VI}.
\end{aligned}
\end{equation}

%& = \Im (\epsilon, \Lambda
%\left[-\tilde Q_b + \Delta Q_b + |\tilde Q_b|^2 \tilde Q_b
%+ i b \Lambda \tilde Q_b\right])\\
%& + 2 \Im(\epsilon, \nabla_{(\tilde r, \tilde z)} \tilde Q_b) +\frac{\lambda}{\mu}
%\Im (\partial_r \epsilon, \Lambda\tilde Q_b)\\
%& = \Im \left[ -(\epsilon, \Lambda \Psi_b)
%+ 2 (\epsilon, \nabla_{(\tilde r, \tilde z)} \tilde Q_b) +\frac{\lambda}{\mu}
%(\partial_r \epsilon, \Lambda\tilde Q_b) \right]\\

%& \leq \Gamma_b^{\frac12(1-C\eta)} \, \mathcal{E}(t)^{1/2} +
%(\delta(b)+\delta(\alpha^*)) \mathcal{E}(t)^{1/2} + \Gamma_b^2
%+ \Gamma_b^{10}\mathcal{E}^{1/2}\\

For term III, use $\lambda\leq \Gamma_b^{10}$, localization of $\tilde Q_b$, thus, $\mu\sim 1$, and Cauchy-Schwarz to obtain
$\ds |\text{III}| \leq \Gamma_b^{10} \mathcal{E}(t)^{1/2}$.
For term IV we use ORTH 4 to obtain
$$
\text{IV} = b [((\tilde r, \tilde z)\cdot \nabla \epsilon_2, \Lambda \Sigma)-((\tilde r, \tilde z)\cdot \nabla \epsilon_1, \Lambda \Theta)] = \Re \left(- i \, b ((\tilde r, \tilde z)\cdot \nabla \epsilon, \Lambda\tilde Q_b) \right).
$$
For terms I, II and IV we use QP 5 scaling invariance property \eqref{E:eqLambdaQb}:
$$
\text{I}+\text{II}+\text{IV} = - \Re(\epsilon, 2 (\tilde Q_b -i b \Lambda \tilde Q_b - \Psi_b) - \Lambda \Psi_b - ib \Lambda^2 \tilde Q_b).
$$
Terms with $\Psi_b$ we estimate by \eqref{E:Psi2}, terms with $ib$ we combine with the term IV above
and estimate using BSI 2 (smallness of $b$), Cauchy-Schwarz and localization properties of $\tilde Q_b$, thus, Terms 3 is bounded by $\Gamma_b^{\frac12(1-C\eta)} \mathcal{E}(t)^{1/2}$.
\smallskip

\noindent\textit{Term 4}.
\begin{equation}
\label{E:bs4}
\left( \frac{\lambda_s}{\lambda}+b\right)[(\Lambda\Theta,\Lambda\Sigma) -
(\Lambda\Sigma,\Lambda\Theta)]= 0.
\end{equation}

\noindent\textit{Term 5}.
\begin{equation}
\begin{aligned}
&\left( \frac{1}{\lambda}(r_s,z_s)\cdot \nabla_{(\tilde r,\tilde z)}\Theta,
\Lambda \Sigma\right)
- \left( \frac{1}{\lambda}(r_s,z_s)\cdot \nabla_{(\tilde r,\tilde
z)}\Sigma, \Lambda \Theta\right) \\
& = \Im \left(\frac{1}{\lambda}(r_s,z_s)\cdot \nabla_{(\tilde r,\tilde z)}\tilde Q_b,
\Lambda \tilde Q_b\right)
= 0,\\
\label{E:bs5}
\end{aligned}
\end{equation}
by the degeneracy of the momentum QP 3.

\noindent\textit{Term 6}.
\begin{equation}
\label{E:bs6}
- \tilde \gamma_s [(\Sigma,\Lambda\Sigma)
+ (\Theta, \Lambda \Theta)] = - \tilde \gamma_s \Re (\tilde Q_b, \Lambda \tilde Q_b)
=0 .
\end{equation}

\noindent\textit{Term 7}.
\begin{equation}
\label{E:bs7}
\left( \frac{\lambda_s}{\lambda}+b\right)[(\Lambda\epsilon_2, \Lambda \Sigma) -
(\Lambda\epsilon_1, \Lambda \Theta)]
= \left( \frac{\lambda_s}{\lambda}+b\right) \Im (\Lambda \epsilon, \Lambda \tilde Q_b)
\leq \left| \frac{\lambda_s}{\lambda}+b\right|\, \mathcal{E}(t)^{1/2},
\end{equation}
by Cauchy-Schwarz and closeness of $\tilde Q_b$ to $Q$ and properties of $Q$ (as in {\it Term 2} above).

\noindent\textit{Term 8}.
\begin{equation}
\label{E:bs8}
\begin{aligned}
&\left( \frac{(r_s,z_s)}{\lambda}\cdot \nabla_{(\tilde
r,\tilde z)} \epsilon_2, \Lambda \Sigma\right) - \left(
\frac{(r_s,z_s)}{\lambda}\cdot \nabla_{(\tilde r,\tilde
z)}\epsilon_1, \Lambda \Theta \right)\\
&= \Im \left( \frac{(r_s,z_s)}{\lambda}\cdot \nabla_{(\tilde
r,\tilde z)} \epsilon, \Lambda \tilde Q_b \right)
\leq \frac{|(r_s,z_s)|}{\lambda} \mathcal{E}(t)^{1/2}.\\
\end{aligned}
\end{equation}

\noindent\textit{Term 9}.
\begin{equation}
\label{E:bs9}
- \tilde \gamma_s[(\epsilon_1, \Lambda\Sigma)
+ (\epsilon_2, \Lambda\Theta)]
\leq |\tilde \gamma_s| \, \mathcal{E}(t)^{1/2}.
\end{equation}

\noindent\textit{Term 10}.
\begin{equation}
\label{E:bs10}
(\Re \tilde \Psi, \Lambda\Sigma)
+ (\Im \tilde \Psi, \Lambda \Theta)
 = \Re (\tilde \Psi_b, \Lambda \tilde Q_b),
\end{equation}
which is estimated by $\delta(\alpha^*)$, since $|\tilde \Psi_b, \Lambda \tilde Q_b| \leq e^{-C/|b|}$, see Lemma 4 in \cite{MR-GAFA}.

\noindent\textit{Term 11}.
\begin{equation}
\label{E:bs11}
(R_1(\epsilon), \Lambda \Theta) - (R_2(\epsilon), \Lambda \Sigma)
\leq \mathcal{E}(t),
\end{equation}
estimating quadratic and cubic in $\epsilon$ terms similar to {\it Term 11} in \eqref{E:lam11}.

Collecting the above estimates on terms \eqref{E:bs1}--\eqref{E:bs11}, keeping only
\eqref{E:bs1}, we obtain
\begin{equation}
\label{E:est3}
|b_s| \, \|\tilde R Q\|^2_{L^2} \leq
\mathcal{E}(t)^{1/2}\left( \left| \frac{\lambda_s}{\lambda}+b\right| +
\frac{|(r_s,z_s)|}{\lambda} + |\tilde \gamma_s|+|b_s|\right) +
\mathcal{E}(t)^{1/2}\Gamma_b^{\frac12(1-C\eta)} + \mathcal{E}(t) \,.
\end{equation}

\subsection{Computation of $\tilde \gamma_s$}
Multiply the equation for $\epsilon_1$ (\ref{E:Epsilon1}) by
$\Lambda^2 \Theta$ and the equation for $\epsilon_2$ (\ref{E:Epsilon2}) by
$- \Lambda^2 \Sigma$ and add.
%(Here we use third orthogonality condition.)

\noindent\textit{Term 1}.
\begin{equation}
\label{E:gs1}
b_s\left[ (\partial_b \Sigma, \Lambda^2 \Theta)
- (\partial_b \Theta, \Lambda^2 \Sigma) \right]
= - b_s \Im (\partial_b\tilde Q_b, \Lambda^2 \tilde Q_b) \leq \delta(\alpha^*)
\end{equation}
by the properties of $\tilde Q_b$ (see the last estimate in Lemma 4 in \cite{MR-GAFA}).

\noindent\textit{Term 2}.
Using ORTH 3 condition, we have
\begin{equation}
\label{E:gs2}
\begin{aligned}
\indentalign (\partial_s \epsilon_1, \Lambda^2 \Theta)
- (\partial_s \epsilon_2, \Lambda^2 \Sigma) \\
&= \partial_s \left[ (\epsilon_1, \Lambda^2 \Theta)
- (\epsilon_2, \Lambda^2 \Sigma) \right] -
b_s \left[ (\epsilon_1, \Lambda^2 \partial_b \Theta)
- (\epsilon_2, \Lambda^2 \partial_b \Sigma) \right]\\
&= - b_s \Im (\epsilon,  \Lambda^2 \partial_b \tilde Q_b)\\
& \leq |b_s| \, \mathcal{E}(t)^{1/2}
\end{aligned}
\end{equation}
by the estimate $|(\epsilon, P(R) \frac{d^m}{dR^m} \partial_b \tilde Q_b (R)| \leq \mathcal{E}(t)^{1/2}$, $0 \leq m \leq 2$, from Lemma 4 in \cite{MR-GAFA}.

\noindent\textit{Term 3}.
\begin{equation}
\label{E:gs3}
\begin{aligned}
\indentalign (-M_-(\epsilon) + b\Lambda\epsilon_1, \Lambda^2 \Theta)
- (M_+(\epsilon) + b\Lambda\epsilon_2, \Lambda^2 \Sigma) \\
&=
\begin{aligned}[t]
&-(\epsilon_1, \Lambda^2\Sigma) - (\epsilon_2, \Lambda^2 \Theta) \\
&+(\Delta_{(\tilde r,\tilde z)}\epsilon_1, \Lambda^2 \Sigma)
+(\Delta_{(\tilde r,\tilde z)}\epsilon_2, \Lambda^2 \Theta) \\
&+\frac{\lambda}{\mu}(\partial_{\tilde r}\epsilon_1,\Lambda^2 \Sigma)
+\frac{\lambda}{\mu}(\partial_{\tilde r}\epsilon_2, \Lambda^2 \Theta)\\
&+ \left((2\Theta^2 + |\tilde Q_b|^2)\epsilon_2, \Lambda^2 \Theta\right)
+ \left((2\Sigma^2+|\tilde Q_b|^2)\epsilon_1, \Lambda^2 \Sigma \right) \\
&+ (2\Sigma\Theta\epsilon_2, \Lambda^2\Sigma)
+ (2\Sigma\Theta\epsilon_1, \Lambda^2 \Theta)\\
&+ b[(\Lambda\epsilon_2, \Lambda^2 \Sigma)
-(\Lambda\epsilon_1, \Lambda^2 \Theta)]\\
\end{aligned}\\
&= \text{I}+\text{II}+\text{III} +\text{IV}+\text{V} +\text{VI}.
\end{aligned}
\end{equation}

For terms III and VI we use the smallness of $\lambda$ and $b$, localization of $\tilde Q_b$ and Cauchy-Schwarz to estimate them by $\Gamma^{10}_b \, \mathcal{E}(t)^{1/2}$.

In terms I, II, IV and V we collect separately terms containing $\epsilon_1$ and $\epsilon_2$.
Recall the closeness of $\tilde Q_b$ to $Q$ and that $Q$ is real, thus, for example, $\partial_r \Theta \approx 0$ as $b \to 0$ (recall QP 1), and hence, the terms containing $\partial_r \Theta$ will be on the order of $\delta(b)$ - these are the terms containing $\epsilon_2$. The terms with $\epsilon_1$ produce
$$
(\epsilon_1, -\Lambda^2 \Sigma + \Delta (\Lambda^2 \tilde Q_b) + \Lambda^2\Sigma
(|\tilde Q_b|^2+2\Sigma^2)+2\Sigma\Theta \Lambda^2 \Theta) = (\epsilon_1, L_+(\Lambda^2 Q)) + \delta(b),
$$
where $L_+ = -\Delta +1-3Q^2$, by property QP 1. % and recalling that $Q$ is real.
%Together we have
%$$
%|\eqref{E:gs3}| \leq (\epsilon_1, L_+(\Lambda^2 Q)) + \delta(b)
%\Gamma_b^{\frac12(1-C\eta)} \mathcal{E}(t)^{1/2}
%+ (\delta(b)+ \delta(\alpha^*))\mathcal{E}(t)^{1/2}
%+ \Gamma_b^2 + \Gamma^{10}_b \mathcal{E}(t)^{1/2}.
%$$

\noindent\textit{Term 4}.
\begin{equation}
\label{E:gs4}
\begin{aligned}
&\left( \frac{\lambda_s}{\lambda}+b\right)[(\Lambda\Sigma,\Lambda^2\Theta) -
(\Lambda\Theta,\Lambda^2\Sigma)]\\
& = 2 \left( \frac{\lambda_s}{\lambda}+b\right) (\Lambda \Sigma,
(\tilde r, \tilde z)\cdot \nabla (\Lambda \Theta)) \\
&\leq \left| \frac{\lambda_s}{\lambda}+b\right| \, \delta(b),
\end{aligned}
\end{equation}
again by QP 1, closeness of $\tilde Q_b$ to $Q$
(e.g., the terms such as $\partial_r \Theta \approx 0$ as $b \to 0$).

\noindent\textit{Term 5}.
\begin{equation}
\left( \frac{1}{\lambda}(r_s,z_s)\cdot \nabla_{(\tilde r,\tilde z)}\Sigma,
\Lambda^2 \Theta \right)
- \left( \frac{1}{\lambda}(r_s,z_s)\cdot \nabla_{(\tilde r,\tilde
z)}\Theta, \Lambda^2 \Sigma \right)
\leq \frac{|(r_s,z_s)|}{\lambda} \, \delta(b),
\label{E:gs5}
\end{equation}
by QP 1 similar to {\it Term 4}.
%, since all the terms contain $\partial_{r} \Theta$ or $\partial_z \Theta$
%which tend to zero as $b \to 0$.

\noindent\textit{Term 6}.
\begin{equation}
\label{E:gs6}
\tilde \gamma_s [(\Sigma, \Lambda^2\Sigma)
+ (\Theta, \Lambda^2 \Theta)] = \tilde \gamma_s (\|\Lambda Q\|_{L^2}^2 + \delta(b)),
\end{equation}
by integration by parts and closeness of $\tilde Q_b$ to $Q$.

\noindent\textit{Term 7}.
\begin{equation}
\label{E:gs7}
\left( \frac{\lambda_s}{\lambda}+b\right)[(\Lambda\epsilon_1, \Lambda^2 \Theta) - (\Lambda\epsilon_2, \Lambda^2 \Sigma)]
\leq \left| \frac{\lambda_s}{\lambda}+b\right|
\mathcal{E}(t)^{1/2},
\end{equation}
by Cauchy-Schwarz and properties of $\tilde Q_b$.

\noindent\textit{Term 8}.
\begin{equation}
\label{E:gs8}
\left( \frac{(r_s,z_s)}{\lambda}\cdot \nabla_{(\tilde
r,\tilde z)} \epsilon_1, \Lambda^2 \Theta \right) - \left(
\frac{(r_s,z_s)}{\lambda}\cdot \nabla_{(\tilde r,\tilde
z)}\epsilon_2, \Lambda^2 \Sigma \right)
\leq
\frac{|(r_s,z_s)|}{\lambda} \mathcal{E}(t)^{1/2},
\end{equation}
similar to {\it Term 7}.

\noindent\textit{Term 9}.
\begin{equation}
\label{E:gs9}
\tilde \gamma_s[(\epsilon_2, \Lambda^2 \Theta)
+ (\epsilon_1, \Lambda^2 \Sigma)] \leq |\gamma_s| \, \mathcal{E}(t)^{1/2}.
%\leq |\tilde \gamma_s| \mathcal{E}(t)^{1/2}.
\end{equation}

\noindent\textit{Term 10}.
\begin{equation}
\label{E:gs10}
(\Im \tilde \Psi_b, \Lambda^2\Theta)
+ (\Re \tilde \Psi_b, \Lambda^2 \Sigma)
 = \Re (\tilde \Psi_b, \Lambda^2 \tilde Q_b) \leq \delta(\alpha^*).
\end{equation}

\noindent\textit{Term 11}.
\begin{equation}
\label{E:gs11}
- (R_1(\epsilon), \Lambda^2 \Sigma) - (R_2(\epsilon), \Lambda^2 \Theta)
\leq \mathcal{E}(t),
\end{equation}
in the same fashion as {\it Term 11} in \eqref{E:lam11}.

Collecting the above estimates on terms \eqref{E:gs1}--\eqref{E:gs11}, keeping only
\eqref{E:gs6} and the estimate for \eqref{E:gs3}, we obtain
\begin{equation}
\label{E:est4}
\begin{aligned}
&\left|\tilde \gamma_s - \frac1{\|\Lambda Q\|^2_{L^2}} (\epsilon_1, L_+(\Lambda^2 Q))\right|\\
&\leq
\mathcal{E}(t)^{1/2}\left(|b_s|+ \left| \frac{\lambda_s}{\lambda}+b\right| +
\frac{|(r_s,z_s)|}{\lambda} + |\tilde \gamma_s|\right) +
\mathcal{E}(t)^{1/2}\Gamma_b^{\frac12(1-C\eta)} + \mathcal{E}(t) \,.\\
\end{aligned}
\end{equation}

We finish this section by observing that solving the system of equations \eqref{E:est1}-\eqref{E:est2}-\eqref{E:est3}-\eqref{E:est4} for parameters
$(b_s,  \frac{\lambda_s}{\lambda}+b,
\frac{(r_s,z_s)}{\lambda} , \tilde \gamma_s)$ gives \eqref{E:IE 3.1} and \eqref{E:IE 3.2}.

\section{Bootstrap Step 7. Deduction of BSO 4 from \eqref{E:IE 3.1}}
\label{S:bs7}

From BSI 3: $\mathcal{E}(t) \leq \Gamma_{b(t)}^{3/4}$,
we deduce from \eqref{E:IE 3.1} the estimate
$$
\left| \frac{\lambda_s}{\lambda} + b\right| + |b_s| \leq \Gamma_{b(t)}^{1/2}.
$$
Direct computation gives
\begin{align}
\frac{d}{ds}(\lambda^2e^{5\pi/b}) &= 2\lambda^2e^{5\pi/b}\left( \frac{\lambda_s}{\lambda}
- \frac{5\pi b_s}{2b^2}\right) \notag\\
&=  2\lambda^2e^{5\pi/b}\left(\frac{\lambda_s}{\lambda} +b
-b- \frac{5\pi b_s}{2b^2}\right) .
\label{E:derivcomp}
\end{align}
Using that
$$
\left| \frac{\lambda_s}{\lambda}+b\right| \leq \Gamma_{b(t)}^{1/2}
$$
and (from ZP 2)
$$
\left| \frac{5\pi b_s}{2b^2} \right| \leq \frac{\Gamma_{b(t)}^{1/2}}{b(t)^2} \leq e^{-\pi/4b},
$$
we get that \eqref{E:derivcomp} implies
$$
\frac{d}{ds}( \lambda^2e^{5\pi/b}) \leq -\lambda^2be^{5\pi/b}\leq 0
$$
and, integrating in $s$, we get
$$
\lambda^2(t)e^{5\pi/b(t)} \leq \lambda^2(0)e^{5\pi/b(0)}.
$$
Since ZP 2: $e^{-5\pi /(4\,b(t))} \leq \Gamma_{b(t)} \leq e^{-3\pi
/(4\,b(t))}$, we have
$$
\frac{\lambda^2(t)|E_0|}{\Gamma_{b(t)}^4} \leq \lambda^2(t)|E_0|e^{5\pi /b(t)}
\leq \lambda^2(0)|E_0| e^{5\pi/b(0)} \leq \lambda^2(0)|E_0| \Gamma_{b(t)}^{-20/3}\leq 1
$$
by IDA 5, which gives BSO 4.

\section{Bootstrap Step 8.  The local virial inequality}
\label{S:bs8}

Multiply the $\epsilon_1$ equation \eqref{E:Epsilon2} by $-\Lambda
\Theta$, multiply the $\epsilon_2$ equation by $\Lambda\Sigma$, and
add.  We analyze the resulting terms one at a time.
$$
\text{I} = b_s[ (-\partial_b\Sigma,\Lambda\Theta) + (\partial_b\Theta,\Lambda\Sigma)] .
$$
Since, for $b$ small, $\partial_b\tilde Q_b \sim \frac{i\tilde
R^2}{4}Q$, we have $\partial_b\Sigma\approx 0$ and $\partial_b\Theta
\approx \frac{R^2}{4}Q$, and thus, the above is essentially
$$
\frac{1}{4}b_s(\tilde R^2Q,\Lambda Q) = -b_s\|\tilde RQ\|_{L^2}^2
$$
by integration by parts.

The next term is
\begin{align*}
\text{II} &= -(\partial_s\epsilon_1,\Lambda\Theta) + (\partial_s\epsilon_2,\Lambda\Sigma)\\
&= \partial_s[-(\epsilon_1,\Lambda\Theta)+(\epsilon_2,\Lambda\Sigma)] +b_s[-(\epsilon_1,\Lambda\partial_b\Theta)-(\epsilon_2,\Lambda\partial_b\Sigma)]
\end{align*}
which, by ORTH 4 is
$$
=b_s[-(\epsilon_1,\Lambda\partial_b\Theta)-(\epsilon_2,\Lambda\partial_b\Sigma)]
\leq |b_s|\mathcal{E}(t)^{1/2} .
$$

We continue, following the proof of \cite[Lemma 8]{MR-Invent} or \cite[Prop. 3]{MR-GAFA}.
Several terms come from $(M_-(\epsilon),\Lambda \Theta) + (M_+(\epsilon), \Lambda \Sigma)$
by taking adjoints using QP 5.

Terms requiring special attention in our case are:
\begin{enumerate}
\item The local linear and quadratic in $\epsilon$ terms
$$
-2(\epsilon_1, \Sigma) - 2(\epsilon_2,\Theta) - 3\int Q^2\epsilon_1^2 - \int Q^2\epsilon_2^2 .
$$
These are related to the nonlocal quadratic-in-$\epsilon$ term
$\int |\nabla_{\tilde r, \tilde z} \epsilon|^2\mu(\tilde r) \, d\tilde r \, d\tilde z$,
but with favorable sign,
via energy conservation \eqref{E:IE 1.2}.  There remain quadratic terms
$$
\int \tilde r Q^3 \partial_{\tilde r} Q \epsilon_1^2
+ \int \tilde r Q^3 \partial_{\tilde r} Q \epsilon_1^2
$$
that are combined with $\int |\nabla_{\tilde r, \tilde z} \epsilon|^2\mu(\tilde r)
\, d\tilde r \, d\tilde z$ to form $H(\epsilon,\epsilon)$, to which the spectral hypothesis
can be applied. The $H(\epsilon, \epsilon)$ term is nonlocal, and the spectral hypothesis
can only be applied locally, so %we must address
this term is addressed with cutoffs as in Raphael \cite{R}.

\item
Terms in which $\frac{\lambda}{r}\partial_r \epsilon$ and
$\frac{\lambda}{r}\partial_r\tilde Q_b$ are paired with $\tilde Q_b$. These are easily
estimated due to the $\lambda$ factor and the localization to the singular ring.

\item
$-(\epsilon_1, \Re(\Lambda \Psi)) - (\epsilon_2, \Im (\Lambda \Psi))$.
These terms are local and thus treated as in \cite{MR-Invent}.
\end{enumerate}

\section{Bootstrap Step 9.  Lower bound on $b(s)$}
\label{S:bs9}

Using \eqref{E:IE 4} (the local virial inequality), we obtain
\begin{align*}
\frac{d}{ds}( e^{\frac{3\pi}{4b}} ) &= -\frac{3\pi}{4b^2}e^{3\pi/4b}b_s \\
&\leq -\frac{3\pi}{4b^2}e^{3\pi/4b}(\delta_0\mathcal{E}(t) - \Gamma_{b}^{1-C\eta}) \\
&\leq \frac{3\pi}{4b^2}e^{3\pi/4b}\Gamma_b^{1-C\eta}\\
&\leq 1
\end{align*}
by ZP 2 in the last step.  We now integrate this to obtain
$$
e^{3\pi/4b(s)}\leq e^{3\pi/4b(s_0)} + s-s_0 =s
$$
by the definition of $s_0$ in \eqref{E:s0}.
Thus, \eqref{E:IE 5} follows. Now from the scaling law\footnote{The constant $\frac59$ is not very important, since $e^x/x \gg 1$ for large $x$;
it is chosen for convenience of calculations, e.g., in \cite{R} it is chosen to be $\frac13$ .}
$$
\left| \frac{\lambda_s}{\lambda} + b\right| \leq \Gamma_b^{1/2}
\leq e^{-\frac{\pi}{4b(s)}} \leq \frac59 \, b
$$
$$
\implies -\frac{14}9\,b \leq \frac{\lambda_s}{\lambda} \leq - \frac49 \,b %\frac23b
$$
$$
\implies -\frac{\lambda_s}{\lambda} \geq \frac49\, b . %\frac23 b .
$$
Integrating this in $s$ and inserting \eqref{E:IE 5}, we get
$$
-\log \lambda(s) \geq -\log \lambda(s_0)
+ \int_{s_0}^s \frac{\pi}{3\log \sigma} \, d\sigma .
$$
But
$$
\int_{s_0}^s \frac{d\sigma}{\log \sigma} \geq \int_{s_0}^s
\frac{\log \sigma -1}{(\log \sigma)^2} \, d\sigma = \frac{s}{\log s} - \frac{s_0}{\log s_0},
$$
and thus,
$$
-\log \lambda(s) \geq -\log \lambda(s_0) + \frac{\pi}{3}\left( \frac{s}{\log s}
- \frac{s_0}{\log s_0} \right).
$$
By IDA 6: $0<\lambda_0 < \exp(-\exp 8\pi/9b_0)$ and the definition of $s_0$:
$s_0=e^{3\pi/4b_0}$, we get $\lambda_0=\lambda(s_0) \geq s_0^{32/37}$,
and thus,
$$
-\log \lambda(s) \geq -\frac{1}{2}\log \lambda(s_0) + \frac{\pi s}{3\log s}
+ \left(- \frac12\log\lambda(s_0)-\frac{\pi}3\frac{s_0}{\log s_0}\right)
$$
and since $s_0 \gg 1$ (or $b_0 \ll 1$), the term in parentheses is positive and we have
\begin{equation}
\label{E:logs1}
-\log \lambda(s) \geq -\frac{1}{2}\log \lambda(s_0) + \frac{\pi s}{3\log s}.
\end{equation}
Exponentiating this equation gives \eqref{E:IE 6}. Now we show how \eqref{E:IE 5}
and \eqref{E:IE 6} imply BSO 6.  Recall \eqref{E:IE 5} and \eqref{E:IE 6}
%\begin{equation}
%%\renewcommand{\theequation}{\ref{E:IE 5}}
%b(s) \geq \frac{3\pi}{4\log s}
%\end{equation}
%\begin{equation}
%%\renewcommand{\theequation}{\ref{E:IE 6}
%\lambda(s) \leq \sqrt{\lambda_0} e^{-\frac{\pi}{3}\frac{s}{\log s}}
%\end{equation}
and that
$$
\text{BSO 6} \Longleftrightarrow \frac{\pi}{5}\leq b(t)\log|\log\lambda(t)|.
$$
%Taking the log of \eqref{E:IE 6},
%$$
%-\log \lambda(s) \geq -\frac12\log \lambda(s_0) + \frac\pi{3}\frac{s}{\log s}.
%$$
Subtract $\log s$ from both sides of \eqref{E:logs1} to get
%Now compute and plug in:
\begin{align*}
-\log(s\lambda(s)) &= -\log s -\log \lambda(s) \\
&\geq -\log s -\frac12 \log \lambda(s_0) +\frac{\pi}{3}\frac{s}{\log s} .
\end{align*}
Since $\lambda(s_0) \ll 1$, we have $-\frac12\log \lambda(s_0) \geq 0$, and hence,
$$
-\log(s\lambda(s)) \geq \frac{\pi}3 \frac{s}{\log s} - \log s
\geq \frac{\pi}{3.1}\frac{s}{\log s}.
$$
Taking the log, we obtain
%\begin{align*}
$$
\log(-\log(s\lambda(s))) \geq \log \frac{\pi}{3.1} + \log \left(\frac{s}{\log s}\right)
\geq \frac12\log s .
$$
%\end{align*}
By \eqref{E:IE 5},
\begin{equation}
\label{E:loglog8}
\log(-\log(s\lambda(s))) \geq \frac{3\pi}{8b(s)}.
\end{equation}
Since $s\gg 1$ and $s\lambda(s)\ll 1$, we have
$-\log \lambda(s) \geq -\log (s\lambda(s))$, and BSO 6 follows.

\section{Bootstrap Step 10.  Control on the radius of concentration}
\label{S:bs10}

The inequality \eqref{E:IE 3.2} implies
$$
\left| \frac{(r_s,z_s)}{\lambda} \right| \leq \delta(\alpha^*) \mathcal{E}(t)^{1/2}
+ \Gamma_b^{1-C\eta}\leq 1.
$$
Thus,
\begin{align*}
|r(s)-r_0| &\leq \int_{s_0}^s |r_s(\sigma)| \, d\sigma \\
&\leq \int_{s_0}^s \lambda(\sigma) \, d\sigma \\
&\leq \int_{2}^{+\infty} \sqrt{\lambda_0} e^{-\frac{\pi}{3}\frac{\sigma}{\log \sigma}} \, d\sigma\\
&\leq \alpha^*,
\end{align*}
where we applied \eqref{E:IE 6}.  It similarly follows that
$$
|z(s)-z_0| \leq \alpha^* \,.
$$

\section{Bootstrap Step 11.  Momentum conservation implies BSO 5}
\label{S:bs11}

Recall that the function $\psi(x,y,z)$ is an axially symmetric cutoff in $r$ only,
and is independent of $z$. Also recall the convention that the default (no subscript)
notation is for $xyz$ space, i.e., $\Delta = \partial_x^2 + \partial_y^2 + \partial_z^2$,
$\nabla=(\partial_x, \partial_y, \partial_z)$, and $\int (\cdots)$ will mean
$\int (\cdots) dxdydz$. Also note that (for axially symmetric $u$) we have
$|\nabla u|^2=|\nabla_{(r,z)}u|^2$. Pair the NLS equation with
$\frac{1}{2}\Delta \psi \bar u + \nabla \psi \cdot \nabla \bar u$, integrate and
take the real part to obtain
\begin{align*}
\indentalign \Im \int \partial_t u ( \tfrac{1}{2}\Delta \psi \bar u
+ \nabla \psi \cdot \nabla \bar u) \\
&=
\begin{aligned}[t]
&\Re \int \Delta u( \tfrac{1}{2}\Delta \psi \bar u + \nabla \psi \cdot \nabla \bar u)\\
&+\Re \int |u|^2 u ( \tfrac{1}{2}\Delta \psi \bar u + \nabla \psi \cdot \nabla \bar u)
\end{aligned}
\end{align*}
that we write as I = II + III.
We first note that by integration by parts, we have
$$
\text{I} = \frac{1}{2}\partial_t \int \nabla \psi \cdot \nabla u \; \bar u.
$$
Also, by integration by parts:
$$
\text{II.1}  = -\frac{1}{4}\int \Delta^2 \psi \; |u|^2
+ \frac12\int \Delta \psi \; |\nabla u|^2 .
$$
We convert the second term in II.1 to cylindrical coordinates to get
$$
\text{II.1}  = -\frac{1}{4}\int \Delta^2 \psi \; |u|^2  + \frac12\int \partial_r^2 \psi \;
|\nabla u|^2 rdrdz +  \frac12\int \partial_r \psi \; |\nabla u|^2 drdz .
$$
Since $\partial_z \psi =0$, we calculate by switching to $(r,z)$ coordinates
and integrating by parts in this setting:
$$
\text{II.2} = \frac12 \int \partial_r^2 \psi (|\partial_r u|^2-|\partial_zu|^2) rdrdz
- \frac12\int \partial_r \psi |\nabla_{(r,z)}u|^2 drdz.
$$
Adding these, we get
$$
\text{II}=-\frac{1}{4}\int \Delta^2 \psi\; |u|^2+\int \partial_r^2 \psi\; |\partial_r u|^2.
$$
The two terms in III are easily manipulated (using integration by parts for the second one
in the $xyz$ variables):
$$
\text{III} = -\frac14 \int \Delta \psi \; |u|^4.
$$
Pulling it all together, we have
$$
\frac{1}{2}\partial_t \; \Im \int \nabla \psi \cdot \nabla u \; \bar u
= -\frac{1}{4}\int \Delta^2 \psi\; |u|^2 + \int \partial_r^2 \psi \; |\partial_r u|^2
-\frac14 \int \Delta \psi \; |u|^4.
$$
Thus, using the axial exterior Gagliardo-Nirenberg inequality (Lemma \ref{L:cyl-Strauss})
to control the 4th power term, we obtain
$$
\left| \partial_t \; \Im \int \nabla \psi \cdot \nabla u \; \bar u \right|
\leq \| \nabla u(t)\|_{H^1}^2 \leq c\frac{1}{\lambda^2(t)},
$$
where in the last step we used the decomposition \eqref{E:comp_ep_1}
%$$
%u(t,r,z) = \frac{1}{\lambda(t)} \tilde Q_{b(t)}\left( \frac{r-r(t)}{\lambda(t)},
%\frac{z-z(t)}{\lambda(t)} \right)e^{-i\gamma(t)} + \frac{1}{\lambda(t)}
%\epsilon\left( \frac{r-r(t)}{\lambda(t)}, \frac{z-z(t)}{\lambda(t)}, t\right)e^{-i\gamma(t)}
%$$
and the estimate $\| \nabla \epsilon \|_{L^2} \leq 1$ furnished by Lemma \ref{L:geomdecomp}.
Since we explicitly chose the origin of the rescaled time as $s_0 = e^{3\pi/4b_0}$, we have
$$
\int_0^ t \frac{d\tau}{\lambda^2(\tau)} = \int_{s_0}^s d\sigma = s-s_0 \leq s.
$$
We have, $\forall \; t \in [0,t_1)$, that
\begin{equation}
\label{E:Step11.1}
\lambda(t) \left| \Im \int \nabla \psi \cdot \nabla u(t) \; \bar u(t) \right|
\leq \lambda(t) \left| \Im \int \nabla \psi \cdot \nabla u_0 \; \bar u_0 \right|
+ c\lambda(t)s(t).
\end{equation}
From \eqref{E:loglog8} it follows that\footnote{The power of $\Gamma$ on the right side
can be 5 or higher.}
%Earlier, we showed
\begin{equation}
\label{E:Step11.2}
\forall \; s\in [s_0,s_1), \quad  s\lambda(s) \leq \Gamma_{b(s)}^{10} .
\end{equation}
Next, we have
\begin{align*}
\partial_s ( \lambda e^{6\pi/b})
&= \lambda e^{6\pi/b} \left( \frac{\lambda_s}{\lambda} - \frac{6\pi b_s}{b^2} \right) \\
&= \lambda e^{6\pi/b} \left( \frac{\lambda_s}{\lambda} +b - \frac{6\pi b_s}{b^2} \right)
- \lambda b e^{6\pi/b}.
\end{align*}
By \eqref{E:IE 3.1}, BSI 3, we have $|\lambda_s/\lambda + b| + |b_s| \leq \Gamma_{b(t)}^{3/4}$.
Also, ZP 2 implies that $1/b \sim \ln(\Gamma_{b(t)}^{-1})$, and thus,
$|6\pi b_s/b^2| \leq \Gamma_{b(t)}^{1/2}$. This shows that the first term is dominated
by the second in the above display, i.e., that $\partial_s (\lambda e^{6\pi/b})
\leq -\frac{1}{2}\lambda b e^{6\pi/b}<0$, and thus, for all $s>s_0$, we have
$\lambda(s) e^{6\pi/b(s)} \leq \lambda_0 e^{6\pi/b_0}$. Using that
$\Gamma_{b(t)}^{-5} \sim e^{5\pi/b(t)}$ (again a consequence of ZP 2), we obtain
\begin{align*}
\indentalign
\lambda(t) \Gamma_{b(t)}^{-5} \left| \Im \int \nabla \psi \cdot \nabla u_0 \; \bar u_0 \right| \\
&\leq \lambda(t) e^{6\pi/b(t)} \left| \Im \int \nabla \psi \cdot \nabla u_0 \; \bar u_0 \right| \\
&\leq \lambda_0e^{6\pi/b_0} \left| \Im \int \nabla \psi \cdot \nabla u_0 \; \bar u_0 \right|\\
&\leq e^{6\pi/b_0} \Gamma_{b_0}^{10}
\end{align*}
by IDA 5. Now since ZP 2 implies $\Gamma_{b_0}^{10} \sim e^{-10\pi/b_0}$, we have that the above is
$\leq 1$, or
\begin{equation}
\label{E:BSO5initial}
\lambda(t) \left| \Im \int \nabla \psi \cdot \nabla u_0 \; \bar u_0 \right| \leq \Gamma_{b(t)}^5.
\end{equation}
Inserting \eqref{E:BSO5initial} and \eqref{E:Step11.2} into \eqref{E:Step11.1}, we get
$$
\lambda(t) \left| \Im \int \nabla \psi \cdot \nabla u(t) \; \bar u(t) \right|
\leq \Gamma_{b(t)}^5.
$$

\section{Bootstrap Step 12.  Refined virial inequality in the radiative regime}
\label{S:bs12}

The goal of this section is to prove \eqref{E:IE 7}. This is the ``radiative virial estimate''
which will later be combined with \eqref{E:IE 8} to obtain \eqref{E:IE 9}, which is a refinement
(using the Lyapunov functional $\mathcal{J}$ in place of $b^2$) to the virial estimate
\eqref{E:IE 4}. Once again, the spectral property is a key ingredient. The idea in \eqref{E:IE 7},
\eqref{E:IE 8}, \eqref{E:IE 9} is that instead of modeling the solution as $\tilde Q_b + \epsilon$,
where $\tilde Q_b$ is supported inside radius $\tilde R \leq \frac{2}{b}$, we model the solution
as $\tilde Q_b+ \tilde \zeta_b + \tilde \epsilon$, where $\tilde \zeta$ corrects the solution
in the ``radiative'' region $\frac{2}{b} \leq \tilde R \ll e^{-\pi/b}$. There is still plenty of
distance between the radiative region  $\frac{2}{b} \leq \tilde R \ll e^{-\pi/b}$ and the radius
$\frac{1}{\lambda}$, which corresponds to a unit-sized distance from the blow-up core in the
original scale.

It amounts to taking the inner product of the $\epsilon$ equation with suitable directions built
on $\tilde Q_b+ \tilde \zeta$.  These computations are carried out in Lemma 6 of \cite{MR-JAMS}.
First, we proceed along the lines of Step 1--4 in the proof of Lemma 6 in \cite{MR-JAMS}.
Since the function $f_1(s)$ is supported in $B(0,2A)$, far inside the $\frac{1}{\lambda}$ width
in $\tilde R$, all interaction estimates are two dimensional -- the factor $\mu(\tilde r)$ can be
inserted up to possible correction terms with a $\lambda$ coefficient, which can easily be
estimated using \eqref{E:IE 0.1} ($\lambda \leq \Gamma_b^{10}$). The result is (as at the
beginning of Step 4 in the proof of Lemma 6 in \cite{MR-JAMS}, equation (4.18)), the bound
\newcommand{\re}{\textnormal{re}}
\newcommand{\im}{\textnormal{im}}
%S: Need to define F and also an equation for \tilde\xi_b
$$
\partial_s f_1 \geq
\begin{aligned}[t]
&H(\tilde \epsilon, \tilde \epsilon)  - \frac{1}{\| \Lambda Q\|_{L^2}^2}
(\tilde \epsilon_1, L_+ \Lambda^2 Q) (\tilde \epsilon_1, \Lambda Q) - C\lambda^2 E_0 \\
&+ (\epsilon_1, \Lambda F_{\re}) + (\epsilon_2, \Lambda F_{\im})
- ( \tilde \zeta_{\re}, \Lambda F_{\re}) - ( \tilde \zeta_{\im}, \Lambda F_{\im}) \\
&- \delta(\alpha^*) (\tilde{\mathcal{E}}(s)+ \lambda^2|E_0|) - \Gamma_b^{1+z_0}.
\end{aligned}
$$
By the spectral property, we have
$$
H(\tilde \epsilon, \tilde \epsilon) - \frac{1}{\|\Lambda Q\|_{L^2}^2}
(\tilde \epsilon_1, L_+\Lambda^2Q)(\tilde \epsilon_1,\Lambda Q)
\geq \tilde \delta_1 \tilde{\mathcal{E}} - C\lambda^2 E_0 - \delta_2 \Gamma_b.
$$
We also have
$$
|(\epsilon_1, \Lambda F_{\re}) + (\epsilon_2, \Lambda F_{\im})|
\leq C\Gamma_b^{1/2} \left( \int_A^{2A} |\epsilon|^2 \right)^{1/2}
\leq \delta_2 \Gamma_b + \frac{1}{\delta_2} \int_A^{2A} |\epsilon|^2 \,.
$$
A flux-type computation (see the (4.20) and its proof in \cite{MR-JAMS})
$$
-(\tilde\zeta_{\re}, \Lambda F_{\re}) - (\tilde \zeta_{\im}, \Lambda F_{\im}) \geq c\Gamma_b \,.
$$
Combining these elements gives \eqref{E:IE 7}.

\section{Bootstrap Step 13.  $L^2$ dispersion at infinity in space}
\label{S:bs13}

In this step, we prove \eqref{E:IE 8}. This gives us control of
$\int_{A \leq \tilde R \leq 2A} |\epsilon|^2$ in terms of the $s$-derivative of the 3d $L^2$
norm of $\epsilon$ ~~ \emph{outside} the radiative region. Since this involves estimating
$\epsilon$ on the whole space, the factor $\mu(\tilde r)$ is crucial, and we fully write out
this step.  With
$$
\phi_4\left(\frac{\tilde r}{A}, \frac{\tilde z}{A}\right)
= \phi_4 \left( \frac{r-r(t)}{\lambda(t)A(t)}, \frac{z-z(t)}{\lambda(t)A(t)} \right).
$$
we write
\begin{align*}
\indentalign 4\pi \frac{d}{dt} \int \phi_4 |u|^2 \, r drdz  \\
&= - 4\pi \int \nabla \phi_4  \cdot (\partial_t \alpha, \partial_t \beta) |u|^2 \,rdrdz
+ \frac{1}{4\pi \lambda A} \Im \int (\nabla \phi_4 \cdot \nabla u )\bar u \, rdrdz ,
\end{align*}
where $\alpha(t) = \frac{r-r(t)}{A(t)\lambda(t)}$ and $\beta = \frac{z-z(t)}{A(t)\lambda(t)}$.
Using that $\frac{ds}{dt} = \frac{1}{\lambda^2}$, we compute
$$
(\alpha_t,\beta_t) = -\frac{1}{\lambda^2} \left( \frac{\lambda_s}{\lambda}
+ \frac{A_s}{A}\right) \left( \frac{r-r(s)}{\lambda}, \frac{z-z(s)}{\lambda}\right)
-\frac{1}{\lambda^2} \left( \frac{r_s}{\lambda}, \frac{z_s}{\lambda}\right).
$$
We have
\begin{align*}
\indentalign
\frac{1}{2} \frac{d}{ds} \int \phi_4\left( \frac{\tilde R}{A}\right) |\epsilon|^2 \mu(\tilde r)
\, d\tilde r \, d\tilde z \\
&=
\begin{aligned}[t]
&\frac{1}{A} \Im \int (\nabla \phi_4 \cdot \nabla \epsilon)  \, \bar \epsilon \, \mu(\tilde r)
\, d\tilde r\, d\tilde z + \frac{b}{2} \int \frac{\Lambda^2}{A} \cdot \nabla \phi_4
\left(\frac{(\tilde r, \tilde z)}{A} \right) \, |\epsilon|^2 \, \mu(\tilde r) \, d\tilde r \, d\tilde z \\
&- \frac{1}{2A} \int \left( \frac{\lambda_s}{\lambda} + b + \frac{A_s}{A} \right)(\tilde r, \tilde z)
\cdot \nabla \phi_4\Big( \frac{\tilde R}{A} \Big) |\epsilon|^2 \mu(\tilde r) \, d\tilde r \, dz \\
&- \frac{1}{2A} \int \frac{(r_s,z_s)}{\lambda}  \cdot \nabla \phi_4\Big( \frac{\tilde R}{A} \Big)
|\epsilon|^2 \mu(\tilde r) \, d\tilde r \, dz.
\end{aligned}
\end{align*}
Using that $A_s/A = -b_s/b^2$ and \eqref{E:IE 3.1}, we obtain that
$$
|\text{III}| \leq \frac{b}{40} \int |\nabla_{\tilde r, \tilde z}\phi_4( \tilde R/A) |\epsilon|^2
\mu(\tilde r) d\tilde r \, d\tilde z.
$$
Using (\eqref{E:IE 3.2} $\implies |(r_s,z_s)|/\lambda \leq 1$), we obtain
$$
|\text{IV}| \leq \lambda \|\tilde u \|_{L^2_{xyz}} \leq \Gamma_b^2.
$$
Term I is estimated via Cauchy-Schwarz and basic interpolation %Peter-Paul
$$
|\text{I}| \leq \frac{b}{40} \int | \nabla \phi_4(\tilde R/A)| |\epsilon|^2 \mu(\tilde r)
d\tilde r d\tilde z + \Gamma_b^{a/2} \int |\nabla \epsilon|^2 \mu(\tilde r) \, d\tilde r \,dz\,.
$$
For Term II, we use that
$$
\frac{(\tilde r, \tilde z)}{A} \cdot \nabla \phi_4\left( \frac{(\tilde r,\tilde z)}{A} \right) \geq
\begin{cases}
\frac{1}{4} & \text{for }A \leq |(\tilde r, \tilde z)| \leq 2A \\
0 & \text{otherwise}
\end{cases}
$$
and the fact that $\mu(\tilde r) \sim 1$ for $|(\tilde r, \tilde z)| \sim A$ to obtain
$$
\text{II} \geq \frac{b}{30} \int_{A \leq |(\tilde r, \tilde z)| \leq 2A} |\epsilon|^2
d\tilde r \, d\tilde z \,,$$
finishing the proof of \eqref{E:IE 8}.

\section{Bootstrap Step 14.  Lyapunov functional in $H^1$}
\label{S:bs14}

In this section, we exhibit the Lyapunov function $\mathcal{J}\sim b^2$ and prove
the upper bound on $\partial_s \mathcal{J}$ given in \eqref{E:IE 9}.
This is obtained by combining \eqref{E:IE 7} (virial identity for dynamics in the soliton core)
and \eqref{E:IE 8} (dispersion relation in the radiative regime) and the $L^2$ conservation (a global quantity that links the two).

Let
$$
\mathcal{J}(s) =
\begin{aligned}[t]
& \int |\tilde Q_b|^2 - \int Q^2 + 2(\epsilon_1, \Sigma) + 2(\epsilon_2, \Theta) \\
& +\frac{1}{r(s)} \int (1-\phi_4 \left( \frac{(\tilde r, \tilde z)}{A} \right)
|\epsilon|^2 \mu(\tilde r) \, d\tilde r \, d\tilde z  \\
& -\frac{\delta_1}{800} \left( b \tilde f_1(b) - \int_0^b \tilde f_1(v)dv
+ b (\epsilon_2, \Lambda \tilde \zeta_\re) - b (\epsilon_1, \Lambda \tilde \zeta_\im) \right),
\end{aligned}
$$
where
$$
\tilde f_1(b) = \frac{b}{4} \|y\tilde Q_b\|_{L^2}^2 + \frac12 \Im \int (\tilde r, \tilde z)
\cdot \nabla \tilde \zeta \; \bar{\tilde \zeta} \,.
$$

The argument follows the proof of Prop. 4 in \cite{MR-JAMS}.
Multiply \eqref{E:IE 7} by $\delta_1 b / 800$ and sum with \eqref{E:IE 8} to obtain
\begin{equation}
\label{E:Lyp-1}
\begin{aligned}
\indentalign
\partial_s \left( \frac{1}{r(s)} \int \phi_4\left( \frac{(\tilde r, \tilde z)}{A}\right)
|\epsilon|^2 \mu(\tilde r) \, d\tilde r \, d\tilde z \right) + \frac{\delta_1 b}{800} \partial_s f_1 \\
&\geq \frac{\delta_1^2 b}{800} \tilde{\mathcal{E}}
+ \frac{b}{800} \int_{A\leq \tilde R \leq 2A} |\epsilon|^2 d\tilde r \, d\tilde z
+ \frac{c\delta_1b}{1000} \Gamma_b -\frac{c}{b^2}\, \lambda^2 \, E_0
- \Gamma_b^{a/2} \int|\nabla_{(\tilde r, \tilde z)}\epsilon|^2 \mu(\tilde r) \,d\tilde r \,d\tilde z.
\end{aligned}
\end{equation}
The last term on the right side is estimated as
\begin{align*}
\Gamma_b^{a/2} \int |\nabla_{(\tilde r, \tilde z)} \epsilon|^2 \mu(\tilde r) \, d\tilde r \, d\tilde z
& \leq \Gamma_b^{a/2}\left( \Gamma_b^{1-C\eta} + \int |\nabla_{(\tilde r, \tilde z)}\tilde \epsilon|^2
\mu(\tilde r) d\tilde r \, d\tilde z \right) \\
& \leq \Gamma_b^{1+a/4} + \Gamma_b^{a/2} \int |\nabla_{(\tilde r, \tilde z)}\tilde \epsilon|^2
\mu(\tilde r) \, d\tilde r \, d\tilde z.
\end{align*}
Using that
$$
f_1(s) = \tilde f_1(s)+(\epsilon_2, \Lambda \tilde \zeta_\re)-(\epsilon_1, \Lambda \tilde \zeta_\im)\,,
$$
we also rewrite the second term on the left side as
\begin{align*}
b\partial_s f_1
&= \partial_s (bf_1) - b_s f_1 \\
&= \partial_s (bf_1) - b_s \tilde f_1 - b_s [ (\epsilon_2,\Lambda \tilde \zeta_\re)
- (\epsilon_1, \Lambda \zeta_\im)]\\
&=
\begin{aligned}[t]
&\partial_s \left( b \tilde f_1(b) - \int_0^b \tilde f_1(v) \, dv
+ b( \epsilon_2, \Lambda \tilde \zeta_\re) - b (\epsilon_1, \Lambda \zeta_\im)\right) \\
&- b_s[ (\epsilon_2,\Lambda \tilde \zeta_\re) - (\epsilon_1, \Lambda \zeta_\im)].
\end{aligned}
\end{align*}
Formula \eqref{E:Lyp-1} becomes, using the bound on $|b_s|$ in \eqref{E:IE 3.1},
\begin{equation}
\label{E:Lyp-3}
\begin{aligned}
&\partial_s \Bigg( \frac{1}{r(s)}\int \phi_4(\frac{\tilde R}{A}) |\epsilon|^2
\mu(\tilde r) \, d\tilde r \, d\tilde z \\
&\quad + \frac{\delta_1}{800} \left[ b\tilde f_1(b) - \int_0^b\tilde f_1(v)dv
+ b(\epsilon_2,\Lambda \tilde \zeta_\re) - b(\epsilon_1, \Lambda \tilde \zeta_\im)\right] \Bigg)  \\
&\qquad \qquad \geq \frac{\delta_1^2b}{800} \left(\tilde{\mathcal{E}}(s)
+ \int_A^{2A} |\epsilon|^2 \right) + \frac{c\delta_1 b}{1000} \Gamma_b. %&
\end{aligned}
\end{equation}
The conservation of $L^2$ norm, written in terms of $\epsilon$ and $\tilde Q_b$ is
$$
\int |\epsilon|^2 \mu(\tilde r) d\tilde r d\tilde z + \int |\tilde Q_b|^2 \mu(\tilde r) \, d\tilde r
d\tilde z + 2\Re \int \epsilon \overline{\tilde Q_b} \mu(\tilde r) d\tilde r d\tilde z = \int |u_0|^2,
$$
which we rewrite as
\begin{equation}
\label{E:Lyp-2}
\begin{aligned}
\indentalign \int |\epsilon|^2 \mu(\tilde r) \, d\tilde r d\tilde z + r(t) \left( |\tilde Q_b|^2 - \int |Q|^2 + 2\Re (\epsilon, \overline{\tilde Q_b}) \right) \\
&= \int |u_0|^2 - r(t) \int Q^2 - 2\lambda \Re (\epsilon, \tilde r \tilde Q_b).
\end{aligned}
\end{equation}
Then write
$$
\int \phi_4(\tilde R/A) |\epsilon|^2 \mu(\tilde r) d\tilde r d\tilde z = \int |\epsilon|^2 \mu(\tilde r)
d\tilde r d\tilde z - \int (1-\phi_4(\tilde R/A)) |\epsilon|^2 \mu(\tilde r) \, d\tilde r d\tilde z
$$
and obtain by taking $\partial_s$ of \eqref{E:Lyp-2},
\begin{align*}
\indentalign
\partial_s \left( \frac{1}{r(s)} \int \phi_4(\tilde R/A) |\epsilon|^2 \mu(\tilde r)\,d\tilde r d\tilde z \right) \\
&=
\begin{aligned}[t]
&- \partial_s \left( \int |\tilde Q_b|^2 - \int Q^2 + 2(\epsilon_1,\Sigma)-2(\epsilon_2,\Theta) \right) \\
&+ \partial_s \left( \frac{1}{r(s)} \int (1-\phi_4(\tilde R/A)) |\epsilon|^2 \mu(\tilde r) \,d\tilde r d\tilde z \right) \\
&- 2 \partial_s \left( \frac{\lambda}{r(s)} \Re (\epsilon, \tilde r \overline{\tilde Q_b}) \right)
- \frac{r_s}{r^2} \int |u_0|^2.
\end{aligned}
\end{align*}
The last two terms are easily estimate with \eqref{E:IE 3.1} and \eqref{E:IE 3.2}.
Inserting into \eqref{E:Lyp-3}, we obtain \eqref{E:IE 9}.

We next show the two estimates \eqref{E:IE 10} and \eqref{E:IE 11} for $\mathcal{J}$.

We first show \eqref{E:IE 10}.  We have $\|\tilde Q_b\|_{L^2}^2 - \|Q\|_{L^2}^2 = (d_0+o(1))b^2$,
and now we estimate the rest of the terms in the definition of $\mathcal{J}$.  By Cauchy-Schwarz,
$$
|(\epsilon_1, \Sigma) + (\epsilon_2, \Theta) |
\lesssim \int |\epsilon|^2 e^{-\tilde R} d\tilde r d\tilde z \leq \Gamma_b^{1/2} \ll b \,.
$$
Next, we have
$$
\int (1-\phi_4(\tilde R/A)) |\epsilon|^2 \mu(\tilde r) d\tilde r d\tilde z
\leq \int_{ \substack{ |r-r(t)| \lesssim \lambda A \\ |z-z(t)| \lesssim \lambda A }} |\tilde u|^2 dxdydz.
$$
The set $\{ (x,y,z) \, | \, |r-r(t)|\lesssim \lambda A \,, |z-z(t)| \lesssim \lambda A \}$ has a volume
$\sim (\lambda A)^2$ in $\mathbb{R}^3$.  Therefore, by H\"older, the above is bounded by
$$
(\lambda A)\left(\int_{ \substack{ |r-r(t)| \lesssim \lambda A \\ |z-z(t)| \lesssim \lambda A }}
|\tilde u|^4 dxdydz \right)^{1/2}.
$$
Then in the $L^4$, we widen the spatial restriction to $r \geq \frac12$ and apply the axial Gagliardo-Nirenberg (Lemma \ref{L:cyl-Strauss}) and the definition of $A$ to obtain
$$
\lambda A \| \tilde u \|_{L^2} \| \nabla \tilde u \|_{L^2} \lesssim  A \left( \int_{\tilde r, \tilde z}
|\nabla_{(\tilde r, \tilde z)} \epsilon|^2 \,\mu(\tilde r)\,d\tilde r\,d\tilde z \right)^{1/2} \ll b \,.
$$
We claim that $|\tilde f_1(b)| \lesssim b$. Indeed, we first note that $\frac14 b \|\tilde R \tilde Q_b\|_{L^2}^2 \lesssim b$.
Also, recalling that $\tilde \zeta = \phi_3(\tilde R/A) \zeta$, we have
\begin{align*}
\left| \Im \int \tilde R \, \partial_{\tilde R} \tilde \zeta \, \tilde \zeta d\tilde r d\tilde z \right|
&\leq \int_{\tilde R\lesssim A} \tilde R \, |\partial_{\tilde R} \zeta| \, |\zeta| \, d\tilde r d\tilde z + \frac{1}{A} \int_{\tilde R\lesssim A} \tilde R \, |\zeta|^2 d \tilde r \, d\tilde z \\
&\leq \| \tilde R \, \zeta \|_{L^\infty}A \|\nabla \zeta\|_{L^2} + \frac{1}{A} \|\tilde R \, \zeta \|_{L^\infty}^2 \int \tilde R^{-1} \, d\tilde r \, d\tilde z \\
&\leq \Gamma_b A (\Gamma_b^{1-C\eta})^{1/2} + \Gamma_b \ll b,
\end{align*}
which establishes that $|\tilde f_1(b)| \lesssim b$. It is similarly straightforward to show that
$\left| \int_0^b \tilde f_1(v) \, dv \right| \lesssim b$. Finally, the local terms
$(\epsilon_2, \Lambda \tilde \zeta_\re)$ and $(\epsilon_1, \Lambda \tilde \zeta_\im)$ are estimated
by Cauchy-Schwarz and the ZP 1--2  properties, and shown to be $\ll b$.
This concludes the proof of \eqref{E:IE 10}.

For the proof of \eqref{E:IE 11}, see the proof of Prop. 5 in \cite{MR-JAMS}, which uses
the conservation of energy.

\section{Bootstrap Step 15.  Control on $\mathcal{E}(t)$ and upper bound on $b(s)$}
\label{S:bs15}

Next, we prove \eqref{E:IE 12}. Using that \eqref{E:IE 10} implies
$\left| \sqrt{\frac{d_0}{\mathcal{J}}} - b^{-1} \right| \ll 1$, we have
$$
\frac{1}{b^2}\exp \left( \frac{5\pi}{4} \sqrt{ \frac{d_0}{\mathcal{J}}} \right)
\gtrsim \frac{1}{b^2}\exp \left( \frac{7\pi}{6} b^{-1} \right) \gtrsim \Gamma_b^{-1} .
$$
Thus, by \eqref{E:IE 9} (note that $-\mathcal{J}_s \geq 0$)
$$
\partial_s \exp \left( \frac{5\pi}{4} \sqrt{ \frac{d_0}{\mathcal{J}}} \right)
\sim  \mathcal{J}^{-3/2} (-\mathcal{J}_s) \exp\left( \frac{5\pi}{4} \sqrt{ \frac{d_0}{\mathcal{J}}}
\right) \gtrsim -\frac{\mathcal{J}_s}{b\Gamma_b} \geq 1 .
$$
Integrating from $s_0$ to $s$, we obtain
$$
\exp \left( \frac{5\pi}{4} \sqrt{ \frac{d_0}{\mathcal{J}(s)}} \right)
- \exp \left( \frac{5\pi}{4} \sqrt{ \frac{d_0}{\mathcal{J}(s_0)}}\right) \geq s-s_0.
$$
Since $\exp \left( \frac{5\pi}{4} \sqrt{ \frac{d_0}{\mathcal{J}(s_0)}}\right) \geq e^{\pi/b_0} > s_0$,
we conclude that
$$
\exp \left( \frac{5\pi}{4} \sqrt{ \frac{d_0}{\mathcal{J}(s)}} \right) \geq s
$$
and \eqref{E:IE 12} follows.

Next, we obtain \eqref{E:IE 13}.
Divide \eqref{E:IE 9} by $J^{1/2}$ and use \eqref{E:IE 10} ( $\mathcal{J} \sim b^2$) to obtain
$$
\Gamma_{b(t)} + \tilde{\mathcal{E}}(t) \leq -C J^{-1/2} \mathcal{J}_s .
$$
Integrate from $s_0$ to $s$ to obtain
$$
\int_{s_0}^s (\Gamma_{b(\sigma)} + \tilde{\mathcal{E}}(\sigma)) \,d\sigma
\lesssim  \mathcal{J}^{1/2}(s_0)-\mathcal{J}^{1/2}(s) \lesssim b(s_0) \lesssim \alpha^*,
$$
where in the last step we applied IDA 2.

Finally, we prove BSO 3. Let $s\in [s_0,s_1)$ (recall that we are carrying out the bootstrap argument on
$0\leq t<t_1$, and $s_1=s(t_1)$). If $b_s(s)\leq 0$, then BSO 3 at $s$ follows immediately from \eqref{E:IE 4}
at $s$ (the local virial identity, $b_s \geq \delta_0 \mathcal{E}(s) - \Gamma_{b(s)}^{1-C\eta}$).

If $b_s(s)>0$, let $s_2 \in [s_0,s)$ be the smallest time such that for all $\sigma\in (s_2,s)$,
we have $b_s(\sigma)>0$. Then $b(s_2)\leq b(s)$ and either $s_2=s_0$ or $b_s(s_2)=0$.
In either case, we claim
\begin{equation}
\label{E:s15-1}
\mathcal{E}(s_2) \leq \Gamma_{b(s_2)}^{6/7}.
\end{equation}
If $s_2=s_0$, then \eqref{E:s15-1} is just IDA 4. If $b_s(s_2)=0$, then \eqref{E:s15-1} just follows
from the local virial identity \eqref{E:IE 4} at $s_2$.  From the second of the two estimates
in \eqref{E:IE 11} and also \eqref{E:s15-1}, we have
$$
\mathcal{J}(s_2) - f_2(b(s_2)) \leq \Gamma_{b(s_2)}^{5/6}.
$$
By the first of the two estimates in \eqref{E:IE 11}, the fact that $\mathcal{J}$ is (nonstrictly) decreasing, and the above estimate, we have
\begin{align*}
f_2(b(s)) + \frac{1}{C}\mathcal{E}(s) &\leq \mathcal{J}(s) + \Gamma_{b(s)}^{5/6} \\
&\leq \mathcal{J}(s_2) + \Gamma_{b(s)}^{5/6} \\
&\leq f_2(b(s_2)) + \Gamma_{b(s_2)}^{5/6} + \Gamma_{b(s)}^{5/6}.
\end{align*}
Combining this estimate with
$$
0< \frac{df_2}{db^2}\Big|_{b^2=0} < +\infty ~~\text{ and }~~ b(s_2)\leq b(s)
\quad \implies \quad f_2(b(s_2))\leq f_2(b(s))
$$
and
$$
b(s_2) \leq b(s) \implies \Gamma_{b(s_2)} \leq \Gamma_{b(s)}
$$
yields
$$
f_2(b(s)) + \frac{1}{C}\mathcal{E}(s) \leq f_2(b(s)) + 2 \Gamma_{b(s)}^{5/6}.
$$
Canceling $f_2(b(s))$ yields BSO 3.

%%%%%%%%%%%%%%%%%%%%%%%%%%%%%%%%%%%%%%%%%%%% To insert  %%%%%%%%%%%%%%%%%%%%%%%%%%%%%%%%%%%%%%%%%%

\section{Bootstrap Step 16.  $H^{1/2}$ interior smallness.}
\label{S:bs16}

\subsection{Preliminaries}

We consider two fractional derivative operators in our analysis:
$D_{rz}^s$ and $D_{xyz}^s$. The operator $D_{rz}^s$  acts directly
on a function $f(r',z')$ defined in $\mathbb{R}^2$  and returns a
function of $(r,z)\in \mathbb{R}^2$:
$$
D_{rz}^s = \mathcal{F}_2^{-1} (\rho^2 + \zeta^2)^{s/2}
\mathcal{F}_2,
$$
where $\mathcal{F}_2$ is the Fourier transform on $\mathbb{R}^2$ and
we adopt the practice of writing the dual variables of $(r,z)$ as
$(\rho,\zeta)$.  The operator $D_{xyz}^s$ acts on a function
$f(x',y',z')$ defined on $\mathbb{R}^3$ and returns a function of
$(x,y,z)$ on $\mathbb{R}^3$
$$
D_{xyz}^s = \mathcal{F}_3^{-1} (\xi^2+\eta^2+\zeta^2)^{s/2}
\mathcal{F}_3,
$$
where we adopt the practice of writing the dual variables of
$(x,y,z)$ as $(\xi,\eta,\zeta)$.  Now if $f(x,y,z)$ is an axially
symmetric function on $\mathbb{R}^3$, then $D_{xyz}^sf(x,y,z)$ is
axially symmetric and can thus be viewed as a function of $(r,z) \in
\mathbb{R}^+\times \mathbb{R}$ and, when extended to an even
function in $r$, it can be viewed as a function on $\mathbb{R}^2$.
In this case, we can write
$$
D_{xyz}^s = \mathcal{H}^{-1} (\rho^2+\zeta^2)^{s/2} \mathcal{H},
$$
where $\mathcal{H}$ is a zero-order Hankel transform\footnote{We
will not bother with constant factors (such as $(2\pi)^{-1}$) in the
definitions of the Fourier and Hankel transforms, as they are
irrelevant in the analysis.} (p. 341 of Bracewell \cite{B86}) in $r$
and a Fourier transform in $z$:
$$
\mathcal{H} f(\rho,\zeta) = \int_{z=-\infty}^{+\infty} \int_{r=0}^\infty J_0(r\rho) e^{-iz\zeta} f(r,z) r \, dr \, dz
$$
and the zero-order Bessel function on $\mathbb{R}$ is:
$$
J_0(\omega) = \frac{1}{2\pi}\int_{\theta=-\pi}^\pi e^{i\omega \sin \theta} d\theta
$$
(which extends $J_0(\omega)$ as an even function in $\omega$, and
thus $\mathcal{H}f(\rho,\zeta)$ is even in $\rho$.)\footnote{Note
that if we wanted to write the domain of $r$-integration in the
definition of $\mathcal{H}$ as $(-\infty,+\infty)$, assuming
$f(r,z)$ is even in $r$, we would need to put $|r|$ in the integrand
in place of $r$.}  We note that
$$
\mathcal{H}^{-1} F(r,z) = \int_{\zeta=-\infty}^{+\infty}
\int_{\rho=0}^{+\infty} J_0(r\rho) e^{iz\zeta} F(\rho,\zeta) \rho \,
d\rho \, d\zeta.
$$

If a function $h(x,y,z)$ is axially symmetric and the support of $h$
is contained in $0 <\delta \leq r \leq \delta^{-1}$, then we will
say that $h$ is \emph{shell supported}.  We note that for a shell
supported function $h$, $\|h\|_{L^p_{rz}} \sim \|h\|_{L^p_{xyz}}$.
Unfortunately,  if $h$ is shell-supported, then fractional
derivatives $D_{rz}^\alpha h$ are no longer shell-supported.  To
deal with this, and handle the conversion from $D_{rz}^s$ to
$D_{xyz}^s$, we use the following standard microlocal fact, which is
a consequence of the pseudodifferential calculus (see Stein
\cite{Stein}, Chapter VI, \S 2, Theorem 1 on p. 234 and \S 3,
Theorem 2 on p. 237; see also Evans-Zworski \cite{EZ}).

\begin{lemma}[disjoint smoothing]
\label{L:disj-smooth}
Suppose that $\operatorname{dist}(\supp \psi_1, \supp \psi_2)> 0$ and $s,\alpha \geq 0$. Then for $h(r,z)$ we have
$$
\| \psi_1 \; D_{rz}^s \; \psi_2 \; h \|_{H_{rz}^\alpha}
\lesssim_{s,\alpha} \| \psi_2 h \|_{L_{rz}^2}\, ,
$$
and for $h(x,y,z)$ we have
$$
\| \psi_1 \; D_{xyz}^s \; \psi_2 \; h \|_{H_{xyz}^\alpha}
\lesssim_{s,\alpha} \| \psi_2 h \|_{L_{xyz}^2}\,.
$$
\end{lemma}
Next, we address the matter of converting from $D_{rz}^s$ to
$D_{xyz}^s$.  We need the composition (Bracewell \cite{B56})
$$
\mathcal{F}_2 \mathcal{H}^{-1} = \mathcal{A},
$$
where $\mathcal{A}$ is the Abel transform (see p. 351 of Bracewell
\cite{B86})
$$
\mathcal{A} f(\rho,\zeta) = \int_\rho^{+\infty} \frac{
\rho' f(\rho',\zeta)}{(\rho'-\rho)^{1/2}(\rho'+\rho)^{1/2}} \,
d\rho'
$$
(which is the identity in $\zeta$).

\begin{lemma}[fractional derivative conversion]
\label{L:convert-frac}
For any $0 \leq s \leq 2$, the composition
$$
r^{1/2} D_{rz}^s D_{xyz}^{-s} r^{-1/2}
$$
is bounded as an operator $L^2_{rz} \to L^2_{rz}$.
\end{lemma}
We remark that scaling requires that the weights in any such
estimate be $r^\alpha$ on the left and $r^{-\alpha}$ on the right
for some $\alpha\in \mathbb{R}$.  We only need $\alpha=\frac12$ in
our analysis and have not explored the possible validity of other
values of $\alpha$.
\begin{proof}
Note that the composition under consideration is
$$
U = r^{1/2} \mathcal{F}^{-1} (\rho^2+\zeta^2)^{s/2} \mathcal{A} ((\rho')^2+\zeta^2)^{-s/2} \mathcal{H}^{-1} r^{-1/2} \,.
$$
Note that $(r+i0)^{1/2}\mathcal{F}^{-1} = \mathcal{F}^{-1}
\partial_\rho I_\rho^{1/2}$where $I^\alpha$ is the fractional
integral operator $I^\alpha f(\rho) = \int_{\rho'=\rho}^{+\infty}
(\rho'-\rho)^{\alpha-1}f(\rho') \, d\rho'$.  This, together with the
fact that $\mathcal{F}^{-1}$ is an $L_{\rho\zeta}^2\to L_{rz}^2$
unitary map, and the fact that $\rho^{1/2}\mathcal{H} r^{-1/2}$ is
an $L^2_{rz}\to L^2_{\rho\zeta}$ unitary map, we reduce to proving
the $L^2_{\rho\zeta} \to L^2_{\rho\zeta}$ boundedness of
$$
U = \partial_\rho I_\rho^{1/2} ((\rho')^2+\zeta^2)^{s/2} \mathcal{A}
((\rho'')^2+\zeta^2)^{-s/2} (\rho'')^{-1/2} \,.
$$
We obtain $Uf = \partial_\rho I_\rho^{1/2}h$, where
$$
h(\rho',\zeta) = \int_{\rho''=\rho'}^{\rho''=+\infty} \frac{
(\rho'')^{1/2} ((\rho')^2+\zeta^2)^{s/2} f(\rho'',\zeta)}{
(\rho''-\rho')^{1/2} (\rho''+\rho')^{1/2}
((\rho'')^2+\zeta^2)^{s/2}} \, d\rho''.
$$
By Fubini's theorem, we obtain
$$
Uf(\rho,\zeta) = \partial_\rho \int_{\rho''=\rho}^{\rho''=+\infty}
K(\rho,\rho'') f(\rho'',\zeta) \, d\rho'',
$$
where
$$
K(\rho,\rho'') = \int_{\rho'=\rho}^{\rho'=\rho''}
\frac{1}{(\rho'-\rho)^{1/2}(\rho''-\rho')^{1/2}}\left( \frac{
(\rho')^2+\zeta^2}{(\rho'')^2+\zeta^2} \right)^{s/2} \left(
\frac{\rho''}{\rho''+\rho'} \right)^{1/2} \,d \rho'.
$$
By changing the variables twice, first $\rho' \mapsto \rho'+\rho$
and then $\rho'=\sigma(\rho''-\rho)$, we have
$$
K(\rho,\rho'') = \int_{\sigma=0}^1
\frac{1}{\sigma^{1/2}(1-\sigma)^{1/2}} \left( \frac{ (\sigma \rho''
+ (1-\sigma)\rho)^2 + \zeta^2 }{ (\rho'')^2 + \zeta^2} \right)^{s/2}
\left( \frac{\rho''}{(1+\sigma)\rho''+(1-\sigma)\rho} \right)^{1/2}
d\sigma.
$$
Thus,
$$
Uf(\rho,\zeta) = -K(\rho,\rho)f(\rho,\zeta) +
\int_{\rho''=\rho}^{+\infty} \partial_\rho K(\rho,\rho'')
f(\rho'',\zeta) d\rho''.
$$
We have $K(\rho,\rho) = 2^{-1/2} \int_{\sigma=0}^1
\sigma^{-1/2}(1-\sigma)^{-1/2} d\sigma < \infty$. For the second
term, we change variable $\rho''=\rho\mu$ to obtain
\begin{equation}
\label{E:G100}
Uf(\rho,\zeta) = -cf(\rho,\zeta) + \int_{\mu=1}^{+\infty} \rho( \partial_\rho K)(\rho,\rho\mu) f(\rho\mu, \zeta) \, d\mu \,.
\end{equation}
But we have $\rho \partial_\rho K(\rho, \rho \mu) = I(\mu,\zeta/\rho)$, where
$$
I(\mu, \lambda) = \int_{\sigma=0}^1 \frac{1}{\sigma^{1/2}
(1-\sigma)^{1/2}} \left( \frac{(\sigma \mu + (1-\sigma))^2 +
\lambda^2}{\mu^2 + \lambda^2} \right)^{\frac{s}2} \left(
\frac{\mu}{(1+\sigma)\mu + 1-\sigma} \right)^{\frac12} J(\mu,
\lambda) \, d\sigma
$$
and
$$
J(\mu, \lambda) = s \frac{ (\sigma \mu + (1-\sigma))(1-\sigma)}{
(\sigma \mu + (1-\sigma))^2 + \lambda^2} + \frac12
\frac{1-\sigma}{(1+\sigma)\mu + (1-\sigma)} =J_1(\mu, \lambda) +
J_2(\mu, \lambda).
$$
Denote by $I_k$ the result of substituting $J_k$ into the expression
for $I$.  Since $|J_2(\mu,\lambda)| \leq (\mu+1)^{-1}$, we have
$$
|I_2(\mu, \lambda)| \lesssim \frac{1}{\mu}.
$$
If $s\leq 2$, we have
\begin{align*}
|I_1(\mu,\lambda)|
&\leq \int_{\sigma=0}^1 \sigma^{-1/2}(1-\sigma)^{1/2}
\frac{s(\sigma \mu + (1-\sigma))}{(\mu^2+\lambda^2)^{\frac{s}{2}} ((\sigma \mu + (1-\sigma))^2 + \lambda^2)^{1-\frac{s}{2}}} \, d\sigma\\
&\leq \frac{s}{\mu^s} \int_{\sigma=0}^1 \frac{(1-\sigma)^{1/2}}{\sigma^{1/2} (\sigma \mu + (1-\sigma))^{1-s}} \, d\sigma
\end{align*}
By separately considering the regions $0\leq \sigma \leq \frac12$ and $\frac12 \leq \sigma \leq 1$, we obtain
\begin{align*}
|I_1(\mu, \lambda) |
&\lesssim \frac{1}{\mu} + \int_{\sigma=0}^{1/2} \frac{d \sigma}{\sigma^{1/2} (\sigma \mu+1)^{1-s}} \\
&\lesssim \frac{1}{\mu} + \int_{\eta=0}^{\eta=\mu} \frac{d \eta}{\eta^{1/2} (\eta +1)^{1-s}}\\
&\lesssim
\begin{cases}
\frac{1}{\mu} & \text{if }s>\frac12\\
\frac{\log (\mu+1)}{\mu} & \text{if }s=\frac12 \\
\frac{s}{\mu^{s+\frac12}} & \text{if }0<s<\frac12.
\end{cases}
\end{align*}

Applying the  $L^2_\rho$ norm to \eqref{E:G100} and using
Minkowski's integral inequality, we bound by
$$
\|U f(\cdot, \zeta)\|_{L^2_\rho} \lesssim
\|f(\cdot,\zeta)\|_{L^2_\rho} + \left( \int_{\mu=1}^{+\infty} \max(
\mu^{-1}, s\mu^{-\frac12-s}) \mu^{-\frac12} \, d\mu \right)
\|f(\cdot, \zeta)\|_{L^2_\rho}.
$$
Following through with the $L^2_\zeta$ norm, we complete the proof.
\end{proof}

The following were established by Strichartz \cite{Str}, Keel-Tao
\cite{Keel-Tao} and Kenig-Ponce-Vega \cite{KPV}.  We say that
$(q,p)$ is 2d admissible if $\frac{2}{q}+\frac{2}{p}=1$ and $2\leq p
<\infty$.
\begin{lemma}[2d Strichartz estimates and local smoothing]
\label{L:2d-Str-smooth}
Suppose that $h(r,z,t)$ satisfies
$$
i\partial_t h +\Delta_{rz} h = g.
$$
Then for 2d admissible pairs $(q_1,p_1)$ and $(q_2,p_2)$,
$$
\|h\|_{L_t^{q_1}L_{rz}^{p_1}} \lesssim \|h_0\|_{L_{rz}^2}+
\left\{
\begin{aligned}
& \| g\|_{L_t^{q_2'}L_{rz}^{p_2'}} \\
& \|\la D_{rz}\ra ^{-1/2} g\|_{L_t^2L_{rz}^2} \quad \text{if $\supp g$ compact}, \\
\end{aligned}
\right.
$$
where $h_0(r,z) = h_0(r,z,0)$ is the initial data.  In the case of
the last bound $\supp g$ should be contained in a fixed compact set
for all $t$.
\end{lemma}

We say that $(q,p)$ is 3d admissible if
$\frac{2}{q}+\frac{3}{p}=\frac32$ and $2\leq p <6$.  \footnote{We
thank Fabrice Planchon for pointing out that, although the endpoint
case $(q_1,p_1)=(2,6)$ is available for the local smoothing
estimate, it does not follow from the Christ-Kiselev lemma.  We will
avoid the use of the endpoint in our analysis.}
\begin{lemma}[3d Strichartz estimates and local smoothing]
\label{L:3dStr}
Suppose that $f(x,y,z,t)$ solves
$$
i\partial_t f + \Delta_{xyz} f = g.
$$
Then for 3d admissible pairs $(q_1,p_1)$ and $(q_2,p_2)$, we have
$$
\|f \|_{L_t^{q_1}L_{xyz}^{p_1}} \lesssim \|f_0 \|_{L_{xyz}^2} +
\left\{
\begin{aligned}
& \| g\|_{L_t^{q_2'}L_{xyz}^{p_2'}} \\
& \|\la D_{xyz}\ra ^{-1/2} g\|_{L_t^2L_{xyz}^2} \quad \text{if $\supp g$ compact}, \\
\end{aligned}
\right.
$$
where $f_0(x,y,z) = f(x,y,z,0)$ is the initial data.  In the case of
the last bound $\supp g$ should be contained in a fixed compact set
for all $t$.
\end{lemma}

The local smoothing estimates in the above lemmas are established in
the homogeneous case by \cite{KPV} Theorem 4.1 on p. 54.  The
inhomogeneous version as stated above then follows from the
Christ-Kiselev  \cite{CK} lemma.

Finally, we record some Gagliardo-Nirenberg embeddings for shell-supported functions.

\begin{lemma}[embeddings]
Suppose that $q(r,z)=q(x,y,z)$ is axially symmetric and shell-supported.  Then
\begin{align}
\label{E:G11}
&\|q\|_{L_t^4L_{xyz}^4} \lesssim \|q\|_{L_t^\infty L_{xyz}^2}^{1/2} \|\nabla q\|_{L_t^2L_{xyz}^2}^{1/2} , \\
\label{E:G12}
&\| D_{xyz}^{1/2} |q|^2 \|_{L_t^{4/3} L_{xyz}^2} \lesssim \|q\|_{L_t^\infty L_{xyz}^2}^{1/2} \| \nabla q \|_{L_t^2 L_{xyz}^2}^{3/2} \,.
\end{align}
\end{lemma}
\begin{proof}
First we prove \eqref{E:G11}.
We have the 2d Gagliardo-Nirenberg estimate
$$
\|q\|_{L_{rz}^4} \lesssim \|q\|_{L_{rz}^2}^{1/2}\|\nabla_{rz} q\|_{L_{rz}^2}^{1/2} \,.
$$
Since $q$ and $\nabla_{rz} q$ are shell supported and $|\nabla_{rz}q|\leq |\nabla_{xyz} q|$, we obtain
$$
\|q\|_{L_{xyz}^4} \lesssim \|q\|_{L_{xyz}^2}^{1/2}\|\nabla_{xyz} q\|_{L_{xyz}^2}^{1/2} \,.
$$
Integrating in time, we obtain \eqref{E:G11}.
Next we prove \eqref{E:G12}.   We begin by noting
\begin{align*}
\| D_{xyz}^{1/2} |q|^2 \|_{L_{xyz}^2}^2 &= |\la D_{xyz}^{1/2}|q|^2, D_{xyz}^{1/2} |q|^2 \ra_{xyz}| \\
&= \la |q|^2 , D |q|^2 \ra_{xyz} \\
&= \la |q|^2, R \nabla |q|^2 \ra_{xyz} \\
&= \la |q|^2, R (\Re q \nabla \bar q) \ra_{xyz},
\end{align*}
where $R$ is the vector Riesz transform.  Using the boundedness of the Riesz transform on $L^{3/2}$, we obtain
\begin{align*}
\| D_{xyz}^{1/2} |q|^2 \|_{L_{xyz}^2}^2 &\lesssim \|q\|_{L_{xyz}^6}^2 \|R (\Re q \nabla \bar q)\|_{L_{xyz}^{3/2}} \\
&\lesssim \|q\|_{L_{xyz}^6}^3 \|\nabla q \|_{L_{xyz}^2}.
\end{align*}
By the 2d Gagliardo-Nirenberg estimate $\|q\|_{L_{rz}^6} \leq
\|q\|_{L_{rz}^2}^{1/3} \|\nabla_{rz} q\|_{L_{rz}^2}^{2/3}$, the fact
that $q$ and $\nabla_{rz}q$ are shell-supported, and $|\nabla_{rz}q|
\leq |\nabla_{xyz}q|$, we obtain
$$
\| D_{xyz}^{1/2} |q|^2 \|_{L_{xyz}^2} \lesssim \|q\|_{L_{xyz}^2}^{1/2} \|\nabla_{xyz} q\|_{L_{xyz}^2}^{3/2} \,.
$$
Integrating in $t$, we obtain \eqref{E:G12}.
\end{proof}

\subsection{Outline and notation}

Effectively, we would like to restrict to outside the singular
circle and reduce matters to the local theory, which is set in
$H^{1/2}$.

Let
$$
\tilde \chi_1(R) =
\begin{cases}
0 & \text{for } R \leq \frac14  \\
1 & \text{for } R \geq \frac34
\end{cases}
$$
and set $\chi_1(r,z) = \tilde \chi_1(((r-1)^2+z^2)^{1/2})$.  This is
a cutoff to \emph{outside} the singular circle, and we call the
support of $\chi_1$ the \emph{external region}. Let
$$
\tilde \chi_2(R) =
\begin{cases}
1 & \text{for } \frac14 \leq R \leq \frac34  \\
0 & \text{for } R \leq \frac18 \text{ and }R\geq \frac78
\end{cases} \,,
$$
$$
\tilde \chi_3(R) =
\begin{cases}
1 & \text{for } \frac18 \leq R \leq \frac78  \\
0 & \text{for } R \leq \frac1{16} \text{ and }R\geq \frac{15}{16}
\end{cases}
$$
and let $\chi_j(r,z) = \tilde \chi_j(((r-1)^2+z^2)^{1/2})$ for
$j=1,2$. Note that $\chi_2(r,z)$ is $1$ on the support of
$\chi_1(1-\chi_1)$ and $\chi_3(r,z)$ is $1$ on the support of
$\chi_2(r,z)$.  We will call the support of $\chi_2$ the \emph{tight
singular periphery} and the support of $\chi_3$ the \emph{wide
singular periphery}.  Let $v=\chi_1u$, $w=\chi_2 u$ and $q=\chi_3
u$.

Our goal is to obtain the estimate BSO 8, i.e.,
\begin{equation}
\label{E:G15}
\|D_{xyz}^{1/2} v \|_{L_t^\infty L_{xyz}^2} \leq (\alpha^*)^{3/8}.
\end{equation}
This will be a multistep process -- estimates on $q$ (Step A)
obtained from BSI 3 will imply estimates on $w$ (Step B), which will
imply estimates on $v$ (Step C).

\subsection{Step A.  Estimates on $q$}

Recall $q=\chi_3 u$, i.e., $u$ restricted to the wide singular
periphery.  Control on $q$ is inherited from the local virial
arguments in the previous sections.  Specifically, mass
conservation, \eqref{E:IE 13}, and BSI 3 imply (by rescaling) that
\begin{equation}
\label{E:qbds}
\|q\|_{L_t^\infty L_{xyz}^2} \leq \|u_0\|_{L_{xyz}^2} \,, \qquad
\|\nabla q\|_{L_t^2L_{xyz}^2} \lesssim (\alpha^*)^{10}\,.
\end{equation}
We also obtain an upper bound on $t_1$.  Indeed, from \eqref{E:IE 6},
\begin{equation}
\label{E:t1bd}
t_1 = \int_{s_0}^{s_1} \lambda^2(s) \, ds \lesssim C
\lambda_0 \int_2^{+\infty} e^{-(\pi/3)(s/\log s)}\leq \alpha^* \,.
\end{equation}

It will be understood that $t$ is always confined to $[0,t_1)$ below (when writing $L_t^\infty$, etc.)

\subsection{Step B.  Estimates on $w$}

Recall $w=\chi_2 u$, i.e., $u$ restricted to the tight singular
periphery.  We use the control on $u$ in the larger wide singular
periphery obtained above in Step A to provide control in a stronger
norm on the smaller tight singular periphery.   Since $w$ is
supported away from $r=0$ and away from $r=\infty$, it effectively
solves a 2d NLS equation.  This effective 2d equation can be
analyzed using the (locally-in-time) stronger 2d Strichartz
estimates as well as the 2d Gagliardo-Nirenberg estimates.

Specifically, the goal of this subsection is to obtain, using
\eqref{E:qbds}, the following bounds on $w$:
\begin{equation}
\label{E:w-goal-bounds}
\|\la D_{xyz} \ra^{1/2}
w\|_{L_t^qL_{xyz}^p}\lesssim (\alpha_*)^{1/2}
\end{equation}
for all 2d admissible $(q,p)$.

The function $D_{xyz}^{1/2}w = D_{xyz}^{1/2} \chi_2 u$ is still
axially symmetric but unfortunately (due the nonlocality of
$D_{xyz}^{1/2}$) no longer shell supported.  By Lemma
\ref{L:disj-smooth}, we have that $(1-\chi_3) \, D_{xyz}^{1/2} \,
\chi_2 u$ is infinitely smoothing, and thus,
$$
\| (1-\chi_3) \, D_{xyz}^{1/2} \, w \|_{L_t^qL_{xyz}^p}
\lesssim  \|w\|_{L_t^\infty L_{xyz}^2}.
$$
Hence, to establish \eqref{E:w-goal-bounds}, it suffices to prove,
for $p=\chi_3 \, D_{xyz}^{1/2} \, w$, the estimates
\begin{equation}
\label{E:p-bounds}
\|p\|_{L_t^qL_{xyz}^p} \lesssim \| p_0\|_{L_{xyz}^2} + (\alpha^*)^{1/2}
\end{equation}
and
\begin{equation}
\label{E:w-low-bounds}
\|w \|_{L_t^\infty L_{xyz}^2} \lesssim \| w_0\|_{L_{xyz}^2} + (\alpha^*)^{1/2} \,.
\end{equation}
The equation \eqref{E:w-goal-bounds} will then follow from IDA 7 and IDA 8.

We begin with \eqref{E:p-bounds}. Note that $p$ is axially symmetric
and shell supported.  Moreover, it solves an equation of the form
$$
i\partial_t p + \Delta_{rz}p = \sum_j  G_j,
$$
where the inhomogeneities can be put into the forms
\begin{align*}
&G_1 = \chi_3 \, D_{xyz}^{1/2} \, \chi_2 \, (i\partial_t + \Delta_{rz}) u \\
&G_2 = \psi_3 \, D_{xyz}^{1/2} \,\psi_2 \, u \\
&G_3 = \nabla_{rz} \, \psi_3 \, D_{xyz}^{1/2} \, \psi_2 \, u \,.
\end{align*}
For each type of term $G_j$, the functions $\psi_2$ and $\psi_3$ are
axial cutoff functions with support in $\frac14 \leq ((r-1)^2 +
z^2)^{1/2} \leq \frac34$ and $\frac18 \leq ((r-1)^2 + z^2)^{1/2}
\leq \frac78$, respectively.  For term $G_3$, we substitute the
equation $(i\partial_t + \Delta_{rz}) u  = -r^{-1}\partial_r u -
|u|^2u$ and then reexpress the term involving $r^{-1}\partial_r u$
to be of the type $G_2$ and $G_3$.  At this point, the
inhomogeneities take the form
\begin{align*}
&G_1 = \chi_3 \, D_{xyz}^{1/2} \, \chi_2 \, |u|^2 u \\
&G_2 = \psi_3 \, D_{xyz}^{1/2} \,\psi_2 \, u \\
&G_3 = \nabla_{rz} \, \psi_3 \, D_{xyz}^{1/2} \, \psi_2 \, u .
\end{align*}
By Lemma \ref{L:2d-Str-smooth},
\begin{equation}
\label{E:p-bounds-2} \|p\|_{L_t^qL_{rz}^p} \lesssim
\|p_0\|_{L_{rz}^2} + \|G_1\|_{L_t^{4/3}L_{rz}^{4/3}} +
\|G_2\|_{L_t^1L_{rz}^2} + \|\la D_{rz}\ra^{-1/2}G_3
\|_{L_t^2L_{rz}^2} ,
\end{equation}
since $G_3$ has compact support in $r$ and $z$.

Since $G_1$ is shell supported, the $rz$ norm converts to an $xyz$
norm.  The fractional Leibniz rule and the identity $\chi_2 u |u|^2
= w |q|^2$ allow us to bound as follows:
\begin{equation}
 \label{E:G1bd}
\|G_1\|_{L_t^{4/3}L_{xyz}^{4/3}} \leq
\|D_{xyz}^{1/2}w\|_{L_t^4L_{xyz}^4}\|q\|_{L_t^4L_{xyz}^4}^2 +
\|w\|_{L_t^\infty
L_{xyz}^4}\|D_{xyz}^{1/2}|q|^2\|_{L_t^{4/3}L_{xyz}^2} \,.
\end{equation}
Since $w$ is shell-supported, we can apply the 2d Sobolev embedding
\begin{equation}
\label{E:G1-1}
\begin{aligned}
\|w\|_{L_t^\infty L_{xyz}^4} &\sim \|w\|_{L_t^\infty L_{rz}^4} \\
&\lesssim \|D_{rz}^{1/2} w\|_{L_t^\infty L_{rz}^2} \\
&\lesssim \|\chi_3 D_{rz}^{1/2} w \|_{L_t^\infty L_{rz}^2} +
\|(1-\chi_3) D_{rz}^{1/2} w \|_{L_t^\infty L_{rz}^2}.
\end{aligned}
\end{equation}
We note that we have
\begin{align*}
\chi_3 D_{rz}^{1/2} w &= \chi_3 D_{rz}^{1/2}D_{xyz}^{-1/2}D_{xyz}^{1/2} w\\
&= (\chi_3 r^{-1/2}) (r^{1/2}D_{rz}^{1/2}D_{xyz}^{-1/2}
r^{-1/2})(r^{1/2} D_{xyz}^{1/2} w).
\end{align*}
By Lemma \ref{L:convert-frac} and  the support properties of $\chi_3$,
\begin{align*}
\|\chi_3 D_{rz}^{1/2} w\|_{L_t^\infty L_{rz}^2}
&\lesssim \| r^{1/2}D_{xyz}^{1/2} w\|_{L_t^\infty L_{rz}^2} = \|D_{xyz}^{1/2} w \|_{L_t^\infty L_{xyz}^2} \\
& \lesssim \|p\|_{L_t^\infty L_{xyz}^2} + \|(1-\chi_3) D_{xyz}^{1/2}
w\|_{L_t^\infty L_{xyz}^2} .
\end{align*}
Applying Lemma \ref{L:disj-smooth} to the second of these terms,
\begin{equation}
\label{E:G1-2}
\|\chi_3 D_{rz}^{1/2} w \|_{L_t^\infty L_{rz}^2}
\lesssim \|p\|_{L_t^\infty L_{xyz}^2} + \|w\|_{L_t^\infty
L_{xyz}^2}.
\end{equation}
On the other hand, by Lemma \ref{L:disj-smooth},
\begin{equation}
\label{E:G1-3}
\|(1-\chi_3) D_{rz}^{1/2} w \|_{L_t^\infty L_{rz}^2}
\lesssim \|w\|_{L_t^\infty L_{rz}^2}.
\end{equation}
Combining \eqref{E:G1-1}, \eqref{E:G1-2}, and \eqref{E:G1-3},
\begin{equation}
\label{E:G1bd1}
\| w \|_{L_t^\infty L_{xyz}^4} \lesssim
\|p\|_{L_t^\infty L_{xyz}^2} + \|w\|_{L_t^\infty L_{xyz}^2}.
\end{equation}
Returning to \eqref{E:G1bd}, writing $D_{xyz}^{1/2} w = \chi_3
D_{xyz}^{1/2} w + (1-\chi_3) D_{xyz}^{1/2} w$, we have
$$
\|D_{xyz}^{1/2} w \|_{L_t^4 L_{xyz}^4} \lesssim \|p\|_{L_t^4
L_{xyz}^4} + \| (1-\chi_3) D_{xyz}^{1/2} \chi_2 u \|_{L_t^\infty
L_{xyz}^4}.
$$
Applying Lemma \ref{L:disj-smooth} to the second of these terms, we obtain
\begin{equation}
\label{E:G1bd2}
\|D_{xyz}^{1/2} w \|_{L_t^4 L_{xyz}^4} \lesssim \|p\|_{L_t^4L_{xyz}^4} + \|w\|_{L_t^\infty L_{xyz}^2}.
\end{equation}
Combining \eqref{E:G1bd}, \eqref{E:G1bd1}, \eqref{E:G1bd2},
\eqref{E:G11}, \eqref{E:G12}, and \eqref{E:qbds}, we obtain
\begin{equation}
\label{E:G1bd3}
\| G_1\|_{L_t^{4/3}L_{xyz}^{4/3}} \lesssim
(\alpha^*)^5(\|p\|_{L_t^4 L_{xyz}^4} + \|p\|_{L_t^\infty L_{rz}^2} +
\|w \|_{L_t^\infty L_{xyz}^2}).
\end{equation}

Since $G_2$ is shell-supported, we convert the $rz$ integration to
$xyz$ integration and appeal to \eqref{E:qbds} to obtain
\begin{equation}
\label{E:G30}
\| G_2 \|_{L_t^1L_{rz}^2} \lesssim t_1 \| q
\|_{L_t^\infty L_{xyz}^2} + t_1^{1/2}\|\nabla_{xyz}
q\|_{L_t^2L_{xyz}^2} \lesssim (\alpha^*)^{1/2}.
\end{equation}

Finally, for $G_3$ we take $\psi_4$ to be a smooth function that is
$1$ on the support of $\psi_2$ and $\psi_3$ but $0$ near $r=0$. Then
decompose
\begin{equation}
\label{E:G2-1}
\begin{aligned}
\|\la D_{rz} \ra^{-1/2}G_3\|_{L_t^2L_{rz}^2}
&\lesssim \|\la D_{rz} \ra ^{1/2} \, \psi_3 \, D_{xyz}^{1/2} \, \psi_2 \, u \|_{L_t^2L_{rz}^2} \\
&\lesssim
\begin{aligned}[t]
&\| \psi_4 D_{rz}^{1/2} \psi_3 D_{xyz}^{1/2} \psi_2 q \|_{L_t^2L_{rz}^2}
+ \| (1-\psi_4) D_{rz}^{1/2} \psi_3 D_{xyz}^{1/2} \psi_2 q \|_{L_t^2L_{rz}^2} \\
&+ \| \psi_3 D_{xyz}^{1/2} \psi_2 q \|_{L_t^2L_{rz}^2}.
\end{aligned}
\end{aligned}
\end{equation}
By Lemma \ref{L:disj-smooth}
\begin{equation}
\label{E:G2-2}
\begin{aligned}
\| (1-\psi_4) D_{rz}^{1/2} \psi_3 D_{xyz}^{1/2} \psi_2 q
\|_{L_t^2L_{rz}^2}
& \lesssim \|\psi_3 D_{xyz}^{1/2} \psi_2 q \|_{L_t^2L_{rz}^2} \sim \|\psi_3 D_{xyz}^{1/2} \psi_2 q \|_{L_t^2L_{xyz}^2}\\
&\lesssim t_1^{1/2}\|q\|_{L_t^\infty L_{xyz}^2} + \|\nabla q\|_{L_t^2L_{xyz}^2} \\
&\lesssim (\alpha^*)^{1/2}.
\end{aligned}
\end{equation}
Note that
$$
\psi_4 D_{rz}^{1/2} \psi_3 D_{xyz}^{1/2} \psi_2 q = (\psi_4
r^{-1/2}) (r^{1/2} D_{rz}^{1/2} D_{xyz}^{-1/2} r^{-1/2}) (r^{1/2}
D_{xyz}^{1/2} \psi_3 D_{xyz}^{1/2} \psi_2 q),
$$
and thus,
$$
\| \psi_4 D_{rz}^{1/2} \psi_3 D_{xyz}^{1/2} \psi_2
q\|_{L_t^2L_{rz}^2} \lesssim \|r^{1/2} D_{xyz}^{1/2} \psi_3
D_{xyz}^{1/2} \psi_2 q \|_{L_t^2L_{rz}^2} = \| D_{xyz}^{1/2} \psi_3
D_{xyz}^{1/2}\psi_2 q \|_{L_t^2L_{xyz}^2}.
$$
By the fractional Leibniz rule
\begin{equation}
\label{E:G2-3}
\| \psi_4 D_{rz}^{1/2} \psi_3 D_{xyz}^{1/2} \psi_2
q\|_{L_t^2L_{rz}^2} \lesssim t_1^{1/2}\|q\|_{L_t^2L_{xyz}^2} +
\|\nabla q \|_{L_t^2L_{xyz}^2} \lesssim (\alpha^*)^{1/2}.
\end{equation}
Now
\begin{equation}
\label{E:G2-4}
\|\psi_3 D_{xyz}^{1/2} \psi_2 q \|_{L_t^2L_{rz}^2}
\sim \|\psi_3 D_{xyz}^{1/2} \psi_2 q \|_{L_t^2L_{xyz}^2} \lesssim
t_1^{1/2} \|q\|_{L_t^\infty L_{xyz}^2} + \|\nabla
q\|_{L_t^2L_{xyz}^2} \lesssim (\alpha^*)^{1/2}.
\end{equation}
Collecting \eqref{E:G2-1}, \eqref{E:G2-2}, \eqref{E:G2-3}, \eqref{E:G2-4},
\begin{equation}
\label{E:G3bd}
\|\la D_{rz} \ra^{-1/2}G_3\|_{L_t^2L_{rz}^2} \lesssim
(\alpha^*)^{1/2}.
\end{equation}
Combine \eqref{E:p-bounds-2}, \eqref{E:G1bd3}, \eqref{E:G30}, and
\eqref{E:G3bd} to obtain (with conversion of some $xyz$-norms on $p$
and $w$ back to $rz$ norms)
$$
\|p\|_{L_t^qL_{rz}^p} \lesssim (\alpha^*)^{1/2}(\|p\|_{L_t^4
L_{rz}^4} + \|p\|_{L_t^\infty L_{rz}^2} + \|w \|_{L_t^\infty
L_{rz}^2}) + (\alpha^*)^{1/2} \,,
$$
from which it follows that \eqref{E:p-bounds} holds once \eqref{E:w-low-bounds} is available.

Next we prove \eqref{E:w-low-bounds}.
Note that $w$ is axially symmetric and shell supported.  Moreover, it solves an equation of the form
$$
i\partial_t w + \Delta_{rz}w = \sum_j  G_j,
$$
where the inhomogeneities can be put into the forms
\begin{align*}
&G_1 =  \chi_2 \, (i\partial_t  + \Delta_{rz}) u \\
&G_2 = \psi_2 \, u \\
&G_3 = \nabla_{rz} \psi_2 \, u \,.
\end{align*}
For each type of term $G_j$, the function $\psi_2$ is an axial
cutoff function with support in $\frac14 \leq ((r-1)^2 + z^2)^{1/2}
\leq \frac34$.  For term $G_1$, we substitute the equation
$(i\partial_t + \Delta_{rz}) u  = -r^{-1}\partial_r u - |u|^2u$ and
then reexpress the term involving $r^{-1}\partial_r u$ to be of the
type $G_2$ and $G_3$.  At this point, the inhomogeneities take the
form
\begin{align*}
&G_1 = \chi_2 \, |u|^2 u \\
&G_2 = \psi_2 \, u \\
&G_3 = \nabla_{rz} \, \psi_2 \, u \,.
\end{align*}
By Lemma \ref{L:2d-Str-smooth},
\begin{equation}
\label{E:w-low-bounds-2}
\|w\|_{L_t^qL_{xyz}^p} \lesssim
\|w_0\|_{L_{rz}^2} + \|G_1\|_{L_t^{4/3}L_{rz}^{4/3}} +
\|G_2\|_{L_t^1L_{rz}^2}+\|G_3\|_{L_t^1L_{rz}^2}.
\end{equation}
For $G_1$, we use that $q$ is shell-supported and apply \eqref{E:G11} to obtain
$$
\|G_1\|_{L_t^{4/3}L_{xyz}^{4/3}} \lesssim
\|q\|_{L_t^4L_{xyz}^4}^3\lesssim \|q\|_{L_{xyz}^2}^{3/2} \|\nabla
q\|_{L_t^2L_{xyz}^2}^{3/2} \lesssim (\alpha^*)^5.
$$
For the remaining inhomogeneities, we argue as follows:
$$
\| G_2\|_{L_t^1L_{rz}^2} +\| G_3\|_{L_t^1L_{rz}^2} \lesssim t_1
\|q\|_{L_t^\infty L_{xyz}^2} + t_1^{1/2} \|\nabla q\|_{L_t^2
L_{xyz}^2} \lesssim (\alpha^*)^{1/2}.
$$

Collecting, we obtain \eqref{E:w-low-bounds}.

\subsection{Step C.  Control on $v$}

We use the control on $u$ in the tight singular periphery to obtain
control on $u$ in the external region.  In this region, $u$ solves a
genuinely 3d NLS equation, and we thus need to work with the 3d
Strichartz estimates, which are, locally-in-time, weaker than the 2d
Strichartz estimates.  Also, we must only use the 3d
Gagliardo-Nirenberg estimates.

Specifically, we use \eqref{E:qbds} and \eqref{E:w-goal-bounds} to prove
\begin{equation}
\label{E:v-goal}
\|D_{xyz}^{1/2} v\|_{L_t^q L_{xyz}^p} \lesssim \|D_{xyz}^{1/2}v_0\|_{L_{xyz}^2} + (\alpha^*)^{1/2}.
\end{equation}
Via IDA 8, we obtain \eqref{E:G15}.

The equation for $v$ is
$$
i\partial_t v + \Delta_{xyz} v + |v|^2v = \sum F_j,
$$
where
\begin{align*}
& F_1 = \psi |u|^2u \\
& F_2 = \psi \, u \\
& F_3 =   \psi \, \nabla \, u ,
\end{align*}
where in each term $F_j$, the function $\psi$ is supported in the
\emph{intersection} of the supports of $\chi_1$ and $1-\chi_1$.
Recall that $\chi_2$ and $\chi_3$ are $1$ on this set.  By Lemma
\ref{L:3dStr},
\begin{align*}
\| D_{xyz}^{1/2} v \|_{L_t^q L_{xyz}^p} &\lesssim
\begin{aligned}[t]
&\|D_{xyz}^{1/2} v_0 \|_{L_{xyz}^2} + \|D_{xyz}^{1/2} \;|v|^2v\|_{L_t^{10/7} L_{xyz}^{10/7}} + \|D_{xyz}^{1/2} F_1 \|_{L_t^{8/5} L_{xyz}^{4/3}} \\
&+ \|D_{xyz}^{1/2} F_2 \|_{L_t^1 L_{xyz}^2} + \|F_3 \|_{L_t^2 L_{xyz}^2}.
\end{aligned}
\end{align*}
Note that for $F_3$, we have used the local smoothing property, and
the fact that $F_3$ is compactly supported.

For $F_1$, we use that $F_1 = \psi w|w|^2$, the fractional Leibniz
rule, and estimate as
$$
\| D_{xyz}^{1/2} F_1\|_{L_t^{8/5}L_{xyz}^{4/3}} \lesssim \|
D_{xyz}^{1/2} w \|_{L_t^4L_{xyz}^4} \|w \|_{L_t^\infty L_{xyz}^4}^2
\lesssim (\alpha^*)^{3/2}
$$
by \eqref{E:G1bd1} and \eqref{E:w-goal-bounds}. For $F_2$, we estimate as
$$
\|D_{xyz}^{1/2} F_2 \|_{L_t^1 L_{xyz}^2} \lesssim t_1 \|\la
D_{xyz}\ra^{1/2} w \|_{L_t^\infty L_{xyz}^2} \lesssim
(\alpha^*)^{3/2}
$$
by \eqref{E:w-goal-bounds}.  For $F_3$, we estimate as
\begin{align*}
\|F_3 \|_{L_t^2 L_{xyz}^2} &\lesssim \|q\|_{L_t^2 L_{xyz}^2} + \|\nabla q\|_{L_t^2 L_{xyz}^2} \\
&\lesssim t_1^{1/2} \|q\|_{L_t^\infty L_{xyz}^2} +  \|\nabla q\|_{L_t^2 L_{xyz}^2}\\
&\lesssim (\alpha^*)^{1/2}.
\end{align*}
We also estimate, using the fractional Leibniz rule, and the 3d Sobolev embedding,
\begin{align*}
\| D_{xyz}^{1/2}(|v|^2v)\|_{L_t^{10/7}L_{xyz}^{10/7}}
&\leq \| D_{xyz}^{1/2} v \|_{L_t^{10/3}L_{xyz}^{10/3}} \|v \|_{L_t^5 L_{xyz}^5}^2 \\
&\lesssim \|D_{xyz}^{1/2} v \|_{L_t^{10/3} L_{xyz}^{10/3}} \|D_{xyz}^{1/2} v\|_{L_t^5 L_{xyz}^{30/11}}^2.
\end{align*}
Collecting, we have
\begin{align*}
\| D_{xyz}^{1/2}v \|_{L_t^q L_{xyz}^p} &\lesssim \|D_{xyz}^{1/2}
v_0\|_{L_{xyz}^2} + \|D_{xyz}^{1/2}
v\|_{L_t^{10/3}L_{xyz}^{10/3}}\|D_{xyz}^{1/2} v\|_{L_t^5
L_{xyz}^{30/11}}^2 + (\alpha^*)^{1/2},
\end{align*}
which yields \eqref{E:v-goal}.

This completes the proof of BSO 1--8.

%%%%%%%%%%%%%%%%%%%%%%%%%%%%%%%%%%%%%%%%%%%%%%%%%%%%%%%%%%%%%%%%%%%%%%%%%%%%%%%%%%%%%%%%%%%%%%%%%%%%%%%

\section{Finite time blow-up and log-log speed}
\label{S:loglog}

In this section, we deduce the finite-time blow-up and log-log speed. Specifically we deduce
\eqref{E:loglog-4}, \eqref{E:loglog-7}, and \eqref{E:loglog-6} below. The convergence
\eqref{E:lambda-rate} claimed in Theorem \ref{T:main} is a refined version of \eqref{E:loglog-4}
that can be obtained by following the techniques of \cite[Prop. 6]{MR-JAMS}.

From the above bootstrap argument, we conclude that $t_1=T$, and hence, by \eqref{E:t1bd},
$T\leq \alpha^* <\infty$. From the local theory in $H^1$, we have that $\|u(t)\|_{H^1}\nearrow +\infty$
as $t\nearrow T$. By BSO 3 and the fact that $\|\nabla u(t)\|_{L^2} \leq \lambda^{-1}(\|Q_b\|_{L^2}
+\mathcal{E}(t))$, we conclude that
\begin{equation}
\label{E:lambda-blowup}
\lambda(t) \to 0 \quad \text{as} \quad t\to T.
\end{equation}
Also, since
$|\lambda_s / \lambda| \leq 1$ on $[s_0,s_1)$ (see BSO 2 and \eqref{E:IE 3.1},
recall $s_1$ corresponds to $t_1=T$), integrating from $s_0$ to $s$, we obtain
$$
|\log \lambda(s)| \leq s-s_0 + |\log \lambda(s_0)| \,.
$$
By \eqref{E:lambda-blowup} $\lambda(s) \to 0$ as $s\to s_1$, thus, the above line implies $s_1=+\infty$.

We now claim that for $t$ close enough to $T$, $\lambda$ satisfies the differential inequality
\begin{equation}
\label{E:loglog-1}
\frac{1}{C} \leq -(\lambda^2 \log |\log \lambda|)_t \leq C
\end{equation}
for some universal constant $C>0$.  To establish \eqref{E:loglog-1}, we first show that
\begin{equation}
\label{E:loglog-2}
\frac{1}{C}\leq b\log|\log\lambda|\leq C.
\end{equation}
To prove \eqref{E:loglog-2}, observe that the lower bound is equivalent to BSO 6.
For the upper bound, first note that \eqref{E:IE 5} and \eqref{E:IE 12} give
\begin{equation}
\label{E:loglog-5}
b(s) \sim \frac{1}{\log s} \,.
\end{equation}
Moreover, \eqref{E:IE 3.1} gives
\begin{equation}
\label{E:loglog-3}
-\frac{\lambda_s}{\lambda} \sim b .
\end{equation}
Integrating \eqref{E:loglog-3} and using \eqref{E:loglog-5}, we obtain
\begin{align*}
-\log \lambda(s) + \log \lambda(s_0)
&\sim \int_{s_0}^{s} b(\sigma) \, d\sigma \\
&\sim \int_{s_0}^{s} \frac{d\sigma}{\log \sigma} \\
& \geq \int_{s_0}^{s} \frac{\log \sigma - 1}{(\log \sigma)^2} \, d\sigma \\
& = \frac{s}{\log s} - \frac{s_0}{\log s_0}.
\end{align*}
Recall that $s_0=e^{3\pi/4b_0}\gg 1$ and $0<\lambda \ll 1$ by BSI 6. These imply
$\log\lambda(s_0) <0$, $\log\lambda(s)<0$, $s_0/\log s_0 >0$.
Consequently,
$$
|\log \lambda(s)| \leq \frac{s}{\log s} + |\log \lambda(s_0)| \,.
$$
Taking the log of both sides, and taking $s \geq s_0 \gg 1$ such that
$s/\log s \gg |\log \lambda(s_0)|$, gives
$$
\log |\log \lambda(s)| \leq (\log s - \log\log s) \sim \log s \sim \frac{1}{b(s)} \, ,
$$
which gives the upper bound in \eqref{E:loglog-2}.  Now we proceed to justify \eqref{E:loglog-1} using \eqref{E:loglog-2}.  Since $\lambda \ll 1$, we have
\begin{align*}
\frac{d}{dt} (\lambda^2\log |\log \lambda|)
&=
\lambda\lambda_t (\log|\log \lambda|)\left( 2
+ \frac{\lambda}{|\log \lambda| \,( \log |\log \lambda|)^2}\right) \\
&\sim \lambda \lambda_t  (\log|\log \lambda|) \\
&= \frac{\lambda_s}{\lambda}  (\log|\log \lambda|).
\end{align*}
By \eqref{E:loglog-2} and \eqref{E:loglog-3}, we obtain \eqref{E:loglog-1}.
Integrating \eqref{E:loglog-1} from $t$ up to the blow-up time $T$ (where $\lambda(T)=0$), we obtain
$$
\lambda^2(t)\log|\log\lambda(t)| \sim (T-t).
$$
This implies
\begin{equation}
\label{E:loglog-4}
\lambda(t) \sim \left( \frac{T-t}{\log|\log(T-t)|} \right)^{1/2}.
\end{equation}
From the relationship between $\lambda(t)$ and $t$ given by \eqref{E:loglog-4},
the relationship between $b(s)$ and $s$ given by \eqref{E:loglog-5},
and the relationship between $b$ and $\lambda$ given by \eqref{E:loglog-2},
we obtain the relationship between $s$ and $t$ as
\begin{equation}
\label{E:loglog-7}
|\log (T-t)|^{\alpha_1} \lesssim s \lesssim |\log (T-t)|^{\alpha_2} \quad
\text{ for some} ~~\alpha_1<\alpha_2
\end{equation}
and the relationship between $b$ and $t$ as
\begin{equation}
\label{E:loglog-6}
b \sim \frac{1}{\log|\log (T-t)|} \,.
\end{equation}
The proof of the exact convergence in \eqref{E:loglog-4} and \eqref{E:loglog-6} follows
%requires some further technical refinements -- see
\cite[Prop. 6]{MR-JAMS}. This proves \eqref{E:lambda-rate}, \eqref{E:b-rate} and \eqref{E:gamma-rate}.

\section{Convergence of the location of the singular circle}
\label{S:position}

In this section, we just note that parameters $r(t)$ and $z(t)$ converge.  Indeed, by \eqref{E:IE 3.2},
we have $|r_s|/\lambda \leq 1$ and $|z_s|/\lambda \leq 1$. Hence, for $0<t_1<t_2<T$, we have
$$
|r(t_2)-r(t_1)| \leq \int_{t_1}^{t_2} |r_t| \, dt = \int_{t_1}^{t_2} \frac{1}{\lambda}
\left| \frac{r_s}{\lambda}\right| \, dt \leq \int_{t_1}^{t_2} \frac{dt}{\lambda}.
$$
By the estimates obtained on $\lambda(t)$ above, this implies that $r(t)$ is Cauchy as $t\to T$.
Similarly, we obtain that $z(t)$ converges as $t\to T$.  BSO 1 (which we proved in Bootstrap Step 10)
then implies that $|r(T)-1|\leq (\alpha^*)^{2/3}$. This proves \eqref{E:r-conv} and \eqref{E:z-conv}.

\section{$L^2$ convergence of the remainder}
\label{S:rem-conv}

In this section, we will establish the existence of $u^*\in L^2(\mathbb{R}^3)$ such that
$\tilde u \to u^*$ in $L^2(\mathbb{R}^3)$. This will prove \eqref{E:L2profile} in Theorem \ref{T:main}.

First, we observe that there exists $u^*\in L^2_{\text{loc}}(\mathbb{R}^3\backslash \{0\})$
such that for each $R>0$,
\begin{equation}
\label{E:outconv}
\lim_{t\nearrow T} \int_{|r(t)-r(T)|^2+|z(t)-z(T)|^2\geq R^2}|\tilde u(t) - u^*|^2 \, dxdydz  =0 \,.
\end{equation}
This is essentially a consequence of the fact that, as in Bootstrap Step 16, we can carry out
the local well-posedness theory in $H^{1/2}$ in this external region.
This local theory gives a continuity of the flow (in this external region) right up to time $T$.

%Fix $R>0$, and let
%$$E_R = \{ \, (x,y,z) \, | \, |r-r(T)|^2 + |z-z(T)|^2 \geq R^2 \,\} \,.$$
%We claim that $\tilde u(t)$ is Cauchy in $L^2(E_R)$ as $t\nearrow T$.
%To prove this, fix $\epsilon>0$.  We must show that there exists
%$T_2=T_2(\epsilon, R)$ such that for any pair of times $t_1$, $t_2$
%with $T_2 \leq t_1,t_2 <T$, we have
%\begin{equation}
%\label{E:Cau}
%\|\tilde u(t_2)-\tilde u(t_1)\|_{L^2(E_R)} \leq C\epsilon
%\end{equation}
%for some absolute constant $C$.

%By the analysis in Bootstrap Step 16, we know that
%$$\int_0^T \int_{E_R} |\nabla u|^2 \, dxdydz \,dt \leq C(R) \,.$$
%Take $T_1=T_1(\epsilon,R)<T$ close enough to $T$ so that
%$$\int_{T_1}^T \int_{E_R} |\nabla u|^2 \, dxdydz \,dt \leq \epsilon \,.$$
%For any $t$ and $\tau$ such that $T_1\leq t<T$, $\tau>0$, and $t+\tau <T$,
%consider
%$$v(t,\tau) = u(t+\tau)-u(t) \,.$$
%We claim that for any such $t$, $\tau$, there holds
%\begin{equation}
%\label{E:closecontrol}
%\|v(t,\tau)\|_{L^2(E_R)}^2 \leq \|v(T_1,\tau)\|_{L^2}^2 +C\epsilon
%\end{equation}
%where $C$ is an absolute constant.  Let us assume \eqref{E:closecontrol}
%for now.
%Since $u(t)$ is continuous in $L^2(\mathbb{R}^3)$ at time $T_1$,
%there exists $\tau^*$ such that, if $0<\tau\leq \tau^*$, we have
%$$\|v(T_1,\tau)\|_{L^2} = \|\tilde u(T_1+\tau)-\tilde u(T_1)\|_{L^2}
%\leq \epsilon \,.$$
%Consequently, take $T_2=T-\tau^*$, and then \eqref{E:Cau} follows from
%\eqref{E:closecontrol}.  We now return to prove \eqref{E:closecontrol}.
%We have that
%$$i\partial_t v + \Delta v = -u|u|^2(t+\tau)+u|u|^2(t) \,.$$
%Let $\phi_R$ \ldots

Having established that \eqref{E:outconv} holds, we note that it follows from
\eqref{E:outconv} that for each $R>0$,
$$
\int_{|r-r(T)|^2+|z-z(T)|^2\geq R^2} |u^*|^2 dxdydz \leq 2\|Q\|_{L_{rz}^2}^2\, ,
$$
and thus, in fact $u^*\in L^2(\mathbb{R}^3)$. We next upgrade \eqref{E:outconv} and show that in fact
\begin{equation}
\label{E:wholeconv}
\tilde u(t)\to u^* \text{ in }L^2(\mathbb{R}^3)
\end{equation}
(not just in $L^2_{\text{loc}}(\mathbb{R}^3\backslash \{0\})$).
The following estimate is, in the 3d case, a straightforward consequence of the Hardy inequality.
But in 2d, it is a little more subtle and requires a logarithmic correction:
\begin{lemma}
\label{L:Hardy}
Suppose that $\phi(x,y)$ is a function on $\mathbb{R}^2$.
Then for any $D\gg 1$, we have (with $r=\sqrt{x^2+y^2}$),
$$
\| \phi\|_{L^2_{xy}(r\leq D)}^2 \lesssim D^2 ( \|\phi\|_{L^2_{xy}(r\leq 1)}^2
+ (\log D)\|\nabla \phi \|_{L^2_{xy}(r\leq D)}^2) \,.
$$
\end{lemma}
\begin{proof}
We follow %reproduce the proof given in
\cite[Apx. C]{MR-JAMS} for the reader's convenience. Represent $\phi$ in polar variables as $\phi(r,\theta)$.
Let $\hat \phi(r,n)$ denote the Fourier series representation in the $\theta$-variable, i.e.,
\begin{equation}
\label{E:phi-hat}
\hat \phi(r,n) = \int e^{-in\theta} \phi(r,\theta) \, d\theta.
\end{equation}
The first step is to obtain \eqref{E:nzfreqbd}, the bound for nonzero frequencies. By the Plancherel theorem,
\begin{align*}
\indentalign \int_{\theta=0}^{2\pi} |\phi(r,\theta)-\hat \phi(r,0)|^2 \, d\theta \\
&= \sum_{\substack{n\in \mathbb{Z} \\ n \neq 0}} |\hat \phi(r,n)|^2 \\
&= \sum_{\substack{n\in \mathbb{Z} \\ n \neq 0}} n^{-2}|\widehat{\partial_\theta \phi}(r,n)|^2 \\
&\leq \sum_{n\in \mathbb{Z}} |\widehat{\partial_\theta \phi}(r,n)|^2 \\
&\leq \int_{\theta=0}^{2\pi} |\partial_\theta \phi(r,\theta)|^2 \, d\theta\, .
\end{align*}
Integrating in $rdr$ from $r=0$ to $r=D$, we obtain the desired bound for nonzero angular frequencies:
\begin{equation}
\label{E:nzfreqbd}
\|\phi(r,\theta) - \hat \phi(r,0)\|_{L^2_{xy}(r\leq D)} \lesssim \|\nabla \phi\|_{L_{xy}^2(r\leq D)} \,.
\end{equation}
The next step is to obtain \eqref{E:zerofreqbd}, the bound for zero angular frequencies.
By the mean-value theorem for integrals, there exists $r_0\in [\frac12, 1]$ such that
$$
\hat \phi(r_0,0) = 2\int_{r=1/2}^1 \hat \phi(r,0) dr\,.
$$
By Cauchy-Schwarz
$$
|\hat \phi(r_0,0)| \leq \frac{2}{\sqrt 2} \left(\int_{1/2}^1 |\hat \phi(r,0)|^2 dr \right)^{1/2}.
$$
By the definition of $\hat \phi(r_0,0)$ from \eqref{E:phi-hat} with $n=0$ and Cauchy-Schwarz,
$$
|\hat \phi(r,0)|  \lesssim \left( \int |\phi(r,\theta)|^2 d\theta \right)^{1/2}.
$$
Combining we obtain
\begin{equation}
\label{E:zero1}
|\hat \phi(r_0,0)| \lesssim \left( \int_{r=1/2}^1\int_{\theta= 0}^{2\pi} |\phi(r,\theta)|^2 rdr
\,d\theta \right)^{1/2} \lesssim \|\phi\|_{L^2_{xy}(r\leq 1)}  \,.
\end{equation}
By the fundamental theorem of calculus,
$$
\hat \phi(r,0) - \hat \phi(r_0,0) = \int_{s=r_0}^{s=r} \partial_s \hat \phi(s,0) \, ds.
$$
Inserting the definition of $\hat \phi(s,0)$,
$$
\hat \phi(r,0)-\hat \phi(r_0,0)=\int_{s=r_0}^{s=r}\int_{\theta=0}^{2\pi} \partial_s \phi(s,\theta) d\theta ds .
$$
By Cauchy-Schwarz,
\begin{equation}
\label{E:zero2}
|\hat \phi(r,0) - \hat \phi(r_0,0)| \lesssim \left(\int_{s=r_0}^{s=D}\int_{\theta=0}^{2\pi} |\partial_s \phi(s,\theta)|^2 sd\theta ds \right)^{1/2}\left( \int_{s=r_0}^{s=D}\int_{\theta=0}^{2\pi}s^{-1}ds d\theta \right)^{1/2}
\end{equation}
$$
\lesssim  (\log D)^{1/2} \|\nabla \phi \|_{L_{xy}^2(r\leq D)}.
$$
Combining \eqref{E:zero1} and \eqref{E:zero2} gives
$$
|\hat \phi(r,0)| \lesssim \|\phi\|_{L_{xy}^2(r\leq 1)}+(\log D)^{1/2}\|\nabla \phi\|_{L_{xy}^2(r\leq D)}\,.
$$
Integrating against $rdrd\theta$,
\begin{equation}
\label{E:zerofreqbd}
\| \hat \phi(r,0)\|_{L^2_{xy}(r\leq D)}^2 \lesssim D^2 \left(\|\phi\|_{L_{xy}^2(r\leq 1)}^2 + (\log D) \|\nabla \phi\|_{L_{xy}^2(r\leq D)}^2\right) \,.
\end{equation}
Combining \eqref{E:nzfreqbd} and \eqref{E:zerofreqbd}, we obtain the claimed bound.
\end{proof}

Now we return to proving \eqref{E:wholeconv}. It follows easily from \eqref{E:outconv} that
$\tilde u(t) \to u^*$ weakly in $L^2(\mathbb{R}^3)$. Thus, to establish \eqref{E:wholeconv}, it suffices to establish
\begin{equation}
\label{E:wholeconv'}
\int |u^*|^2 = \lim_{t \nearrow T} \int |\tilde u(t)|^2 .
\end{equation}
Let $R(t) = e^{a/b(t)}\lambda(t)$ for some constant $0<a\ll 1$, so that $R(t)\searrow 0$ as $t\nearrow T$
just a little slower that $\lambda(t)$.  We first claim that
\begin{equation}
\label{E:wholeconv1}
\int_{|r-r(t)|^2+|z-z(t)|^2 \leq R(t)^2} |\tilde u(t)|^2 dxdydz \to 0 .
\end{equation}
We now prove \eqref{E:wholeconv1}.  By rescaling and applying Lemma \ref{L:Hardy},
\begin{align*}
\indentalign
\int_{|r-r(t)|^2+|z-z(t)|^2 \leq R(t)^2} |\tilde u(t)|^2 dxdydz \\
&= \int_{\tilde r^2+\tilde z^2 \lesssim e^{2a/b(t)}} |\epsilon(\tilde r,\tilde z)|^2 \mu(\tilde r) d\tilde rd\tilde z \\
&\lesssim \int_{\tilde r^2+\tilde z^2 \leq e^{2a/b(t)}}|\epsilon(\tilde r,\tilde z)|^2 d\tilde rd\tilde z \\
&\lesssim e^{2a/b(t)}\left( \int_{\tilde r^2+\tilde z^2 \leq 1}|\epsilon(\tilde r,\tilde z)|^2 d\tilde rd\tilde z + \frac{1}{b}\int_{\tilde r^2+\tilde z^2 \leq e^{2a/b(t)}} |\nabla \epsilon(\tilde r,\tilde z)|^2 d\tilde rd\tilde z \right) \\
&\lesssim \frac{e^{2a/b(t)}}{b} \mathcal{E}(t) \to 0.
\end{align*}
Let $\phi(r,z) =1$ when $r^2+z^2\geq 1$ but $\phi(r,z)=0$ for $r^2+z^2 \leq \frac12$.
It remains to prove that
$$
\lim_{t\nearrow T} \int \phi\left( \frac{r-r(t)}{R(t)}, \frac{z-z(t)}{R(t)}\right)
|\tilde u(t)|^2 dxdydz = \int |u^*|^2 dxdydz .
$$
To prove this, it suffices to show that
\begin{equation}
\label{E:wholeconv3}
\left| \int \phi\left( \frac{r-r(t)}{R(t)}, \frac{z-z(t)}{R(t)}\right) |u(t)|^2 - \int \phi\left(
\frac{r-r(T)}{R(t)}, \frac{z-z(T)}{R(t)}\right) |u^*|^2 \right| \underset{t\nearrow T}{\to} 0 .
\end{equation}
(Note that here we have replaced $\tilde u$ by $u$; this is acceptable, since the contribution from
$Q_b$ is negligible in this region of space.)
To prove this, let, for $t$ fixed,
$$
v(\tau) = \int \phi\left( \frac{r-r(\tau)}{R(t)}, \frac{z-z(\tau)}{R(t)}\right) |u(\tau)|^2 dxdydz .
$$
Then observe
$$
\partial_\tau v = -\frac{(r_\tau,z_\tau)}{R(t)} \cdot \int \nabla \phi(\cdot, \cdot) |u|^2 \, dxdydz
+ \frac{2}{R(t)}\Im \int \nabla \phi(\cdot,\cdot) \cdot \nabla u \; \bar u \, dxdydz .
$$
By \eqref{E:IE 3.2}, $|(r_\tau,z_\tau)| \leq \lambda^{-1}$, and by BSO 3,
$\|\nabla u\|_{L^2} \sim \lambda^{-1}(1+O(\Gamma^{4/5})) \sim \lambda^{-1}$.  Hence,
$$
|\partial_\tau v(\tau)| \lesssim \frac{1}{e^{a/b(t)}\lambda(t)} \cdot \frac{1}{\lambda(\tau)} \,.
$$
Note that by \eqref{E:loglog-4},
\begin{align*}
\indentalign \int_t^T \frac{d\tau}{\lambda(\tau)} \lesssim \int_t^T \left( \frac{\log|\log(T-\tau)|}{T-\tau} \right)^{1/2} \, d\tau \\
&\sim \int_t^T (T-\tau)^{-1/2}(\log|\log(T-\tau)|)^{1/2}\left( 1 + \frac{1}{\log|\log(T-\tau)| \; \log(T-\tau)} \right) \, d\tau \\
&= 2(T-t)^{1/2}(\log|\log(T-t)|)^{1/2} .
\end{align*}
The rate of $b(t)$ in \eqref{E:loglog-6} implies that $e^{a/b(t)} \sim |\log(T-t)|^C$ for some $C>0$.
Hence,
$$
|v(T)-v(t)| \lesssim \frac{1}{e^{a/b(t)}}\cdot \frac{1}{\lambda(t)} \int_t^T \frac{d\tau}{\lambda(\tau)}
\lesssim \frac{\log|\log(T-t)|}{|\log (T-t)|^C}.
$$
But by \eqref{E:outconv},
$$
v(T) = \lim_{\tau\nearrow T} v(\tau) = \int \phi\left( \frac{r-r(T)}{R(t)}, \frac{z-z(T)}{R(t)}\right)
|u^*|^2
$$
yielding \eqref{E:wholeconv3}, and hence, \eqref{E:wholeconv} is established.
%\section{Structure of the $L^2$ remainder profile}
%\label{S:rem-desc}
%The structure of the $L^2$ remainder profile follows \cite{MR-CMP}.
This finishes the proof of Theorem \ref{T:main}.


\begin{thebibliography}{00}

\bibitem{B56}
Bracewell, R. N. \emph{Strip integration in radio astronomy},
Austral. J. Phys.  9  (1956), 198--217

\bibitem{B86}
Bracewell, R. N. \emph{The Fourier transform and its applications.} Third
edition. McGraw-Hill Series in Electrical Engineering. Circuits and Systems.
McGraw-Hill Book Co., New York, 1986. xx+474 pp. ISBN: 0-07-007015-6

\bibitem{CK} M. Christ, A. Kiselev,  \emph{Maximal Functions Associated to Filtrations}, Jour. Func. Anal. 179 (2001), pp. 409--425.

\bibitem{DHR} T. Duyckaerts, J. Holmer, and S. Roudenko,
\emph{Scattering for the non-radial 3D cubic nonlinear Schr\"odinger equation},  Math. Res. Lett.  15  (2008),  no. 6, pp. 1233--1250.

\bibitem{DM} T. Duyckaerts and F. Merle ,
\emph{Dynamic of threshold solutions for energy critical NLS}, Geom. Func. Anal.
 18  (2009),  no. 6, 1787--1840.

\bibitem{DR} T. Duyckaerts and S. Roudenko,
\emph{Threshold solutions for the focusing 3d cubic Schr\"odinger equation}, Revista Mat. Iber., to appear, \texttt{arxiv:0806.1752}

\bibitem{BFN} G. Baruch, G. Fibich, N. Gavish,
\emph{Singular standing-ring solutions of nonlinear partial differential equations}, arxiv.org preprint \texttt{arXiv:0907.2016}.

\bibitem{EZ} L.C. Evans and M. Zworski, \emph{Lectures on semiclassical analysis}, \texttt{http://math.berkeley.edu/$\sim$zworski/semiclassical.pdf}.

\bibitem{FMR} G. Fibich, F. Merle, and P. Rapha\"el,
\emph{Proof of a spectral property related to the singularity formation for the $L\sp 2$ critical nonlinear Schr\"odinger equation},  Phys. D  220  (2006),  no. 1, pp. 1--13.

\bibitem{GV} J.\ Ginibre and G.\ Velo,
\emph{On a class of nonlinear Schr\"odinger equation. I. The
Cauchy problems; II. Scattering theory, general case}, J. Func.
Anal. 32 (1979), 1-32, pp.\ 33-71.

\bibitem{GM}
L.\ Glangetas and F. Merle,
\emph{ A geometrical approach of existence of blow up solutions in $H^1$ for nonlinear Schr\"odinger equation}, Rep. No. R95031, Laboratoire d'Analyse Num\'erique, Univ.
Pierre and Marie Curie.

\bibitem{G77} R. Glassey,
\emph{On the blowing up of solutions to the Cauchy problem for
nonlinear Schr\"odinger equation}, J. Math. Phys., 18 (1977) no. 9, pp.\
1794--1797.

\bibitem{HPR}
J. Holmer, R. Platte and S. Roudenko,
\emph{Blow-up criteria for the 3d cubic nonlinear Schr\"odinger equation},
Nonlinearity, to appear, \texttt{arxiv.org:0911.3955}

\bibitem{HR1} J.\ Holmer and S.\ Roudenko,
\emph{On blow-up solutions to the 3D cubic nonlinear Schr\"odinger equation}, AMRX Appl. Math. Res. Express, v. 1 (2007), article ID abm004, 31 pp.

\bibitem{HR2} J. Holmer and S. Roudenko,
\emph{A sharp condition for scattering of
the radial 3d cubic nonlinear Schr\"odinger equation},  Comm. Math.
Phys.  282  (2008),  no. 2, pp. 435--467.

\bibitem{HR3} J. Holmer and S. Roudenko,
\emph{Norm divergence for infinite-variance nonradial solutions to 3d NLS}, Comm. PDE, to appear, \texttt{arxiv.org:0906.0203}

\bibitem{Keel-Tao} M. Keel and T. Tao,
\emph{Endpoint Strichartz estimates},  Amer. J. Math.  120  (1998), pp. 955--980.

\bibitem{KM}  C. E. Kenig and F. Merle,
\emph{Global well-posedness, scattering and blow-up for the
energy-critical, focusing, non-linear Schr\"odinger equation in the
radial case}, Invent. Math. 166 (2006), no. 3, pp. 645--675.

\bibitem{KPV} C. E. Kenig, G. Ponce, L.  Vega,
\emph{Oscillatory integrals and regularity of dispersive equations},
Indiana Univ. Math. J.  40  (1991) pp. 33--69.

\bibitem{KRRT} E.A. Kuznetsov, J. Juul Rasmussen, K. Rypdal, S.K. Turitsyn,
\emph{Sharper criteria for the wave collapse}, Physica D, Vol. 87,
Issues 1-4 (1995), pp. 273--284.

\bibitem{Lu95} P.M. Lushnikov,
\emph{Dynamic criterion for collapse}, Pis'ma Zh.
\'Eksp. Teor. Fiz. 62 (1995) pp. 447--452.

\bibitem{M97} Y. Martel,
\emph{Blow-up for the nonlinear Schr\"odinger equation in
nonisotropic spaces},  Nonlinear Anal.  28  (1997),  no. 12,
1903--1908.

\bibitem{M93} F. Merle,
\emph{Determination of blow-up solutions with minimal mass for nonlinear Schrödinger equations with critical power},  Duke Math. J.  69  (1993),  no. 2, 427--454.

\bibitem{MR-JAMS} F. Merle and P. Rapha\"el,
\emph{On a sharp lower bound on the blow-up rate for the $L\sp 2$ critical nonlinear Schr\"odinger equation},  J. Amer. Math. Soc.  19  (2006),  no. 1, pp. 37--90.

\bibitem{MR-Annals} F. Merle and P. Rapha\"el,
\emph{The blow-up dynamic and upper bound on the blow-up rate for critical nonlinear Schr\"odinger equation},  Ann. of Math. (2)  161  (2005),  no. 1, pp. 157--222.

\bibitem{MR-CMP} F. Merle and P. Rapha\"el,
\emph{Profiles and quantization of the blow up mass for critical nonlinear Schr\"odinger equation},  Comm. Math. Phys.  253  (2005),  no. 3, pp. 675--704.

\bibitem{MR-Invent} F. Merle and P. Rapha\"el,
\emph{On universality of blow-up profile for $L\sp 2$ critical nonlinear Schr\"odinger equation},  Invent. Math.  156  (2004),  no. 3, pp. 565--672.

\bibitem{MR-GAFA} F. Merle and P. Rapha\"el,
\emph{Sharp upper bound on the blow-up rate for the critical nonlinear Schr\"odinger equation},  Geom. Funct. Anal.  13  (2003),  no. 3, pp. 591--642.

\bibitem{MRS} F. Merle, P. Raphael, J. Szeftel,
\emph{Stable self similar blow up dynamics for slightly $L^2$ supercritical NLS equations}, arxiv.org preprint \texttt{arXiv:0907.4098}.

\bibitem{N} H. Nawa,
\emph{Asymptotic and limiting profiles of blowup solutions of the
nonlinear Schr\"odinger equation with critical power},  Comm. Pure
Appl. Math.  52  (1999),  no. 2, 193--270.


\bibitem{OT91} T. Ogawa and Y. Tsutsumi,
\emph{Blow-up of $H\sp 1$ solution for the nonlinear Schr\"odinger equation},
J. Differential Equations 92 (1991), no. 2, pp.  317--330.

\bibitem{R-stability} P. Rapha\"el,
\emph{Stability of the log-log bound for blow up solutions to the critical non linear Schr\"odinger equation},  Math. Ann.  331  (2005),  no. 3, pp. 577--609.

\bibitem{R} P. Rapha\"el,
\emph{Existence and stability of a solution blowing up on a sphere for an $L\sp 2$-supercritical nonlinear Schr\"odinger equation},  Duke Math. J.  134  (2006),  no. 2, pp. 199--258.

\bibitem{RS} P. Rapha\"el and J. Szeftel,
\emph{Standing ring blow up solutions to the N-dimensional quintic nonlinear Schr\"odinger equation}, Comm. Math. Phys.  290  (2009), no. 3, pp. 973--996.

\bibitem{Stein} E. Stein,
\emph{Harmonic analysis: real-variable methods, orthogonality, and oscillatory integrals}, Princeton University Press, 1993.

\bibitem{S} W. Strauss,
\emph{Existence of solitary waves in higher dimensions},  Comm. Math. Phys.  55  (1977), no. 2, pp. 149--162.

\bibitem{Str} R. Strichartz,
\emph{Restrictions of Fourier transforms to quadratic surfaces and decay of solutions of wave equations},  Duke Math. J.  44  (1977) pp. 705--714.

\bibitem{SS} C. Sulem and  P.-L. Sulem,
\emph{The nonlinear Schr\"odinger equation. Self-focusing and wave collapse}.
Applied Mathematical Sciences, 139. Springer-Verlag, New York, 1999. xvi+350 pp.

\bibitem{T93} S.K. Turitsyn,
\emph{Nonstable solitons and sharp criteria for wave collapse},
Phys. Rev. E (3)  47  (1993),  no. 1, R13--R16.

\bibitem{VPT}  S.N. Vlasov,   V.A. Petrishchev, and  V.I. Talanov,
\emph{Averaged  description  of wave  beams in linear and nonlinear  media
(the method  of moments)}, Radiophysics and Quantum Electronics 14 (1971) pp. 1062--1070.
Translated from Izvestiya Vysshikh Uchebnykh Zavedenii,
Radiofizika, 14 (1971) pp. 1353--1363.

\bibitem{W83} M. Weinstein,
\textit{Nonlinear Schr\"odinger equations and sharp interpolation
estimates}, Comm. Math. Phys. 87 (1982/83), no. 4, pp.\ 567--576.

\bibitem{Z72} V.E. Zakharov,
\emph{Collapse of Langmuir waves}, Soviet Physics JETP (translation of the
Journal of Experimental and Theoretical Physics of the Academy of Sciences of the USSR),
35 (1972) pp. 908--914.

\bibitem{Zw} I. Zwiers,
\emph{Standing Ring Blowup Solutions for Cubic NLS},
preprint, {\tt arxiv.org/abs/1002.1267}

\end{thebibliography}
\end{document}